\documentclass[a4paper,10pt]{article}

\usepackage[british]{babel}
\usepackage{amsmath,amssymb,wasysym,amsfonts,amsthm,amsthm,bm}
\usepackage{empheq}

\usepackage{graphicx,epsfig}
\usepackage{subfig}

\usepackage{txfonts}
\usepackage[T1]{fontenc}
\usepackage{textcomp}
\usepackage[utf8x]{inputenc}
\usepackage{ucs}
\usepackage{titletoc,titlesec}
\usepackage{color}
\usepackage{overpic}

\usepackage[a4paper,margin=1in]{geometry}
\usepackage{hyperref,url,enumerate,framed}

\newtheorem{thm}{Theorem}
\newtheorem{lem}[thm]{Lemma}

\theoremstyle{remark}
\newtheorem{rmk}[thm]{Remark}
\numberwithin{equation}{section}
\numberwithin{thm}{section}

\newcommand{\eps}{\varepsilon}

\title{Travelling pulses on three spatial scales in a Klausmeier-type vegetation-autotoxicity model}
\author{Paul Carter\thanks{Department of Mathematics, University of California, Irvine, USA}\and Arjen Doelman\thanks{Mathematical Institute, Leiden University, Leiden, Netherlands} \and Annalisa Iuorio\thanks{Department of Engineering, Parthenope University of Naples, Naples, Italy} \and Frits Veerman\footnotemark[2]}

\begin{document}

\maketitle

\begin{abstract} 
Reaction-diffusion models describing interactions between vegetation and water reveal the emergence of several types of patterns and travelling wave solutions corresponding to structures observed in real-life. Increasing their accuracy by also considering the ecological factor known as \emph{autotoxicity} has lead to more involved models supporting the existence of complex dynamic patterns. In this work, we include an additional carrying capacity for the biomass in a Klausmeier-type vegetation-water-autotoxicity model, which induces the presence of two asymptotically small parameters: $\varepsilon$, representing the usual scale separation in vegetation-water models, and $\delta$, directly linked to autotoxicity. We construct three separate types of homoclinic travelling pulse solutions based on two different scaling regimes involving $\varepsilon$ and $\delta$, with and without a so-called \emph{superslow plateau}. The relative ordering of the small parameters significantly influences the phase space geometry underlying the construction of the pulse solutions. We complement the analysis by numerical continuation of the constructed pulse solutions, and demonstrate their existence (and stability) by direct numerical simulation of the full PDE model.\\

\end{abstract}

\section{Introduction}

In the last few decades, extensive ecological and mathematical investigations have been devoted to the improvement of our understanding of vegetation dynamics because of its key role as ecological indicator: the well-being -- and the resilience -- of an ecosystem can typically be measured 
by investigating the spatial structures (or \emph{patterns}) vegetation forms in order to enhance its survival probability in a certain environment \cite{Bastiaansen_2018,Bastiaansen_2020,Deblauwe.2008,Kefi_2007,Meron_2015,Rietkerk_2021}. These structures may appear both in the form of stable (``Turing'') and dynamic (travelling) patterns \cite{Consolo_2022, Eigentler_2020, Gandhi_2018}. Transient patterns, in particular, play a crucial role in a wide range of important ecological aspects, including biodiversity \cite{Iuorio_2023}. 
The optimization process behind the emergence of vegetation patterns is mainly based on plant-soil interactions (often known as \emph{feedbacks}), particularly in terms of nutrients and attacks by external factors \cite{Dekker.2010,Dekker.2009,D_Odorico_2007,Manfreda_2013,Rietkerk.2004,Vincenot_2017}. These feedbacks are scale-dependent, as they mainly act as short-range activation and long-range inhibition \cite{Rietkerk_2008}.\\
Nutrients are mainly linked to water, as this is the principal source of growth for vegetation. Therefore, the basic modelling framework used to analytically investigate vegetation patterns is based on partial differential equations (PDEs) of reaction-diffusion type describing the interactions between biomass and water densities \cite{Gowda_2014,von_Hardenberg.2001,Klausmeier.1999}.
In some cases, water is split into soil and surface components in order to more accurately capture plant-water feedback in arid environments \cite{Gilad_2004,Gilad.2007,HilleRisLambers.2001,Rietkerk.2002}. Additional approaches shift from a parabolic to a hyperbolic framework with the aim to include realistic inertial effects in the water diffusion process \cite{Consolo_2017, Consolo_2019}. \\
From a mathematical viewpoint, patterns typically arise from a destabilization of a uniform steady-state via heterogeneous perturbations, also known as Turing instability. Linear and nonlinear stability analyses allow to determine the parameter ranges supporting this type of instability together with many other rich, complex dynamics including Hopf dances and homoclinic snaking \cite{Bastiaansen_2020,Dawes_2015,Doelman_2012,Kealy.2011,Siteur_2014}.
Moreover, the scale-dependent feedbacks between the variables naturally induce multiple scales in the system: this allows for the application of Geometric Singular Perturbation Theory (GSPT), which has proved to be a valuable method to investigate emerging far-from-equilibrium patterns in several models (of Klausmeier-Gray-Scott type) \cite{Doelman_1997,DoelmanVeerman.2015,SewaltDoelman.2017}.
\\
Among the external factors negatively affecting vegetation dynamics (including e.g.~soil-borne pathogens \cite{bever2015maintenance}, animal grazing \cite{Noy_Meir_1975}, and human activity \cite{Gowda_2018}), a prominent role is played by the so-called \emph{autotoxicity}, which is strongly related to the presence of extracellular DNA produced by plants during decomposition \cite{Bonanomi.2011,Mazzoleni.2007, Mazzoleni_etal.2014}. This mechanism has proved to have a strong impact on several relevant ecological phenomena, including species coexistence and biodiversity \cite{Bonanomi_2005,Mazzoleni.2010}. Moreover, coupling reaction-diffusion models describing biomass (and water) dynamics to an additional equation for autotoxicity (in some cases only implicitly depending on space, leading to a so-called reaction-diffusion-ODE system) reveals how this factor influences plants' spatial organisation in a wide variety of cases, including e.g.~clonal rings \cite{Bonanomi.2014,Carteni.2012}, fairy rings \cite{Salvatori_2023}, and vegetation patterns \cite{Marasco_2020,Marasco_etal.2014}. \\
In \cite{Marasco_etal.2014}, in particular, the combination of biomass, water, and autotoxicity dynamics in a reaction-diffusion-ODE model leads to novel, ecologically relevant features in the emerging patterns which are not present in classical biomass-water models of Klausmeier type. For instance, in parameter regimes with low precipitation rate and strong autotoxicity,
spot and stripe patterns -- which are stable in the absence of autotoxicity -- become dynamic and nonsymmetric.
As numerical simulations reveal that these patterns exhibit a multi-scale structure, in \cite{IuorioVeerman.2021} a detailed investigation has been performed in order to rigorously construct travelling wave solutions using GSPT. The starting point hence consisted in the following nondimensionalization of the original reaction-diffusion-ODE system presented in \cite{Marasco_etal.2014}
\begin{subequations}\label{eq:UVS_RDsystem}
\begin{align}
 \frac{\partial U}{\partial t} &= \Delta U + \mathcal{A} \left(1-U\right) - U V^2, \label{eq:UVS_RDsystem_U}\\
 \frac{\partial V}{\partial t} &= \eps^2 \Delta V + U V^2 - \mathcal{B} V - \mathcal{H} V S, \label{eq:UVS_RDsystem_V}\\
 \mathcal{D} \frac{\partial S}{\partial t} &= -S + \mathcal{B} V + \mathcal{H} V S, \label{eq:UVS_RDsystem_S}
\end{align}
\end{subequations}
with positive parameters $\mathcal{A}, \mathcal{B}, \mathcal{D}, \mathcal{H}$ and $0 < \eps \ll 1$ (see \cite{IuorioVeerman.2021} for further details).  An application of the approach from e.g.~\cite{DoelmanVeerman.2015,SewaltDoelman.2017} allowed to establish the existence of stationary and travelling pulse solutions in \eqref{eq:UVS_RDsystem} on a one-dimensional unbounded spatial domain, since it suffices in capturing the main features of the phenomena we aimed to investigate.
%
Through an extensive scaling analysis, it was possible to construct both stationary and travelling pulses; the latter ones, however, did not match those observed in numerical simulations,
as the region in parameter space where these travelling pulses are consistently observed is not contained in the asymptotic scaling of the model parameters assumed in the analysis in \cite{IuorioVeerman.2021}. In other words, unlike the traveling pulses observed in the simulations of \cite{IuorioVeerman.2021}, the ecological relevance of the traveling pulses constructed in that same paper is expected to be limited.\\
The biomass-water-autotoxicity model analysed in \cite{IuorioVeerman.2021,Marasco_etal.2014} is based on the assumption that vegetation growth in drylands is always water limited. Incorporating an explicit logistic growth term in the equation describing vegetation dynamics, however, is an important extension both from the ecological and the mathematical viewpoint: assuming a carrying capacity to describe the total concentration of vegetation that can be supported at a certain location, in fact, increases the model's accuracy while facilitating the analysis \cite{BCD.2019,CarterDoelman.2018}. Furthermore, this assumption can be ecologically interpreted as a growth limitation due to the interaction with other factors and/or species, e.g. fungi \cite{Salvatori_2023}. Finally, as we aim to construct far-from-equilibrium patterns, we explicitly incorporate situations where the model components (in particular $V$) attain relatively high values, which might reasonably be in the range of carrying capacity.
Inspired by \cite{BCD.2019}, we hence here modify the autotoxicity system \eqref{eq:UVS_RDsystem} by introducing a logistic term in the biomass equation as follows
\begin{subequations}\label{eq:UVS_RDsystem_modified}
\begin{align}
 \frac{\partial U}{\partial t} &= \Delta U + \mathcal{A} \left(1-U\right) - U V^2,\\
 \frac{\partial V}{\partial t} &= \eps^2 \Delta V + U V^2(1-k V) - \mathcal{B} V - \mathcal{H} V S,\\
 \mathcal{D} \frac{\partial S}{\partial t} &= -S + \mathcal{B} V + \mathcal{H} V S,
\end{align}
\end{subequations}
where the parameters $\mathcal{A}, \mathcal{B}, \mathcal{D}, \mathcal{H}$ and $0 < \eps \ll 1$ coincide with the ones introduced above and the additional parameter $k$ represents the inverse of the carrying capacity. We assume here that all model parameters are $\mathcal{O}(1)$, with the exception of $\mathcal{D}$ in order to extend the parameter space considered in the analysis based on numerical observations. Our goal is to construct travelling wave solutions that may be expected to have ecological relevance, in the sense that our analytically constructed travelling waves correspond -- also at a quantitative level -- to travelling pulses that can be  observed in numerical simulations of model \eqref{eq:UVS_RDsystem_modified} (and thus can likely be stable).
To this aim,  we introduce the co-moving inertial frame $z = x - \mathcal{C} t$, which brings System \eqref{eq:UVS_RDsystem_modified} into the following form
\begin{subequations}\label{eq:ODEsystem_1_mod}
\begin{align}
 U_{zz} + \mathcal{C} U_z - U V^2 + \mathcal{A}(1-U) &= 0,\\
 \eps^2 V_{zz} + \mathcal{C} V_z + U V^2(1-kV) - \mathcal{B} V - \mathcal{H} V S &= 0,\\
 \mathcal{C}\mathcal{D} S_z - S + \mathcal{B} V+\mathcal{H} V S &= 0.
\end{align}
\end{subequations}
Scaling $\mathcal{C} = \eps c$ allows us to formulate the slow system
\begin{subequations}\label{eq:ODEsystem_modified_rescaled}
 \begin{align}
  u_z &= p,\\
  p_z &= u v^2-\mathcal{A}\left(1-u\right) - \eps c p,\\
  \eps v_z &= q,\\
  \eps q_z &= \mathcal{B} v  - u v^2(1-k v) + \mathcal{H} v s - c q,\\
 \eps s_z &= \delta \left[- \mathcal{B} v- \mathcal{H} v s + s\right], 
 \end{align}
\end{subequations}
where
\begin{equation} \label{eq:delta}
    \delta := \frac{1}{c \mathcal{D}}.
\end{equation}
The associated fast system is obtained by a reformulation of \eqref{eq:ODEsystem_modified_rescaled} in terms of the fast coordinate $\zeta = z/\eps$, corresponding to
\begin{subequations}\label{eq:ODEsystem_modified_fast_rescaled}
 \begin{align}
  u_\zeta &= \eps p,\\
  p_\zeta &= \eps\left[u v^2-\mathcal{A}\left(1-u\right) - \eps c p\right],\\
  v_\zeta &= q,\\
  q_\zeta &= \mathcal{B} v  - u v^2(1-k v) + \mathcal{H} v s - c q,\\
 s_\zeta &= \delta\left[- \mathcal{B} v- \mathcal{H} v s + s\right]. 
 \end{align}
\end{subequations}
The asymptotic magnitude of $\delta$ determines the scaling hierarchy of \eqref{eq:ODEsystem_modified_rescaled}/\eqref{eq:ODEsystem_modified_fast_rescaled}. We identify five possible scaling regimes:
\begin{enumerate}
    \item[(i)] $0 < \eps \ll \delta \ll 1$
    \item[(ii)] $0 < \delta \ll \eps \ll 1$
    \item[(iii)] $0 < \eps \ll 1 \ll \delta$
    \item[(iv)] $0 < \eps \sim \delta \ll 1$
    \item[(v)] $0 < \eps \ll \delta \sim 1$
\end{enumerate}
Regime (iii) has been studied in \cite{IuorioVeerman.2021} in the absence of carrying capacity \eqref{eq:UVS_RDsystem}; the analysis for $k \neq 0$ (cf. \eqref{eq:UVS_RDsystem_modified}) would be similar. We expect that regime (v) poses an analytical challenge due to the fact that the reduced fast system would be fully three-dimensional; see also \cite[Section 4.2.3]{IuorioVeerman.2021}. For regime (iv), we would obtain a fully three-dimensional slow system; however, as $u$ and $s$ are not directly coupled, this does not a priori prohibit an analytical approach. Nevertheless, as case (iv) can be seen as a transition between cases (i) and (ii), we choose to focus on these two distinguished limits, and refer to Remark \ref{rmk:case_iv} for more comments on the transitionary case (iv). An additional motivation to focus on these two cases in particular is linked to the importance of double limits in differential equations, due to their emergence in a wide range of application areas and to the complex and interesting behaviours they induce \cite{Kuehn_2022}. \\
Direct simulations of System \eqref{eq:UVS_RDsystem_modified} in parameter regimes which correspond to cases (i) and (ii) confirm the existence of stable travelling wave solutions. These emerging patterns, which we here rigorously construct using GSPT, exhibit novel features, including the presence of a so-called \emph{superslow plateau}. In case (i), in particular, we encounter both solutions without and with the superslow plateau  (see Figure \ref{fig:1Dprof_i}), whereas in case (ii) all solutions present such plateau (see Figure \ref{fig:1Dprof_ii}).

\begin{figure}[!ht]
    \begin{minipage}{.3\textwidth}
    \centering
    \vspace{.5cm}
    \begin{overpic}[scale=0.65]{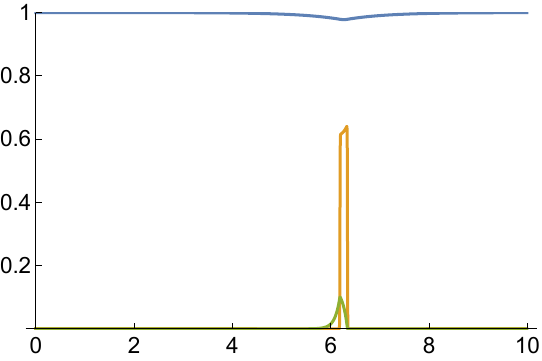}
    \put(50,70){\textbf{(a)}}
    \put(50,-5){$x$}
    \end{overpic}
    \end{minipage}
    \hspace{3cm}
    \begin{minipage}{.3\textwidth}
    \centering
    \vspace{.5cm}
    \begin{overpic}[scale=0.65]{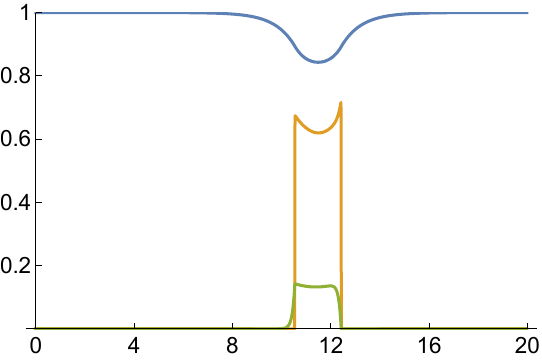}
    \put(50,70){\textbf{(b)}}
    \put(50,-5){$x$}
    \end{overpic}
    \end{minipage}
    \vspace{.5cm}
    \caption{Different patterned solutions of System \eqref{eq:UVS_RDsystem_modified} studied in this paper in case (i). Presented figures show cross-sections of the variables $U(x)$ (blue), $V(x)$ (orange), and $S(x)$ (green) obtained from direct numerical simulations on a 1D domain of length $L$ up to time $T$ (a) without and (b) with superslow plateau. The parameter values used in the simulations are $\mathcal{A} = 1.5$, $\mathcal{B} = 0.2$, $\mathcal{H} = 0.1$, $\eps = 10^{-3}$, $T=5 \cdot 10^4$, and (a) $\mathcal{D} = 3160$, $k = 1.059$, $L=10$; (b) $\mathcal{D} = 2277$, $k = 0.955$, $L = 20$.}
    \label{fig:1Dprof_i}
\end{figure}

\begin{figure}[!ht]
    \centering
    \begin{overpic}[scale=0.65]{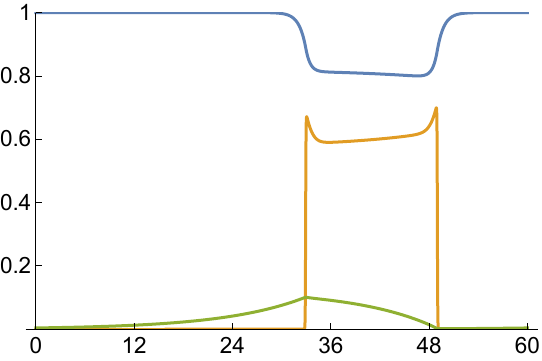}
    \put(50,-5){$x$}
    \end{overpic}
    \vspace{.5cm}
    \caption{Patterned solution of System \eqref{eq:UVS_RDsystem_modified} studied in this paper in case (ii). The presented figure shows cross-sections of the variables $U(x)$ (blue), $V(x)$ (orange), and $S(x)$ (green) obtained from direct numerical simulations on a 1D domain of length $L$ up to time $T$. The parameter values used in the simulations are $\mathcal{A} = 1.5$, $\mathcal{B} = 0.2$, $\mathcal{D} = 37492$, $\mathcal{H} = 0.1$, $\eps = 0.01$, $k = 0.955$, $L = 60$, and $T=5 \cdot 10^4$.}
    \label{fig:1Dprof_ii}
\end{figure}

The goal of our analysis is to construct a homoclinic orbit in system \eqref{eq:ODEsystem_modified_rescaled}, given that both $\eps$ and $\delta$ are asymptotically small. Although system \eqref{eq:ODEsystem_modified_rescaled} has multiple equilibria, we restrict ourselves to orbits that are homoclinic to the (trivial) `bare soil' state $(1,0,0,0,0)$. In both case (i) and (ii), the overall strategy is to construct a singular orbit by concatenating orbits in the three time scales, and subsequently prove that this singular orbit persists for sufficiently small $\eps$ and $\delta$. From a more general mathematical point of view, the presence of two independent small parameters adds an extra `deeper' layer to the geometric approach classically employed for the construction of singular homoclinic (or heteroclinic) patterns (see \cite{BCD.2019,CarterDoelman.2018,DoelmanVeerman.2015} and the references therein) that is expected to be relevant beyond the present specific setting.  The construction of the three time scale singular orbits in case (i) (both without and with superslow plateau) and case (ii) is worked out in analytical and geometrical detail in Section \ref{sec:casei} and \ref{sec:caseii}, respectively. In Section \ref{sec:num} we present some numerical computations confirming the analytical construction using both MATLAB based simulations and continuation with AUTO \cite{Doedel_1981}. Finally, our results are summarized and discussed in Section \ref{sec:concl}.

\section{Case (i): intermediate $s$} \label{sec:casei}
Case (i) results in a three timescale system in which $(u,p)$ evolve on a `superslow' timescale, $(v,q)$ evolve on a fast timescale, and $s$ evolves on an intermediate slow timescale, due to the hierarchy 
\begin{align}\label{eq:scaling_eps_delta_case(i)}
0 < \eps \ll \delta :=\frac{1}{c \mathcal{D}} \ll 1.
\end{align}
In this regime, the system~\eqref{eq:ODEsystem_modified_rescaled} therefore corresponds to the superslow system, while the fast system is that given by~\eqref{eq:ODEsystem_modified_fast_rescaled}. We introduce the intermediate slow coordinate $y=\frac{\delta}{\eps} z$, where $\frac{\eps}{\delta}\ll1$, which results in the intermediate system
\begin{subequations}\label{eq:ODEsystem_case(i)}
 \begin{align}
 \delta u_y &= \eps p,\\
 \delta p_y &= \eps\left[u v^2-\mathcal{A}\left(1-u\right) - \eps c p\right],\\
  \delta v_y &= q,\\
  \delta q_y &= \mathcal{B} v  - u v^2(1-k v) + \mathcal{H} v s - c q,\\
 s_y &= - \mathcal{B} v- \mathcal{H} v s + s. 
 \end{align}
\end{subequations}
The upcoming analysis will show that in scaling regime \eqref{eq:scaling_eps_delta_case(i)}, two distinct pulse solutions can be constructed, each arising from a different geometric configuration of several persistent manifolds in phase space that structure the flow, to wit pulses \emph{with} or \emph{without superslow plateau}.\\
Pulses \emph{without} superslow plateau (see Section \ref{ssec:case(i)_withoutplateau} and Theorem \ref{case(i)_thm_noplateau}) have $u \approx 1$ and $p \approx 0$. The pulse is constructed by directly concatenating the fast flow in $(v,q)$ and the superslow flow in $s$ \eqref{case(i)_homoclinic_concatenation}, see also Figure \ref{fig:case(i)_no_plateau}. The singular construction as outlined in Section \ref{ssec:case(i)_withoutplateau} yields one matching condition for $s$, which determines the amplitude of the $s$-peak.\\
The construction of pulses \emph{with} superslow plateau is more elaborate, see Section \ref{ssec:case(i)_withplateau} and Theorem \ref{case(i)_thm_plateau}. Here, the singular pulse structure is a concatenation of fast $(v,q)$, intermediate $(u,p)$ and superslow $s$ dynamics \eqref{case(i)_homoclinic_concatenation_slow_plateau}, see also Figure \ref{fig:case(i)_plateau}. The singular construction yields two matching conditions, for both $u$ and $s$, determining the height of the $s$-plateau and the width of the $v$-plateau.\\
Although their construction does not a priori prohibit coexistence of both pulse types, detailed inspection of existence conditions reveals that pulses with and without superslow plateau cannot coexist; see Remark \ref{rmk:case(i)_nocoexistence}.

\subsection{Superslow dynamics}\label{sec:case(i)_superslow}
The superslow dynamics are those which occur on the $z$-timescale. We let $\eps \to 0$ in~\eqref{eq:ODEsystem_modified_rescaled}, which results in three algebraic equations defining a two-dimensional critical manifold
\begin{align}\label{eq:case(1)_superslow_C}
\mathcal{C}:=\left\{q=0, \mathcal{B}v-uv^2(1-kv)+\mathcal{H}vs=0, -\mathcal{B} v - \mathcal{H} v s+s = 0 \right\}.
\end{align}
The manifold $\mathcal{C}$ has (up to) four branches $\mathcal{C}_i$, $i=0,1,2,3$, defined by
\begin{align}\label{eq:case(i)_C0_Ci}
\mathcal{C}_0:=\left\{v=q=s=0\right\},\quad \mathcal{C}_i:=\left\{q=0, v=v_i(u), s=\frac{\mathcal{B} v_i(u)}{1-\mathcal{H}v_i(u)}\right\}
\end{align}
where $v_i(u), i=1,2,3$ are the (up to) three real roots of the cubic
\begin{align}\label{eq:case(i)_cubicroots}
\mathcal{H}kuv^3-(k+\mathcal{H})uv^2+uv-\mathcal{B}=0.
\end{align}
We will examine the nature of the roots of this cubic in the following subsections.
The reduced superslow dynamics on the critical manifolds is given by 
\begin{subequations}\label{eq:ODEsystem_case(i)_superslow}
 \begin{align}
  u_z &= p,\\
  p_z &= u v_i(u)^2-\mathcal{A}\left(1-u\right)
 \end{align}
\end{subequations}
where $v_0(u)\equiv0$. The global bare soil equilibrium $(1,0,0,0,0)$ lies on $\mathcal{C}_0$, where it coincides with the saddle-type fixed point $(u,p) = (1,0)$ in~\eqref{eq:ODEsystem_case(i)_superslow}. In order to construct singular homoclinic orbits to this fixed point, we next examine the flow in the fast and intermediate slow timescales. To facilitate the upcoming analysis, we fix $\delta >0$ and take the limit $\eps \downarrow 0$ in \eqref{eq:ODEsystem_case(i)} to obtain the so-called semi-reduced system 
\begin{subequations}\label{eq:ODEsystem_modified_rescaled_semireduced_2}
 \begin{align}
  u &= u_0,\\
  p &= p_0,\\
  \delta v_y &= q,\\
  \delta q_y &= \mathcal{B} v  - u v^2(1-k v) + \mathcal{H} v s - c q,\\
  s_y &= -\mathcal{B} v- \mathcal{H} v s + s.
 \end{align}
\end{subequations}
The semi-reduced system will be studied for $0<\delta \ll 1$ in the upcoming subsections.

\subsection{Slow dynamics}\label{ssec:case(i)_slowdynamics}
The intermediate slow dynamics on the $y$-timescale are uncovered by taking $\delta\to0$ in \eqref{eq:ODEsystem_modified_rescaled_semireduced_2}. The limiting system yields the reduced dynamics
\begin{equation}\label{eq:case(i)_sflow_E}
 s_y = -\mathcal{B} v- \mathcal{H} v s + s, \end{equation}
and admits a one-dimensional critical manifold
\begin{align}
\mathcal{E}:=\left\{u=u_0, p=p_0, q=0, \mathcal{B}v-uv^2(1-kv)+\mathcal{H}vs=0\right\}.
\end{align}
The manifold $\mathcal{E}$ admits three normally hyperbolic branches $\mathcal{E}_0,\mathcal{E}_1^\pm$ defined by
\begin{align}\label{eq:case(i)_E0_Epm}
\mathcal{E}_0:=\{v=q=0\},\quad \mathcal{E}_1^\pm:=\left\{q=0, v=v_\pm(u_0,s)\right\}
\end{align}
where
\begin{align}\label{eq:case(i)_vpm}
v_\pm(u_0,s):=\frac{1}{2k}\left(1\pm \sqrt{1-\frac{4k(\mathcal{B}+\mathcal{H}s)}{u_0}}   \right).
\end{align}
The two branches $\mathcal{E}_1^\pm$ meet along the nonhyperbolic fold defined by 
\begin{align}\label{eq:case(i)_fold_F}
\mathcal{F}:=\left\{u=u_0, p=p_0, q=0, v=\frac{1}{2k}, s=\frac{1}{\mathcal{H}}\left(\frac{u_0}{4k}-\mathcal{B}\right)\right\};
\end{align}
note that branch $\mathcal{E}_1^-$ is normally hyperbolic if and only if $c \neq 0$ (see Section \ref{sec:case(i)_fast}).
The reduced flow on the branch $\mathcal{E}_0$ is given by
\begin{equation}\label{eq:case(i)_sflow_E0}
    s_y = s,
\end{equation}
 while the reduced flow on the branches $\mathcal{E}_1^\pm$ is given by
 \begin{align}\label{eq:ODEsystem_case(i)_slowreduced}
 s_y &= - \mathcal{B} v_\pm(u_0,s)- \mathcal{H} v_\pm(u_0,s) s + s, 
 \end{align}
cf. \eqref{eq:case(i)_sflow_E}. The equilibria of the reduced flow correspond to the intersection of $\mathcal{E}$ with the superslow manifold $\mathcal{C}$ \eqref{eq:case(1)_superslow_C} restricted to the hyperplane $\left\{ u=u_0, p=p_0\right\}$. We now examine the nature of this intersection, and identify several possible subcases.\\

The $s$-dynamics on branch $\mathcal{E}_0$ \eqref{eq:case(i)_sflow_E0} admits a repelling equilibrium at $s=0$. The equilibria of~\eqref{eq:ODEsystem_case(i)_slowreduced} on the other branches $\mathcal{E}_1^\pm$ are given by the roots of the cubic~\eqref{eq:case(i)_cubicroots}; their existence, position and nature depend on $u_0$ and on the values of the parameters $\mathcal{B}, \mathcal{H}$ and $k$. We consider the intersection of the nullcline of \eqref{eq:case(i)_sflow_E} with the projection of $\mathcal{E}_1^\pm$ on the $(v,s)$-plane. 
Both curves can be written as graphs over $v$:
\begin{subequations}\label{eq:ODEsystem_case(i)_slownullclinecurves}
 \begin{align}
 s &= \frac{1}{\mathcal{H}}\left(u_0 v(1-k v)-\mathcal{B}\right), \label{eq:ODEsystem_case(i)_slownullclinecurves_a}\\
 s&= \frac{\mathcal{B} v}{1-\mathcal{H} v}. \label{eq:ODEsystem_case(i)_slownullclinecurves_b}
 \end{align}
\end{subequations}
The parabola~\eqref{eq:ODEsystem_case(i)_slownullclinecurves_a} opens downward, with a maximum at $(v,s)=\left(\frac{1}{2k},\frac{1}{\mathcal{H}}\left(\frac{u_0}{4k}-\mathcal{B}\right)\right)$ corresponding to the fold $\mathcal{F}$ \eqref{eq:case(i)_fold_F}, which lies in the region $v,s>0$ provided $\mathcal{B}<\frac{u_0}{4k}$. The hyperbola~\eqref{eq:ODEsystem_case(i)_slownullclinecurves_b} has a vertical asymptote at $v=\frac{1}{\mathcal{H}}>0$ and horizontal asymptote at $s=-\frac{\mathcal{B}}{\mathcal{H}}<0$. Since the parabola~\eqref{eq:ODEsystem_case(i)_slownullclinecurves_a} opens downward and intersects the $s$-axis at $s=-\frac{\mathcal{B}}{\mathcal{H}}$, the graphs \eqref{eq:ODEsystem_case(i)_slownullclinecurves} always intersect in the region $v>\frac{1}{\mathcal{H}}, s<-\frac{\mathcal{B}}{\mathcal{H}}$, corresponding to a real root of~\eqref{eq:case(i)_cubicroots}. Note that this equilibrium is not relevant for our analysis, as it always occurs in the region $s<0$. There may be, however, up to two additional intersections, or roots of~\eqref{eq:case(i)_cubicroots} which may occur in the region $v,s>0$. These correspond to equilibria of~\eqref{eq:ODEsystem_modified_rescaled_semireduced_2} and are given by $(v,q,s)=(v_i(u_0),0,s_i(u_0))$, $i=1,2$, where
\begin{align}\label{eq:case(i)_si_vi}
    s_i(u_0) = \frac{\mathcal{B}v_i(u_0)}{1-\mathcal{H}v_i(u_0)},
\end{align}
and $v_1(u_0) \leq v_2(u_0)$. This pair of equilibria appears as the nullcline~\eqref{eq:ODEsystem_case(i)_slownullclinecurves_b} tangentially intersects (the projection of) $\mathcal{E}_1^-$, in the region $v<\frac{1}{2k}$. If this pair of equilibria exists, then either both lie on $\mathcal{E}_1^-$, in which case the lower equilibrium $(v_1,0,s_1)$ is repelling and the upper equilibrium $(v_2,0,s_2)$ is attracting (within $\mathcal{E}_1^-$), or $(v_1,0,s_1)$ lies on $\mathcal{E}_1^-$, while $(v_2,0,s_2)$ has crossed through the fold $\mathcal{F}$ and lies on $\mathcal{E}_1^+$; in the latter case, both equilibria are repelling as equilibria of the slow reduced flow on their respective branches $\mathcal{E}_1^\pm$ \eqref{eq:ODEsystem_case(i)_slowreduced}, see also Figure \ref{fig:case(i)_Epm_Snullcline_vsplane}. The equilibria lie on separate branches of $\mathcal{E}_1^\pm$ provided
\begin{align}\label{eq:case(i)_v2_on_E+}
    \mathcal{H}<2k\quad \text{and}\quad\frac{u_0}{4k}-\mathcal{B}>\frac{\mathcal{B}\mathcal{H}}{2k-\mathcal{H}}.
\end{align}

\begin{figure}
    \centering
    \includegraphics[width=0.4\textwidth]{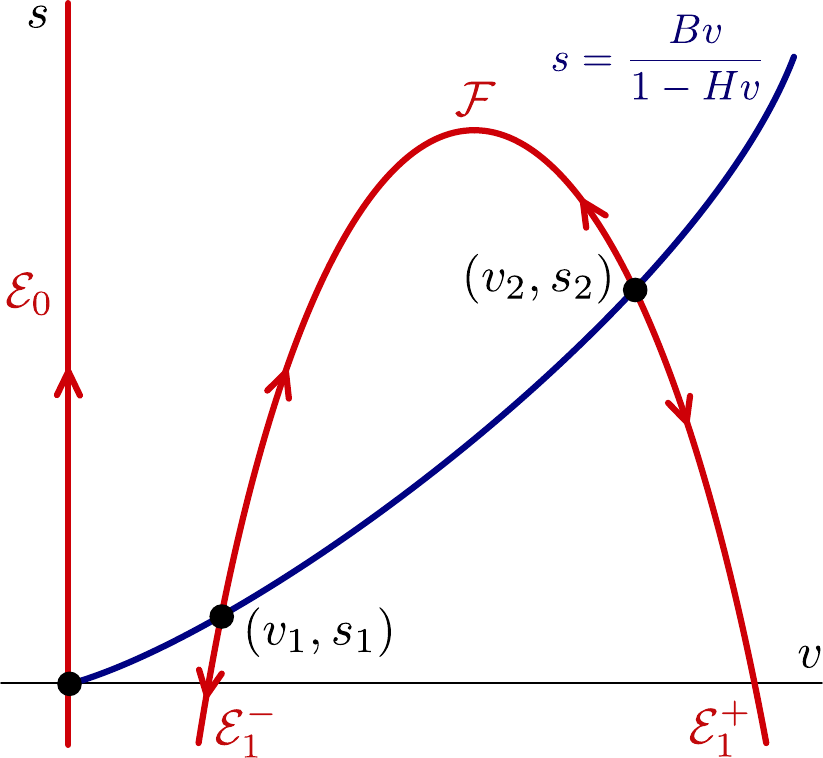}\hspace{0.1\textwidth}
    \includegraphics[width=0.4\textwidth]{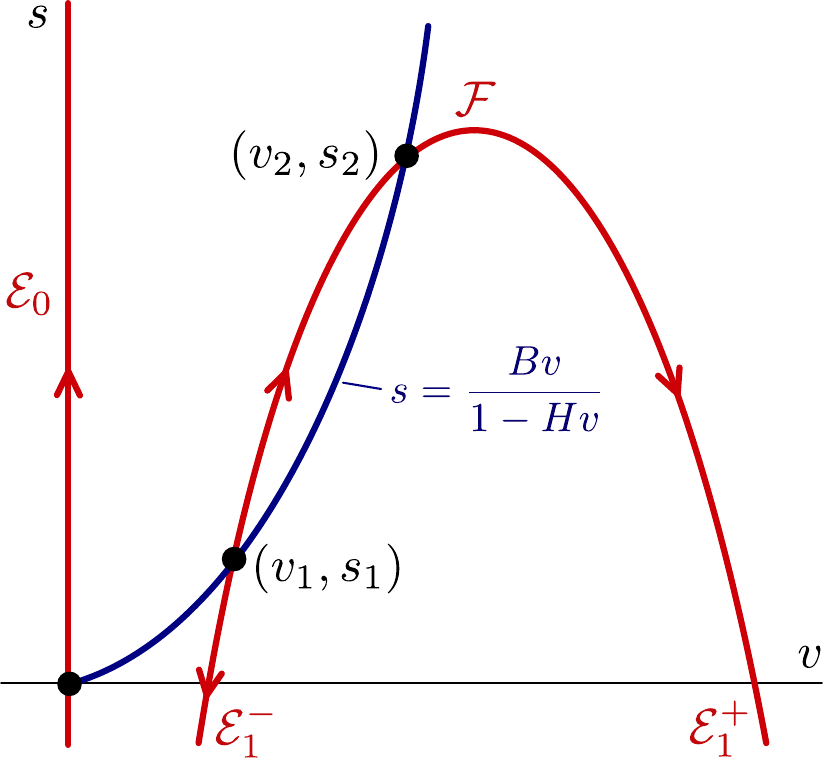}
    \caption{The projection of the equilibria $(v_i,0,s_i)$ on the $(v,s)$-plane in the case $\mathcal{H}<2k$ and $\frac{u_0}{4k}-\mathcal{B}>\frac{\mathcal{B}\mathcal{H}}{2k-\mathcal{H}}$ (left), $\frac{u_0}{4k}-\mathcal{B}<\frac{\mathcal{B}\mathcal{H}}{2k-\mathcal{H}}$ (right). The reduced flow \eqref{eq:ODEsystem_case(i)_slowreduced} on the branches $\mathcal{E}_0$ and $\mathcal{E}_1^\pm$ is indicated in red.}
    \label{fig:case(i)_Epm_Snullcline_vsplane}
\end{figure}

\subsection{Fast dynamics}\label{sec:case(i)_fast}
The fast dynamics take place on the $\zeta$-timescale. In the fast limit $\eps\to0, \delta\to0$ in \eqref{eq:ODEsystem_modified_fast_rescaled}, we obtain the planar layer problem
\begin{subequations}\label{eq:ODEsystem_case(i)_layer}
 \begin{align}
  u &= u_0,\\
  p &= p_0,\\
  v_\zeta &= q,\\
  q_\zeta &= \mathcal{B} v  - u v^2(1-k v) + \mathcal{H} v s - c q,\\
  s &= s_0.
 \end{align}
\end{subequations}
In the region $s_0>\frac{1}{\mathcal{H}}\left(\frac{u_0}{4k}-\mathcal{B}\right)$, this system admits a single saddle-type equilibrium at $(v,q)=(0,0)$, while for $s_0<\frac{1}{\mathcal{H}}\left(\frac{u_0}{4k}-\mathcal{B}\right)$, there are two additional equilibria $(v,q) = (v_\pm(u_0,s_0),0)$ \eqref{eq:case(i)_vpm}, where the equilibrium $(v_+(u_0,s_0),0)$ is also of saddle type. We search for heteroclinic connections between the normally hyperbolic saddle branches $\mathcal{E}_0$ and $\mathcal{E}_1^+$ \eqref{eq:case(i)_E0_Epm} by constructing heteroclinic orbits in~\eqref{eq:ODEsystem_case(i)_layer} between $(0,0)$ and $(v_+(u_0,s_0),0)$; note that for $s_0 >0$ this implies
\begin{equation}\label{eq:case(i)_assumption_u0_heteroclinic}
    u_0 >4 k \mathcal{B}.
\end{equation}
By rewriting
\begin{subequations}\label{eq:ODEsystem_case(i)_layer_re}
 \begin{align}
  v_\zeta &= q,\\
  q_\zeta &= - c q +k u_0 v(v-v_+(u_0,s_0))(v-v_-(u_0,s_0)),
 \end{align}
\end{subequations}
and searching for solutions which can be represented as a graph of $q$ over $v$, in the region $s_0<\frac{1}{\mathcal{H}}\left(\frac{u_0}{4k}-\mathcal{B}\right)$, we find two possible heteroclinic connections (see Figure~\ref{fig:case(i)_fast_heteroclinic})
\begin{subequations}\label{eq:ODEsystem_case(i)_heteroclinics}
 \begin{align}
\phi_\dagger(\zeta;u_0,s_0)&:=\left(v_\dagger, q_\dagger\right)(\zeta), \\
\phi_\diamond(\zeta;u_0,s_0)&:=\left(v_\diamond, q_\diamond\right)(\zeta),
 \end{align}
\end{subequations}
with $v$-profiles
\begin{subequations}\label{eq:ODEsystem_case(i)_heteroclinics_profiles}
 \begin{align}
v_\dagger(\zeta):=\frac{v_+(u_0,s_0)}{2}\left(1+\tanh\left(\frac{v_+(u_0,s_0)\sqrt{k u_0}}{2\sqrt{2}} \zeta\right)\right),\\
v_\diamond(\zeta):=\frac{v_+(u_0,s_0)}{2}\left(1-\tanh\left(\frac{v_+(u_0,s_0)\sqrt{k u_0}}{2\sqrt{2}} \zeta\right)\right),
 \end{align}
\end{subequations}
and corresponding wave speeds
\begin{equation}\label{eq:ODEsystem_case(i)_heteroclinics_wavespeeds}
 c_\dagger(u_0,s_0) :=\sqrt{\frac{k u_0}{2}}\left(2v_-(u_0,s_0)-v_+(u_0,s_0)\right) = \frac{1}{2k}\left(1- \sqrt{1-3\frac{4k(\mathcal{B}+\mathcal{H}s_0)}{u_0}} \right) =: - c_\diamond(u_0,s_0).
\end{equation} 

\begin{figure}
    \centering
    \includegraphics[width=0.9\textwidth]{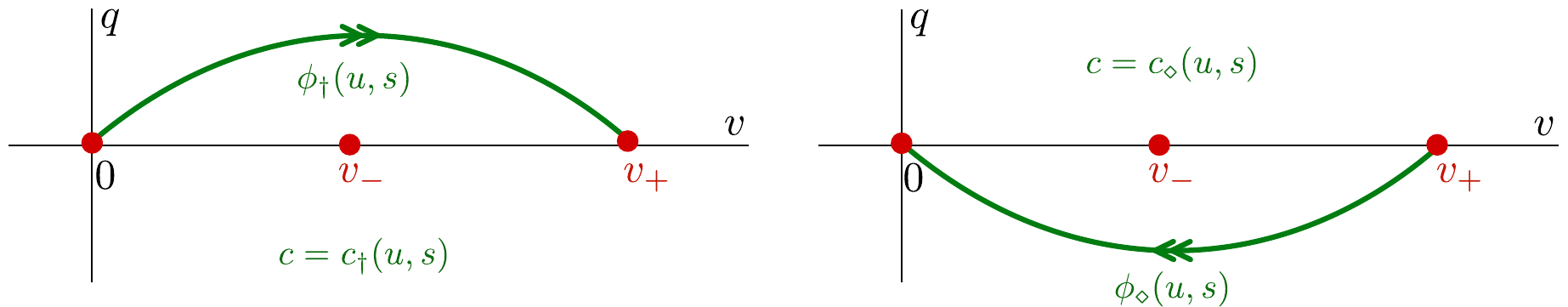}
    \caption{The fast heteroclinic jumps $\phi_\diamond,\phi_\dagger$.}
    \label{fig:case(i)_fast_heteroclinic}
\end{figure}
The heteroclinic $\phi_\dagger(\zeta)$ corresponds to a connection from $\mathcal{E}_0$ to $\mathcal{E}_1^+$, while $\phi_\diamond(\zeta)$ corresponds to a connection from $\mathcal{E}_1^+$ to $\mathcal{E}_0$. In particular, as both $\mathcal{E}_1^+$ and $\mathcal{E}_0$ transversely intersect the subspace $\left\{s=0\right\}$, there exists a heteroclinic connection $\phi_\diamond(\zeta;u_0,0)$ from $\mathcal{E}_1^+$ to $\mathcal{E}_0$ in the subspace $\left\{s=0\right\}$. 
Direct calculation reveals that there exists a unique value $s_0 = s_*<\frac{1}{\mathcal{H}}\left(\frac{u_0}{4k}-\mathcal{B}\right)$, where $s_*$ is a solution to the equation
\begin{equation}\label{eq:caseI_implicit_c}
c_\dagger(u_0,s_*)=c_\diamond(u_0,0),
\end{equation}
for which $\phi_\dagger(\zeta;u_0,s_*)$ is a heteroclinic connection from $\mathcal{E}_0$ to $\mathcal{E}_1^+$, such that $\phi_\diamond(\zeta;u_0,0)$ and $\phi_\dagger(\zeta;u_0,s_*)$ have the same wave speed $c$. This value $s_*$ is positive if
\begin{equation}
u_0 > \frac{9}{2} \mathcal{B} k.
\end{equation}

Given such a heteroclinic connection $\phi_*$ with speed $c_*(u_0,s_0)$ for some $(u_0,s_0)$, where $*=\dagger$ or $\diamond$, we note that this connection persists for any choice of $p$ since the fast subsystem is independent of $p$. However, the following result shows that this heteroclinic connection breaks transversely upon varying $(u,s,c)$ near $(u,s,c_*(u_0,s_0))$.

\begin{lem}\label{lem:case(i)_layer_transversality}
Fix $(u_0,c_0)$ so that~\eqref{eq:ODEsystem_case(i)_layer} admits a heteroclinic connection $\phi_*(\zeta)=(v_*,q_*)(\zeta)$ with speed $c_*(u_0,s_0)$, where $*=\dagger$ or $\diamond$. Then this connection breaks transversely upon varying $(u,s,c)$, as measured by the splitting distance 
\begin{align}\label{eq:heteroclinic_split}
\begin{split}
    D(u,s,c) &= M^u_*(u-u_0)+M^s_*(s-s_0)+M^c_*(c-c_*(u_0,s_0))\\
    &\qquad+\mathcal{O}\left(|u-u_0|^2+ |s-s_0|^2+|c-c_*(u_0,s_0)|^2 \right)
    \end{split}
\end{align}
where the coefficients $M^u_*,M^s_*,M^c_*$ are given by the nonzero Melnikov integrals
\begin{align*}
    M^u_* &= \int_{-\infty}^\infty e^{c_*\zeta}v_*(\zeta)^2\left(1-k v_*(\zeta)\right)v_*'(\zeta)  \mathrm{d}\zeta \neq 0\\
    M^s_*&= -\int_{-\infty}^\infty e^{c_*\zeta}\mathcal{H} v_*(\zeta)v_*'(\zeta)  \mathrm{d}\zeta \neq 0 \\
    M^c_*&= \int_{-\infty}^\infty e^{c_*\zeta}v_*'(\zeta)^2  \mathrm{d}\zeta >0
\end{align*}
\end{lem}
\begin{proof}
The splitting distance~\eqref{eq:heteroclinic_split} within the planar system
\begin{subequations}\label{eq:ODEsystem_case(i)_layer_planar}
 \begin{align}
  v_\zeta &= q,\\
  q_\zeta &= \mathcal{B} v  - u v^2(1-k v) + \mathcal{H} v s - c q,
 \end{align}
 \end{subequations}
 can be computed using Melnikov theory. Linearizing about the heteroclinic connection $\phi_*$, we obtain the variational equation
 \begin{align}\label{eq:ODEsystem_case(i)_layer_var}
\begin{pmatrix}v\\q  \end{pmatrix}_\zeta &= \begin{pmatrix} 0 & 1\\ \mathcal{B}+\mathcal{H} s - u \left(2v_*(\zeta)-3k v_*(\zeta)^2\right) & -c \end{pmatrix}\begin{pmatrix}v\\q  \end{pmatrix}
 \end{align}
 and the corresponding adjoint equation
 \begin{align}\label{eq:ODEsystem_case(i)_layer_adj}
\begin{pmatrix}v\\q  \end{pmatrix}_\zeta &= -\begin{pmatrix} 0 & \mathcal{B}+\mathcal{H} s - u \left(2v_*(\zeta)-3k v_*(\zeta)^2\right)\\ 1 & -c \end{pmatrix}\begin{pmatrix}v\\q  \end{pmatrix}
\end{align}
which admits a unique (up to scalar multiple) bounded solution 
 \begin{align}\label{eq:ODEsystem_case(i)_layer_adj_bounded}
\psi_*(\zeta):=e^{c_*\zeta}\begin{pmatrix}q_*'(\zeta)\\-v_*'(\zeta)  \end{pmatrix}=e^{c_*\zeta}\begin{pmatrix}q_*'(\zeta)\\-q_*(\zeta).  \end{pmatrix}
 \end{align}
 We let $F_0(v,q;u,s,c)$ denote the right hand side of~\eqref{eq:ODEsystem_case(i)_layer_planar} and define the Melnikov integrals
 \begin{align}
     M^\nu_*:=\int_{-\infty}^\infty D_\nu F(v_*(\zeta),q_*(\zeta);u_0,s_0,c_*(u_0,s_0))\cdot \psi_*(\zeta)\mathrm{d}\zeta
 \end{align}
 where $\nu = u,s,c$. We compute
 \begin{align*}
    M^u_* &= \int_{-\infty}^\infty e^{c_*\zeta}v_*(\zeta)^2\left(1-k v_*(\zeta)\right)v_*'(\zeta)  \mathrm{d}\zeta \neq 0\\
    M^s_*&= -\int_{-\infty}^\infty e^{c_*\zeta}\mathcal{H} v_*(\zeta)v_*'(\zeta)  \mathrm{d}\zeta \neq 0 \\
    M^c_*&= \int_{-\infty}^\infty e^{c_*\zeta}v_*'(\zeta)^2  \mathrm{d}\zeta >0
\end{align*}
and obtain the leading order distance function~\eqref{eq:heteroclinic_split}. We note that the Melnikov integrals are all nonzero, but the signs of $M^u_*, M^s_*$ depend on whether $*=\dagger$ or $\diamond$, since this determines the sign of $v_*'(\zeta)$.
\end{proof}

\subsection{Singular profiles}\label{sec:case(i)_profiles}
We now construct singular homoclinic orbits to the fixed point $(u,p,v,q,s)=(1,0,0,0,0)$. We identify two cases in which it is possible to construct such a singular orbit. These two cases are distinguished by whether the associated traveling pulse has a `superslow plateau'.

\subsubsection{Homoclinic orbits without superslow plateau}\label{ssec:case(i)_withoutplateau}
We first consider the simpler case, without a superslow plateau. In this case, we aim to construct a singular pulse contained entirely in the subspace $\{u_0=1, p_0=0\}$. Note that this subspace is not invariant under the full flow \eqref{eq:ODEsystem_modified_rescaled}/\eqref{eq:ODEsystem_modified_fast_rescaled}; hence, the persistent solution based on the upcoming singular construction (see Theorem \ref{case(i)_thm_noplateau}) will have $u \not\equiv 1$.

Consider the fast dynamics~\eqref{eq:ODEsystem_case(i)_layer} in the subspace $\{u_0=1, p_0=0\}$. From the previous section, provided there exists $s=s_*<\frac{1}{\mathcal{H}}\left(\frac{1}{4k}-\mathcal{B}\right)$ such that $c_\dagger(1,s_*)=c_\diamond(1,0)$ (cf. \eqref{eq:caseI_implicit_c}), we can construct a pair of heteroclinic orbits $\phi_\diamond(\zeta;1,0)$ and $\phi_\dagger(\zeta;1,s_*)$, where $\phi_\dagger(\zeta;1,s_*)$ jumps from $\mathcal{E}_0$ to $\mathcal{E}_1^+$ \eqref{eq:case(i)_E0_Epm} in the subspace $\left\{s=s_*\right\}$, and $\phi_\diamond(\zeta;1,0)$ jumps from $\mathcal{E}_1^+$ back to $\mathcal{E}_0$ in the subspace $\left\{s=0\right\}$. The condition $c_\dagger(1,s_*)=c_\diamond(1,0)$ \eqref{eq:caseI_implicit_c} can be satisfied by using the fact that $v_-(1,s)+v_+(1,s)=\frac{1}{k}$ \eqref{eq:case(i)_vpm} and solving
\begin{align}
\sqrt{\frac{k}{2}}\left(\frac{2}{k}-3v_+(1,s_*)\right) = \sqrt{\frac{k}{2}}\left(3v_+(1,0) - \frac{2}{k}\right),
\end{align}
cf. \eqref{eq:ODEsystem_case(i)_heteroclinics_wavespeeds}, which has a solution when $4k\mathcal{B}>\frac{5}{9}$, given by
\begin{align}\label{case(i)_solution_sstar}
s_* = \frac{-1+3\sqrt{1-4k\mathcal{B}}}{9\mathcal{H}k}.
\end{align}
which is positive provided $4k\mathcal{B}<\frac{8}{9}$.

\begin{rmk}
In the region  $4k\mathcal{B}<\frac{5}{9}$, it is still possible to construct a heteroclinic orbit $\phi_\dagger(\zeta;1,s_*)$, departing $\mathcal{E}_0$ in the subspace $s=s_*=\frac{1}{\mathcal{H}}\left(\frac{1}{4k}-\mathcal{B}\right)$. However, this orbit would then land precisely at the nonhyperbolic fold point $\mathcal{F}$ \eqref{eq:case(i)_fold_F}. While it should still be feasible to construct traveling pulses in this regime (see e.g.~\cite{BCD.2019}), we exclude this particular case from consideration here and focus only on the case where $\phi_\dagger(\zeta;1,s_*)$ touches down on a normally hyperbolic portion of $\mathcal{E}_1^+$.
\end{rmk}

In order to complete the construction of a singular homoclinic orbit, it remains to consider the slow $s$-dynamics. The reduced flow on $\mathcal{E}_0$ \eqref{eq:case(i)_sflow_E0} is simply $s_y = s$, so that an orbit from the fixed point $(v,q,s)=(0,0,0)$ can travel upwards along $\mathcal{E}_0$ via the slow flow until reaching $s=s_*$, then following the heteroclinic connection $\phi_\dagger(\zeta;1,s_*)$ to $\mathcal{E}_1^+$. The flow on $\mathcal{E}_1^+$ depends on the configuration of the equilibrium $(v,q,s) = (v_2(1),0,s_2(1))$ \eqref{eq:case(i)_si_vi}, in particular whether it lies on the branch $\mathcal{E}_1^+$, see Figure \ref{fig:case(i)_Epm_Snullcline_vsplane} (left). We recall that this situation occurs provided $\mathcal{H}<2k$ and $    \frac{1}{4k}-\mathcal{B}>\frac{\mathcal{B}\mathcal{H}}{2k-\mathcal{H}}$ \eqref{eq:case(i)_v2_on_E+}.
In this case, there is a possibility that $0< s_2(1) < s_*$, the equilibrium $(v_2(1),0,s_2(1))$ thereby `blocking' any orbit from following the slow flow from $s=s_*$ down to $s=0$. However, if either $s_2(1) > s_*$, or one of the conditions \eqref{eq:case(i)_v2_on_E+} is violated in an open region of parameter space, then the equilibrium $(v_2(1),0,s_2(1))$ will not cause an obstruction, and the slow flow can be followed from $s=s_*$ to $s=0$, after which the heteroclinic orbit $\phi_\diamond(\zeta;1,0)$ can be followed back to the equilibrium $(v,q,s)=(0,0,0)$. We can therefore construct a singular homoclinic orbit of the full system in the subspace $\{u=1, p=0\}$ via the concatenation
\begin{align}\label{case(i)_homoclinic_concatenation}
\mathcal{E}_0[0,s_*] \cup \phi_\dagger(\zeta;1,s_*) \cup \mathcal{E}_1^+[0,s_*] \cup \phi_\diamond(\zeta;1,0)
\end{align}
where $\mathcal{E}_0[0,s_*]$ denotes the intersection $\mathcal{E}_0\cap \{0\leq s\leq s_*\}$ and similarly for $\mathcal{E}_1^+$;  see Figure~\ref{fig:case(i)_no_plateau}.

\begin{figure}
    \centering
    \includegraphics[width=0.45\textwidth]{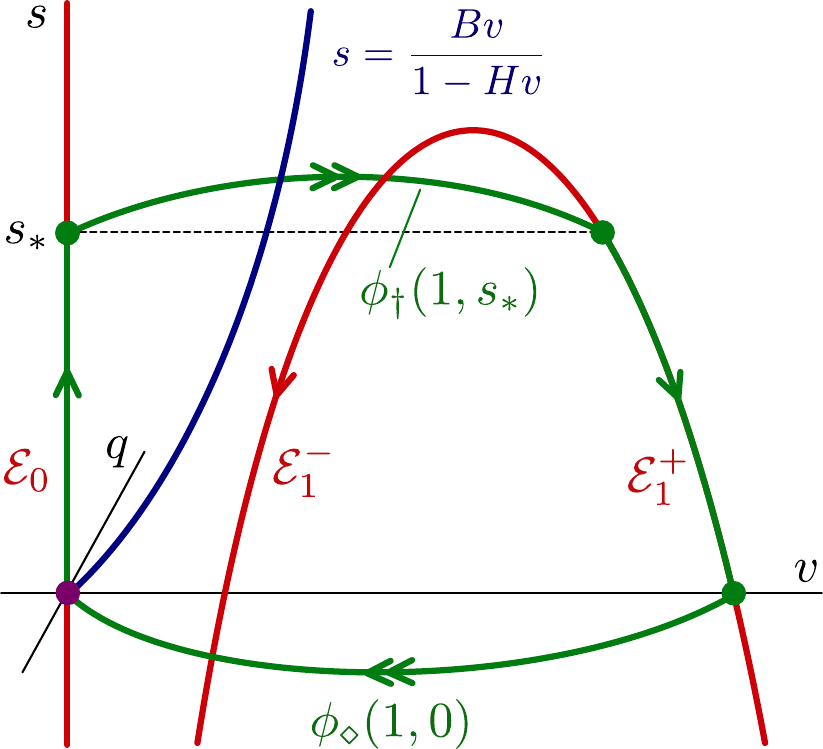}
    \caption{The singular orbit $\mathcal{E}_0[0,s_*] \cup \phi_\dagger(\zeta;1,s_*) \cup \mathcal{E}_1^+[0,s_*] \cup \phi_\diamond(\zeta;1,0)$ in the subspace $(u_0,p_0) = (1,0)$ the case without superslow plateau.}
    \label{fig:case(i)_no_plateau}
\end{figure}

\subsubsection{Homoclinic orbits with superslow plateau}\label{ssec:case(i)_withplateau}
The singular orbits constructed in the previous section were contained entirely within the subspace $\{u=1,\,p=0\}$. We now consider singular homoclinic orbits with nontrivial superslow $(u,p)$ dynamics, which will contain a superslow plateau state in the vegetated $(v >0)$ region of the pulse. It is possible to construct such an orbit in the case where the equilibrium $(v,q,s) = (v_2,0,s_2)$ of~\eqref{eq:ODEsystem_case(i)} for $\eps=0$ has moved onto $\mathcal{E}_1^+$ for some values of $u<1$, which occurs provided
$\mathcal{H}<2k$ and$ \frac{u}{4k}-\mathcal{B}>\frac{\mathcal{B}\mathcal{H}}{2k-\mathcal{H}}$, cf. \eqref{eq:case(i)_v2_on_E+}.
Note that the union of $(v_2,0,s_2)$ taken over all $(u,p)$-values corresponds to the superslow critical manifold $\mathcal{C}_2$ \eqref{eq:case(i)_C0_Ci}; thus here we are interested in the case where $\mathcal{C}_2$ intersects $\mathcal{E}$ on the branch $\mathcal{E}_1^+$ in the region $u<1$. For particular values of $u=u_*<1, p=p_*$, it may be possible to construct slow/fast heteroclinic orbits connecting the equilibria $(v,q,s)=(0,0,0)$ and $(v,q,s) = (v_2(u_*),0,s_2(u_*))$ and vice versa within the subspace $\{u=u_*\}$ for $\delta\ll1$ sufficiently small. These two orbits can then be glued along a superslow plateau trajectory which is contained in $\mathcal{C}_2$.

To see this, we recall from~\S\ref{sec:case(i)_superslow} that the reduced superslow dynamics are restricted to the two-dimensional critical manifold $\mathcal{C}$ \eqref{eq:case(1)_superslow_C}, and that in particular the bare soil equilibrium $(u,p,v,q,s) = (1,0,0,0,0)$ of the full system corresponds to the saddle fixed point $(u,p) = (1,0)$ on $\mathcal{C}_0$ with corresponding dynamics
\begin{subequations}\label{eq:ODEsystem_case(i)_superslow_C0}
 \begin{align}
  u_z &= p,\\
  p_z &= -\mathcal{A}\left(1-u\right).
 \end{align}
\end{subequations}
Within $\mathcal{C}_0$, the fixed point $(1,0)$ has one-dimensional stable and unstable manifolds given by the lines 
\begin{subequations}\label{case(i)_superslow_ell0pm}
\begin{align}
   \mathcal{W}_0^\mathrm{u}(1,0) &= \ell^+_0:=\left\{p = \sqrt{\mathcal{A}}\left(u-1\right)\right\},\\
   \mathcal{W}_0^\mathrm{s}(1,0) &= \ell^-_0:=\left\{p = \sqrt{\mathcal{A}}\left(1-u\right)\right\}.
\end{align}
\end{subequations}
We next consider the flow on the manifold $\mathcal{C}_2$ \eqref{eq:ODEsystem_case(i)_superslow}, with dynamics
\begin{subequations}\label{eq:ODEsystem_case(i)_superslow_C2}
 \begin{align}
  u_z &= p,\\
  p_z &= u\, v_2(u)^2-\mathcal{A}\left(1-u\right).\label{eq:ODEsystem_case(i)_superslow_C2_p}
 \end{align}
\end{subequations}
Although the explicit expression of $v_2(u)$ is algebraically rather complicated, we can infer properties of the flow on $\mathcal{C}_2$ by solving the defining equation \eqref{eq:case(i)_cubicroots} for $u$, yielding
\begin{equation}\label{eq:case(i)_invertedcubic}
    u = \frac{\mathcal{B}}{v\left(1-\mathcal{H}v\right)\left(1 - k v\right)}.
\end{equation}
Note that the projection of $\mathcal{C}_1 \cup \mathcal{C}_2 \cup \mathcal{C}_3$ \eqref{eq:case(i)_C0_Ci} onto the $(v,u)$-plane is given by the graph of \eqref{eq:case(i)_invertedcubic}. The projection of the $p$-nullcline of \eqref{eq:ODEsystem_case(i)_superslow_C2} onto the $(v,u)$-plane is given by the graph of $u = \frac{\mathcal{A}}{\mathcal{A} + v^2}$. As this graph is strictly monotonically decreasing for $v>0$, intersects the $u$-axis at $u=1$, and limits to $0$ as $v \to \infty$, its intersections with the graph of \eqref{eq:case(i)_invertedcubic} occur for $u$-values between $0$ and $1$.\\
Assume that none of these intersections lie on the projection of $\mathcal{C}_2$ onto the $(v,u)$-plane, which means that \eqref{eq:ODEsystem_case(i)_superslow_C2} does not have an equilibrium. Then the sign of $p_z$ is fixed and positive, as the right hand side of \eqref{eq:ODEsystem_case(i)_superslow_C2_p} is positive for $u>1$. Let $(u_{\text{min}},\infty)$ be the projection of $\mathcal{C}_2$ onto the $u$-axis. From the Hamiltonian structure of \eqref{eq:ODEsystem_case(i)_superslow_C2}, it follows that for all $\hat{u} \in (u_{\text{min}},1)$ the orbit of \eqref{eq:ODEsystem_case(i)_superslow_C2} through the point $(\hat{u},0)$ intersects the line $\ell^-_0$ at $\left(u_*,\sqrt{\mathcal{A}}(1-u_*)\right)$, where $u_* \in (\hat{u}, 1)$. By reflection symmetry in the $u$-axis, the same orbit intersects the line $\ell^+_0$ at $\left(u_*,\sqrt{\mathcal{A}}(u_*-1)\right)$.\\
Alternatively, assume that the graph of $u = \frac{\mathcal{A}}{\mathcal{A}+v^2}$ does intersect the projection of $\mathcal{C}_2$ onto the $(v,u)$-plane, for a certain value $u = \hat{u}<1$. Linearisation of \eqref{eq:ODEsystem_case(i)_superslow_C2} at $(\hat{u},0)$ shows that this equilibrium is a saddle, as $\frac{\text{d}}{\text{d} u} v_2(u) > 0$. Again, from the Hamiltonian structure of \eqref{eq:ODEsystem_case(i)_superslow_C2}, it follows that for all $\tilde{u} \in (\hat{u},1)$ the orbit of \eqref{eq:ODEsystem_case(i)_superslow_C2} through the point $(\tilde{u},0)$ intersects the line $\ell^-_0$ at $\left(u_*,\sqrt{\mathcal{A}}(1-u_*)\right)$, where $u_* \in (\tilde{u}, 1)$. By reflection symmetry in the $u$-axis, the same orbit intersects the line $\ell^+_0$ at $\left(u_*,\sqrt{\mathcal{A}}(u_*-1)\right)$.\\
In either case, we denote the orbit segment from $\left(u_*,\sqrt{\mathcal{A}}(u_*-1)\right)$ to $\left(u_*,\sqrt{\mathcal{A}}(1-u_*)\right)$ by $\ell_2(u_*)$; see Figure \ref{fig:case(i)_dynamics_C2}.\\

\begin{figure}
    \centering
    \includegraphics[width=0.8\textwidth]{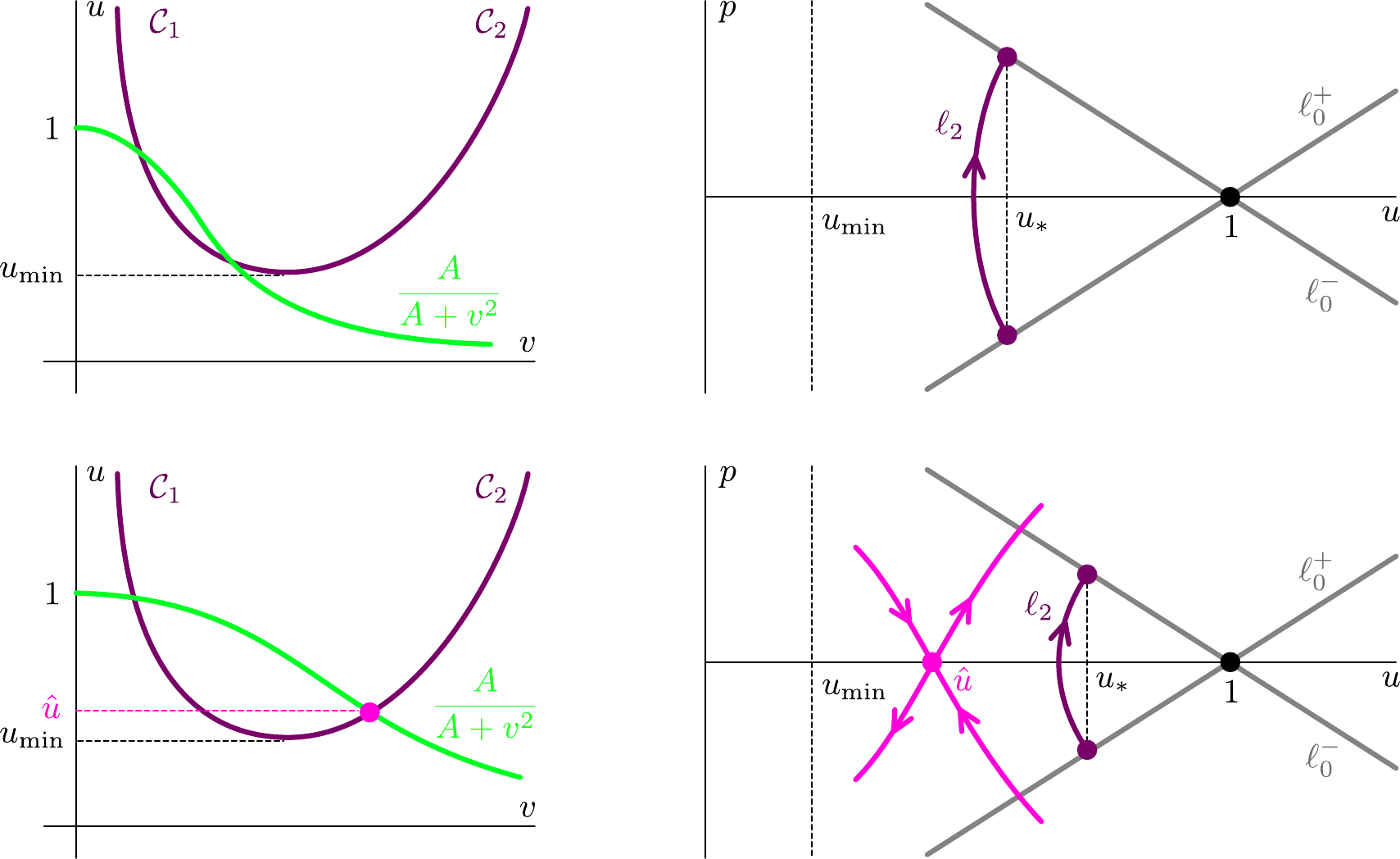}
    \caption{Top row, left: no equilibrium on $\mathcal{C}_2$; right: the associated superslow flow on $\mathcal{C}_2$ with orbit segment $\ell_2$. Bottom row, left: one equilibrium on $\mathcal{C}_2$; right: the associated superslow flow on $\mathcal{C}_2$ with saddle at $(\hat{u},0)$ and orbit segment $\ell_2$.}
    \label{fig:case(i)_dynamics_C2}
\end{figure}

We now construct slow/fast singular heteroclinic orbits connecting the equilibria $(v,q,s)=(0,0,0)$ and $(v,q,s) = (v_2(u_*),0,s_2(u_*))$ within the system
\begin{subequations}\label{eq:ODEsystem_case(i)_slowfast}
 \begin{align}
 v_\zeta &= q,\\
   q_\zeta &= \mathcal{B} v  - u_* v^2(1-k v) + \mathcal{H} v s - c q,\\
 s_\zeta &= \delta\left(- \mathcal{B} v- \mathcal{H} v s + s\right). 
 \end{align}
\end{subequations}
in the limit $\delta\to0$. Consider the fast dynamics~\eqref{eq:ODEsystem_case(i)_layer} in the subspace $\{u=u_*\}$. We recall from~\S\ref{sec:case(i)_fast} the existence of the heteroclinic orbits $\phi_\diamond(\zeta;u_0,0)$ and $\phi_\dagger(\zeta;u_0,s_*)$, where $\phi_\dagger(\zeta;u_0,s_*)$ jumps from $\mathcal{E}_0$ and $\mathcal{E}_1^+$ in the subspace $\left\{s=s_*\right\}$, and $\phi_\diamond(\zeta;u_0,0)$ jumps from $\mathcal{E}_1^+$ back to $\mathcal{E}_0$ in the subspace $\left\{s=0\right\}$, provided $u_0 > 4 k \mathcal{B}$ \eqref{eq:case(i)_assumption_u0_heteroclinic}.

\begin{figure}
    \centering
    \includegraphics[width=0.45\textwidth]{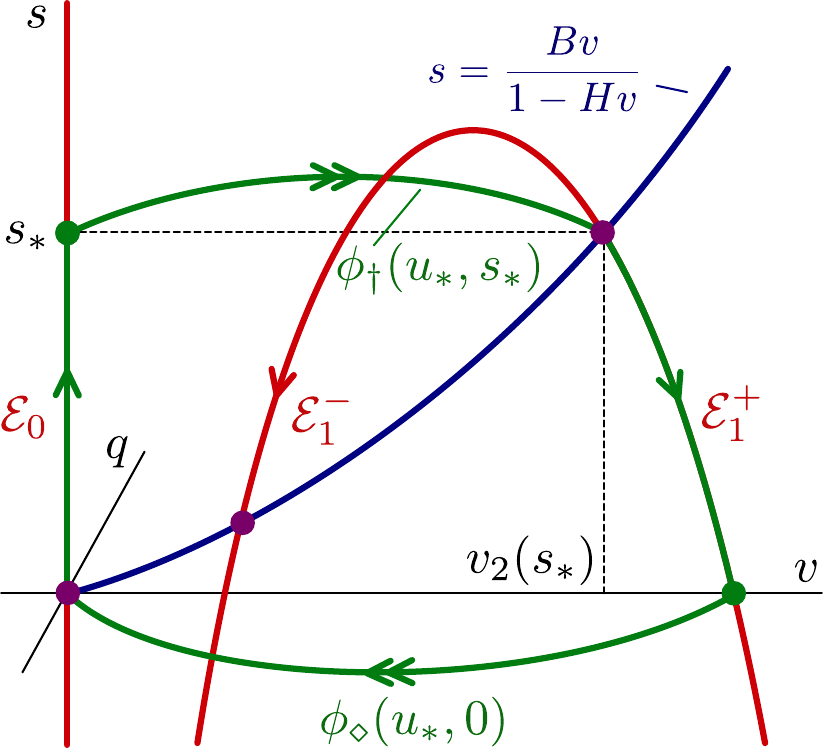}
    \caption{The singular slow/fast heteroclinic orbits $\mathcal{E}_0[0,s_*] \cup \phi_\dagger(\zeta;u_*,s_*)$ and  $\mathcal{E}_1^+[0,s_*] \cup \phi_\diamond(\zeta;u_*,0)$ in the subspace $u = u_*$ in the case with superslow plateau.}
    \label{fig:case(i)_plateau}
\end{figure}

We can construct a slow/fast heteroclinic orbit between $(v,q,s)=(0,0,0)$ and $(v,q,s) = (v_2(u_*),0,s_2(u_*))$ by first following the slow reduced flow on $\mathcal{E}_0$ from $s=0$ to $s=s_2(u_*)$, then following the fast jump $\phi_\dagger(\zeta;u_*,s_*)$ upon choosing $s_* = s_2(u_*)$. Likewise, a heteroclinic orbit between $(v,q,s) = (v_2(u_*),0,s_2(u_*))$ and $(v,q,s)=(0,0,0)$ can be obtained by following the slow flow on $\mathcal{E}_1^+$ from $s=s_2(u_*)$ to $s=0$, then taking the fast jump $\phi_\diamond(\zeta;u_*,0)$ back to the fixed point $(v,q,s)=(0,0,0)$. In order for this pair of heteroclinic orbits to exist simultaneously, the associated fast jumps $\phi_\dagger(\zeta;u_*,s_2(u_*))$ and $\phi_\diamond(\zeta;u_*,0)$ must have the same speed, that is, we must choose $u_*$ so that $c_\dagger(u_*,s_2(u_*))=c_\diamond(u_*,0)$, which results in the equation
\begin{align}
\sqrt{\frac{ku_*}{2}}\left(\frac{2}{k}-3v_+(u_*,s_2(u_*))\right) = \sqrt{\frac{ku_*}{2}}\left(3v_+(u_*,0) - \frac{2}{k}\right)
\end{align}
or equivalently,
\begin{align}\label{eq:case(i)_s_2_ustar}
s_2(u_*) = \frac{\left(-1+3\sqrt{1-\frac{4k\mathcal{B}}{u_*}}\right)u_*}{9\mathcal{H}k}.
\end{align}
where $s_*:=s_2(u_*)$ satisfies
\begin{align}
s_* &= \frac{\mathcal{B}v_+(u_*,s_*)}{1-\mathcal{H}v_+(u_*,s_*)}= \frac{\mathcal{B}\left(5-3\sqrt{1-\frac{4k\mathcal{B}}{u_*}}\right)}{6k-\mathcal{H}\left(5-3\sqrt{1-\frac{4k\mathcal{B}}{u_*}}\right)}.
\end{align}
Equating these two expressions, we obtain an equation for $U_* := \sqrt{1-\frac{4 k \mathcal{B}}{u_*}}$:
\begin{equation}\label{eq:case(i)_cubic_Ustar}
    \frac{\mathcal{B}}{9 \mathcal{H}} \frac{24 k (1-3U_*)+\mathcal{H}(5-3 U_*)^2 (1+3U_*)}{(U_*^2-1)(6k+\mathcal{H}(3U_*-5))} =0.
\end{equation}
From the analysis of the flow on $\mathcal{C}_2$ \eqref{eq:ODEsystem_case(i)_superslow_C2}, we know that for all $\mathcal{A}>0$ there exists a $0<\tilde{u}<1$ such that for all $u_* \in (\tilde{u},1)$, there exists an orbit segment $\ell_2(u_*)$ that connects $\ell_0^+$ to $\ell_0^-$. Note that this admissible interval can be made arbitrarily small by taking $\mathcal{A}$ sufficiently large, so that the saddle equilibrium $(\hat{u},0)$ is arbitrarily close to $(1,0)$.\\
The conditions \eqref{eq:case(i)_v2_on_E+} combined with the fact that $u_*<1$ translate to
\begin{equation}
   \sqrt{\frac{\mathcal{H}}{2k}} < U_* < \sqrt{1-4 k \mathcal{B}}.
\end{equation}

The existence of roots of \eqref{eq:case(i)_cubic_Ustar} in the required interval $\sqrt{\frac{\mathcal{H}}{2k}} < U_* < \sqrt{1-4 k \mathcal{B}}<1$ can be investigated by introducing $\gamma := \frac{\mathcal{H}}{k}$ and $x=3 U_*$. We determine the positive ($x>0$) intersections of the cubic $p(x) = (5-x)^2(1+x)$ with the line $l(x) = \frac{24}{\gamma}(x-1)$, in the interval $1<x<3$. The monotonicity properties of both $p$ and $l$ on this interval ensure that a unique positive intersection $x(\gamma)$ exists for all $0<\gamma<3$; its inverse is given by $\gamma(x) = \frac{24(x-1)}{(5-x)^2(x+1)}$, with $1<x<3$. We can directly calculate that $\gamma'(x) > 0$ and $\gamma''(x)>0$ for all $0<x<3$; hence, $\gamma(x)$ is a strictly monotonically increasing, convex function of $x$ on the interval $(0,3)$. It follows that $x(\gamma)$ is a strictly monotonically increasing, concave function of $\gamma$ on the interval $\gamma \in (0,3)$, with $x(0)=1$ and $x(3)=3$.\\
The lower inequality $\sqrt{\frac{\mathcal{H}}{2k}} < U_*$ translates to $\frac{3}{\sqrt{2}} \sqrt{\gamma} < x(\gamma)$. As the function $\gamma \mapsto \frac{3}{\sqrt{2}} \sqrt{\gamma}$ is strictly monotonically increasing and concave on the interval $\gamma \in (0,3)$, and since $x(0) = 1 > \frac{3}{\sqrt{2}} \sqrt{0} = 0$ and $x(3) = 3 < \frac{3}{\sqrt{2}} \sqrt{3}$, there exists a unique value $\gamma_*$ for which $x(\gamma_*) = \frac{3}{\sqrt{2}} \sqrt{\gamma}$, and that $\frac{3}{\sqrt{2}} \sqrt{\gamma} < x(\gamma)$ for all $0<\gamma<\gamma_*$. We can determine $\gamma_*$ explicitly by applying the remainder theorem to the polynomial $\frac{2}{9} x^2(5-x)^2(x+1)-24(x-1)$ since $\gamma_* = \frac{2}{9} x_*^2$, which yields $x_*=2$ (corresponding to $\gamma_*=\frac{8}{9}$) as an integer root, which also conveniently lies in the required interval $1<x<3$. Hence, we find that for all $0 < \gamma = \frac{\mathcal{H}}{k} < \frac{8}{9}$, there exists a unique root of \eqref{eq:case(i)_cubic_Ustar} with $1 < x_* = 3 U_* < 2$.
For the upper inequality $U_* < \sqrt{1-4 k \mathcal{B}}$, it is therefore sufficient to require $\frac{2}{3} < \sqrt{1-4 k \mathcal{B}}$, which implies $k \mathcal{B} < \frac{5}{36}$.\\



\begin{figure}
    \centering
    \includegraphics[width=0.5\textwidth]{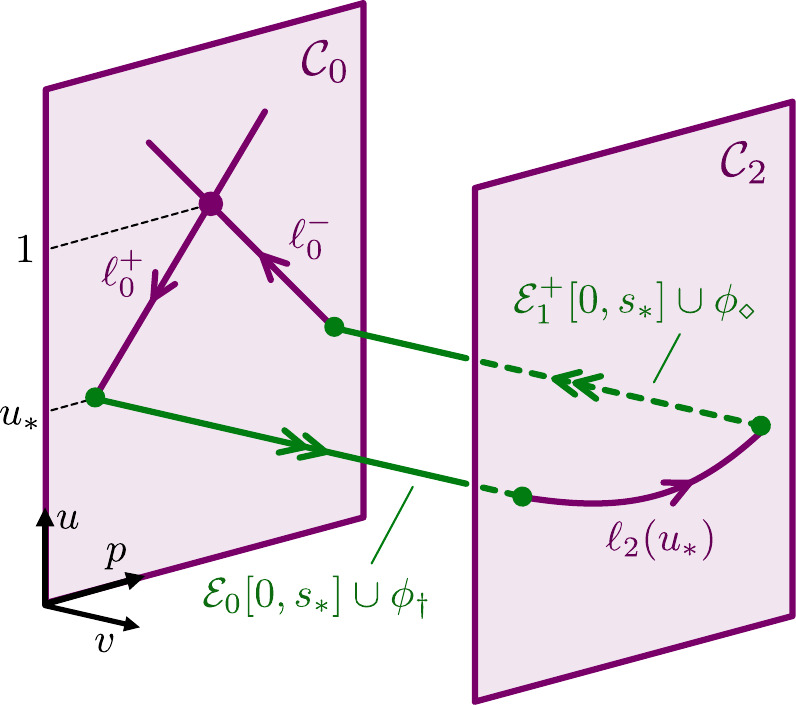}
    \caption{The singular homoclinic orbit $\ell^+_0 \cup \mathcal{E}_0[0,s_*] \cup \phi_\dagger(\cdot;u_*,s_*) \cup \ell_2(u_*) \cup \mathcal{E}_1^+[0,s_*] \cup \phi_\diamond(\cdot;u_*,0)\cup \ell^-_0$ in the case with superslow plateau.}
    \label{fig:case(i)_superslow_reduced}
\end{figure}

With $u_*, s_*$ defined as above, we can now complete the construction of a singular homoclinic orbit in~\eqref{eq:ODEsystem_case(i)} as follows. Starting at the equilibrium $(u,p,v,q,s)=(1,0,0,0,0)$, the orbit follows the super slow flow along the line $\ell^+_0$ for decreasing $u$ until reaching $(u,p) = (u_*,\mathcal{A}(u_*-1))$. Then within the subspace $(u,p)=(u_*,\mathcal{A}(u_*-1))$, the orbit traverses the slow manifold $\mathcal{E}_0$ from $s=0$ to $s=s_*$, followed by the fast heteroclinic orbit $\phi_\dagger(\zeta;u_*,s_*)$ to the equilibrium $(v,q,s) = (v_2(u_*),0,s_*)$ on the critical manifold $\mathcal{E}_1^+$; see Figure~\ref{fig:case(i)_plateau}. Next the superslow plateau trajectory $\ell_2(u_*)$ is traversed, which connects the points $(u,p) = (u_*,\pm \mathcal{A}(u_*-1))$ via the superslow flow on $\mathcal{C}_2$. Next, within the subspace $(u,p)=(u_*,-\mathcal{A}(u_*-1))$, the critical manifold $\mathcal{E}_1^+$ is traversed from $s=s_*$ down to $s=0$, followed by the fast heteroclinic orbit $\phi_\diamond(\zeta;u_*,0)$, which returns to the equilibrium $(v,q,s) = (0,0,0)$ on $\mathcal{E}_0$. Lastly the line $\ell^-_0$ is traversed for increasing $u$ from $(u,p) = (u_*,-\mathcal{A}(u_*-1))$ back to the equilibrium $(u,p,v,q,s)=(1,0,0,0,0)$. Taken together, the singular orbit is formed by the concatenation of superslow/slow/fast orbits given by
\begin{align}\label{case(i)_homoclinic_concatenation_slow_plateau}
\ell^+_0 \cup \mathcal{E}_0[0,s_*] \cup \phi_\dagger(\cdot;u_*,s_*) \cup \ell_2(u_*) \cup \mathcal{E}_1^+[0,s_*] \cup \phi_\diamond(\cdot;u_*,0)\cup \ell^-_0.
\end{align}
See Figures~\ref{fig:case(i)_plateau} and~\ref{fig:case(i)_superslow_reduced} for a visualization of the singular homoclinic orbit.\\

\subsection{Existence theorems}
The singular profiles for homoclinic orbits \eqref{case(i)_homoclinic_concatenation}/\eqref{case(i)_homoclinic_concatenation_slow_plateau}, both with and without slow plateau, are the key element in the following existence theorems, where the existence of associated homoclinic orbits in the full system \eqref{eq:ODEsystem_modified_rescaled}/\eqref{eq:ODEsystem_modified_fast_rescaled} is established for asymptotically small values of $\eps$ and $\delta$, in the relative scaling \eqref{eq:scaling_eps_delta_case(i)}.


\begin{thm}[Homoclinic orbits without superslow plateau]\label{case(i)_thm_noplateau} Take $c = c_\diamond(1,0)$ \eqref{eq:ODEsystem_case(i)_heteroclinics_wavespeeds}. Assume that
\begin{equation}\label{case(i)_thm_assumptions1}
    \frac{5}{9} < 4 k \mathcal{B} < \frac{8}{9},
\end{equation}
and that at least one of the following inequalities holds:
\begin{equation}\label{case(i)_thm_assumptions2}
    \mathcal{H}<2k,\quad \frac{1-4 k \mathcal{B}}{4k \mathcal{B}} < \frac{ \mathcal{H}}{2k-\mathcal{H}},\quad \frac{\mathcal{B} v_2(1)}{1-\mathcal{H} v_2(1)} > 1, 
\end{equation}
where $v_2(1)$ is the largest real solution to \eqref{eq:case(i)_cubicroots} for $u=1$. There exists $\delta_0 > 0$ such that for all $0<\delta<\delta_0$, there exists $\eps_0(\delta)$ such that for all $0<\eps<\eps_0(\delta)$, there exists a solution $\left(u_{h1},p_{h1},v_{h1},q_{h1},s_{h1}\right)$ to system \eqref{eq:ODEsystem_modified_fast_rescaled} that is homoclinic to the equilibrium $(1,0,0,0,0)$. To leading order in $\eps$ and $\delta$, this solution is given by the concatenation
\begin{displaymath}
\mathcal{E}_0[0,s_*] \cup \phi_\dagger(\zeta;1,s_*) \cup \mathcal{E}_1^+[0,s_*] \cup \phi_\diamond(\zeta;1,0)
\end{displaymath}
(cf. \eqref{case(i)_homoclinic_concatenation}), and its speed $c$ is given by
\begin{align}
    c = c_\diamond(1,0)+\mathcal{O}(|\eps|+|\delta|).
\end{align}

\end{thm}

\begin{proof}
We consider system \eqref{eq:ODEsystem_modified_rescaled_semireduced_2}, which is independent of $\eps$. For fixed $u=u_0$ and $p=p_0$, \eqref{eq:ODEsystem_modified_rescaled_semireduced_2} is a singularly perturbed system when $0<\delta\ll 1$. From Fenichel theory, it follows that there exists a $0<\delta_0$ such that for all $0<\delta<\delta_0$, the normally hyperbolic branches $\mathcal{E}_0$ and $\mathcal{E}_1^+$ of the one-dimensional critical manifold $\mathcal{E}$ \eqref{eq:case(i)_E0_Epm} persist as locally invariant manifolds. The two-dimensional stable and unstable manifolds of both branches, $\mathcal{W}^{u,s}(\mathcal{E}_0)$ and $\mathcal{W}^{u,s}(\mathcal{E}_1^+)$, persist as well.

Fix $u_0=1$ and $p_0 = 0$. For $\delta=0$, $\mathcal{W}^u(\mathcal{E}_0)$ and $\mathcal{W}^s(\mathcal{E}_1^+)$ intersect transversely along the heteroclinic orbit $\phi_\dagger(1,s_*)$ \eqref{eq:ODEsystem_case(i)_heteroclinics}, with $s_*$ given by \eqref{case(i)_solution_sstar}. Note that both $s_*$ and $v_2(1)$ are positive due to \eqref{case(i)_thm_assumptions1}. Likewise, $\mathcal{W}^u(\mathcal{E}_1^+)$ and $\mathcal{W}^s(\mathcal{E}_0)$ intersect transversely in the heteroclinic orbit $\phi_\diamond(1,0)$ \eqref{eq:ODEsystem_case(i)_heteroclinics}, since $c_\dagger(1,s_*) = c_\diamond(1,0)$ by the choice of $s_*$ \eqref{case(i)_solution_sstar}. We note that due to  Lemma~\ref{lem:case(i)_layer_transversality}, the heteroclinic connection $\phi_\dagger(1,s_*)$ breaks transversely upon varying $(u,s,c)$ near $(1,s_*,c_\dagger(1,s_*))$, while $\phi_\diamond(1,0)$ breaks transversely upon varying $(u,c)$ near $(1,c_\diamond(1,0))$. The absence of an equilibrium of \eqref{eq:ODEsystem_case(i)_slowreduced} on $\mathcal{E}_1^+$ by assumptions \eqref{case(i)_thm_assumptions2} ensures the existence of a singular homoclinic concatenation \eqref{case(i)_homoclinic_concatenation}. 

For all sufficiently small $\delta_0$, the two-dimensional unstable manifold $\mathcal{W}^\mathrm{u}(0,0,0)$ of the equilibrium $(v,q,s)=(0,0,0)$ of \eqref{eq:ODEsystem_modified_rescaled_semireduced_2} restricted to the subspace $\{u_0=1,p_0=0\}$ coincides with the perturbed manifold $\mathcal{W}^u(\mathcal{E}_0)$, while the one-dimensional stable manifold $\mathcal{W}^\mathrm{s}(0,0,0)$ lies on $\mathcal{W}^\mathrm{s}(\mathcal{E}_0)$, $\mathcal{O}(\delta)$-close to the singular heteroclinic connection $\phi_\diamond(1,0)$. For values of $c\approx c_\dagger(1,s_*)$ and sufficiently small $\delta>0$, we track the manifold $\mathcal{W}^\mathrm{u}(0,0,0)$ along the heteroclinic orbit $\phi_\dagger(1,s_*)$; due to the transverse intersection of $\mathcal{W}^u(\mathcal{E}_0)$ and $\mathcal{W}^s(\mathcal{E}_1^+)$ for $\delta=0$, we have that $\mathcal{W}^\mathrm{u}(0,0,0)$ transversely intersects $\mathcal{W}^s(\mathcal{E}_1^+)$ and, by the exchange lemma~\cite{Schecter.2008}, aligns exponentially close (in $\delta$) to $\mathcal{W}^u(\mathcal{E}_1^+)$ upon exiting a neighborhood of $\mathcal{E}_1^+$ near $s=0$. Using the transversality of the heteroclinic connection $\phi_\diamond(1,0)$ with respect to the wave speed $c$, there exists
\begin{align}
   c= c_0(\delta)=c_\diamond(1,0)+\mathcal{O}(\delta) = c_\dagger(1,s_*)+\mathcal{O}(\delta)
\end{align}
such that $\mathcal{W}^\mathrm{u}(0,0,0)$ intersects $\mathcal{W}^\mathrm{s}(0,0,0)$ upon exiting a neighborhood of $\mathcal{E}_1^+$, corresponding to a homoclinic orbit $(v_\mathrm{h,0},q_\mathrm{h,0},s_\mathrm{h,0})$ of  \eqref{eq:ODEsystem_modified_rescaled_semireduced_2} in the subspace $\{u=1,p=0\}$ which is $\mathcal{O}(\delta)$-close to the singular concatenation \eqref{case(i)_homoclinic_concatenation}. Further, this homoclinic orbit breaks transversely upon varying the speed $c$ near $c_0(\delta)$.

Now, fix $0<\delta<\delta_0$ and consider \eqref{eq:ODEsystem_case(i)}, which has $P_0 := (1,0,0,0,0)$ as a hyperbolic equilibrium which, for $\eps=0$, lies on the two-dimensional normally hyperbolic critical manifold $\mathcal{C}_0$. The manifold $\mathcal{C}_0$ persists for sufficiently small $\eps>0$, as does its three-dimensional stable manifold $\mathcal{W}^\mathrm{s}(\mathcal{C}_0)$ and four-dimensional unstable manifold $\mathcal{W}^\mathrm{u}(\mathcal{C}_0)$. In the limit $\eps \to0$, the unstable manifold $\mathcal{W}^\mathrm{u}(P_0)$ of $P_0$ corresponds to the three dimensional manifold obtained by taking the union
\begin{align}
   \mathcal{W}^\mathrm{u}(P_0) = \bigcup_{(u,p)\in \ell_0^+} \mathcal{W}_0^u\left((0,0,0)\right)
\end{align}
of the two dimensional manifolds $\mathcal{W}_0^u\left((0,0,0)\right)$ of the origin in $(v,q,s)$-space in~\eqref{eq:ODEsystem_modified_rescaled_semireduced_2} for $(u,p)\in \ell_0^+\subseteq \mathcal{C}_0$. Similarly the the stable manifold $\mathcal{W}^\mathrm{s}(P_0)$ of $P_0$ corresponds to the two dimensional manifold obtained by taking the union
\begin{align}
   \mathcal{W}^\mathrm{s}(P_0) = \bigcup_{(u,p)\in \ell_0^-} \mathcal{W}_0^s\left((0,0,0)\right)
\end{align}
of the one-dimensional stable manifolds $\mathcal{W}_0^s\left((0,0,0)\right)$ for $(u,p)\in \ell_0^-\subseteq \mathcal{C}_0$.

When $\eps = 0$, the manifolds $\mathcal{W}^s(P_0)$ and $\mathcal{W}^u(P_0)$ intersect in the subspace $\{u=1,p=0\}$ for $c=c_0(\delta)$, along the intersection of $\mathcal{W}_0^s\left((0,0,0)\right)$ with $\mathcal{W}_0^u\left((0,0,0)\right)$, corresponding to the homoclinic orbit $(v_\mathrm{h,0},q_\mathrm{h,0},s_\mathrm{h,0})$ constructed above. From Fenichel theory, it follows that there exists an $\eps_0(\delta)>0$ such that for all $0<\eps<\eps_0(\delta)$, both $\mathcal{W}^s(P_0)$ and $\mathcal{W}^u(P_0)$ persist as locally invariant manifolds. Due to the transversality of the homoclinic orbit $(v_\mathrm{h,0},q_\mathrm{h,0},s_\mathrm{h,0})$ with respect to varying $c$, by adjusting $c= c_0(\delta)+\mathcal{O}(\eps)$, we can ensure that $\mathcal{W}^s(P_0)$ and $\mathcal{W}^u(P_0)$ intersect along homoclinic orbit of the full system  \eqref{eq:ODEsystem_case(i)}, which is to leading order given by the homoclinic concatenation \eqref{case(i)_homoclinic_concatenation}.
\end{proof}

\begin{thm}[Homoclinic orbit with superslow plateau]\label{case(i)_thm_plateau} Assume that
\begin{equation}
    \mathcal{H}<\frac{8}{9} k\quad\text{and}\quad 1-4 k \mathcal{B}>0.
\end{equation}
Let $k, \mathcal{B}$ be such that the unique solution $U_*$ to \eqref{eq:case(i)_cubic_Ustar} satisfies 
\begin{equation}
\sqrt{\frac{\mathcal{H}}{2k}} < U_* < \sqrt{1-4 k \mathcal{B}}.
\end{equation}
Let $u_* = \frac{4 k \mathcal{B}}{1-U_*^2}$, fix $c = c_\diamond(u_*,0)$ \eqref{eq:ODEsystem_case(i)_heteroclinics_wavespeeds}, and let
\begin{equation}\label{case(i)_plateau_value_sstar}
    s_* = \frac{\mathcal{B}\left(5-3\sqrt{1-\frac{4k\mathcal{B}}{u_*}}\right)}{6k-\mathcal{H}\left(5-3\sqrt{1-\frac{4k\mathcal{B}}{u_*}}\right)}.
\end{equation}
There exists a $\delta_0 > 0$ such that for all $0<\delta<\delta_0$, there exists a $\eps_0(\delta)$ such that for all $0<\eps<\eps_0(\delta)$, there exists a solution $\left(u_{h2},p_{h2},v_{h2},q_{h2},s_{h2}\right)$ to system \eqref{eq:ODEsystem_modified_fast_rescaled} that is homoclinic to the equilibrium $(1,0,0,0,0)$. To leading order in $\eps$ and $\delta$, this solution is given by the concatenation 
\begin{displaymath}
    \ell^+_0 \cup \mathcal{E}_0[0,s_*] \cup \phi_\dagger(\cdot;u_*,s_*) \cup \ell_2(u_*) \cup \mathcal{E}_1^+[0,s_*] \cup \phi_\diamond(\cdot;u_*,0)\cup \ell^-_0
\end{displaymath}
(cf. \eqref{case(i)_homoclinic_concatenation_slow_plateau}), and its speed is given by
\begin{align}
    c =  c_\diamond(u_*,0)+\mathcal{O}(|\eps|+|\delta|)
\end{align}
\end{thm}

\begin{proof}
We consider system \eqref{eq:ODEsystem_modified_rescaled_semireduced_2}, which is independent of $\eps$. For fixed $u=u_0$ and $p=p_0$, \eqref{eq:ODEsystem_modified_rescaled_semireduced_2} is a singularly perturbed system when $0<\delta\ll 1$. From Fenichel theory, it follows that there exists a $0<\delta_0$ such that for all $0<\delta<\delta_0$, the normally hyperbolic branches $\mathcal{E}_0$ and $\mathcal{E}_1^+$ of the one-dimensional critical manifold $\mathcal{E}$ \eqref{eq:case(i)_E0_Epm} persist as locally invariant manifolds. The two-dimensional stable and unstable manifolds of both branches, $\mathcal{W}^{u,s}(\mathcal{E}_0)$ and $\mathcal{W}^{u,s}(\mathcal{E}_1^+)$, persist as well.\\
Fix $u_0=u_*$ and $p_0 = p_* = \sqrt{\mathcal{A}}(u_*-1)$ such that $(u_*,p_*) \in \ell_0^+$ \eqref{case(i)_superslow_ell0pm}. For $\delta=0$, $\mathcal{W}^u(\mathcal{E}_0)$ and $\mathcal{W}^s(\mathcal{E}_1^+)$ intersect transversely in the heteroclinic orbit $\phi_\dagger(u_*,s_*)$ \eqref{eq:ODEsystem_case(i)_heteroclinics}, with $s_*$ given by \eqref{case(i)_plateau_value_sstar}.
The concatenation $\mathcal{E}_0[0,s_*] \cup \phi_\dagger(\cdot;u_*,s_*)$ provides the intersection of the two-dimensional unstable manifold 
of the origin $\mathcal{W}_0^u\left((0,0,0)\right)$ with $\mathcal{W}_0^s \left( (v_2(u_*),0,s_2(u_*))\right)$, the one-dimensional stable manifold of the unique equilibrium on $\mathcal{E}_1^+$. We recall from Lemma~\ref{lem:case(i)_layer_transversality} that the heteroclinic connection $\phi_\dagger(\cdot;u_*,s_*)$ breaks transversely upon varying $u \approx u_*$ or upon varying the speed $c\approx c_\dagger(u_*,s_*)$. Thus for all sufficiently small $\delta>0$, we can ensure the intersection of $\mathcal{W}_0^u\left((0,0,0)\right)$ with $\mathcal{W}_0^s \left( (v_2(u_*),0,s_2(u_*))\right)$ persists, by varying $u$ (resp. $c$) near $u=u_*$ (resp. $c=c_\dagger(u_*,s_*)$).

Likewise, fix $u_0=u_*$ and $p_0 = -p_*$ such that $(u_*,p_*) \in \ell_0^-$ \eqref{case(i)_superslow_ell0pm}, $\mathcal{W}^u(\mathcal{E}_1^+)$ and $\mathcal{W}^s(\mathcal{E}_0)$ intersect transversely in the heteroclinic orbit $\phi_\diamond(u_*,0)$ \eqref{eq:ODEsystem_case(i)_heteroclinics}, since $c = c_\dagger(u_*,s_*) = c_\diamond(u_*,0)$ by the choice of $s_*$ \eqref{case(i)_plateau_value_sstar}. Since $s_2(u_*)$ is the only equilibrium of \eqref{eq:ODEsystem_case(i)_slowreduced} on $\mathcal{E}_1^+$, the concatenation $\mathcal{E}_1^+[0,s_*] \cup \phi_\diamond(\cdot;u_*,0)$ provides the intersection of $\mathcal{W}_0^u\left((v_2(u_*),0,s_2(u_*))\right)$, the two dimensional unstable manifold of the unique equilibrium on $\mathcal{E}_1^+$, with the one-dimensional stable manifold of the origin $\mathcal{W}_0^s\left((0,0,0)\right)$. By a similar argument as above, for all sufficiently small $\delta>0$, we can ensure the intersection of $\mathcal{W}_0^s\left((0,0,0)\right)$ with $\mathcal{W}_0^u \left( (v_2(u_*),0,s_2(u_*))\right)$ persists, by varying $u$ (resp. $c$) near $u=u_*$ (resp. $c=c_\diamond(u_*,0)$). In the same manner, this intersection breaks transversely upon varying $(u,c)$.

We can now lift these intersections into the geometry of the full five-dimensional system~\eqref{eq:ODEsystem_case(i)} when $\eps=0$. We recall that this system admits the two-dimensional critical manifold $\mathcal{C}$ \eqref{eq:case(1)_superslow_C} which has two normally hyperbolic branches, $\mathcal{C}_0$ and $\mathcal{C}_2$, and the equilibrium $P_0 := (1,0,0,0,0)$ lies on $\mathcal{C}_0$. These normally hyperbolic branches persist for sufficiently small $\eps>0$, as do their three-dimensional stable manifolds $\mathcal{W}^\mathrm{s}(\mathcal{C}_0), \mathcal{W}^\mathrm{s}(\mathcal{C}_2)$ and four-dimensional unstable manifolds $\mathcal{W}^\mathrm{u}(\mathcal{C}_0), \mathcal{W}^\mathrm{u}(\mathcal{C}_2)$. In the limit $\eps \to0$ the unstable manifold $\mathcal{W}^\mathrm{u}(P_0)$ corresponds to the three dimensional manifold obtained by taking the union
\begin{align}
   \mathcal{W}^\mathrm{u}(P_0) = \bigcup_{(u,p)\in \ell_0^+} \mathcal{W}_0^u\left((0,0,0)\right)
\end{align}
of the two dimensional manifolds $\mathcal{W}_0^u\left((0,0,0)\right)$ of the origin of $(v,q,s)$-space in~\eqref{eq:ODEsystem_modified_rescaled_semireduced_2} for $(u,p)\in \ell_0^+\subseteq \mathcal{C}_0$. Similarly the the stable manifold $\mathcal{W}^\mathrm{s}(P_0)$ of $P_0$ corresponds to the two dimensional manifold obtained by taking the union
\begin{align}
   \mathcal{W}^\mathrm{s}(P_0) = \bigcup_{(u,p)\in \ell_0^-} \mathcal{W}_0^s\left((0,0,0)\right)
\end{align}
of the one-dimensional stable manifolds $\mathcal{W}_0^s\left((0,0,0)\right)$ for $(u,p)\in \ell_0^-\subseteq \mathcal{C}_0$. By the argument above, when $c=c_\dagger(u_*,s_*)=c_\diamond(u_*,0)$ and $\delta=0$, $\mathcal{W}^\mathrm{s}(P_0)$ transversely intersects the four dimensional manifold $\mathcal{W}^u(\mathcal{C}_2)$ in the hyperplane $\left\{ u = u_*,\; p = -p_* \right\}$ along the intersection of $\mathcal{W}_0^s\left((0,0,0)\right)$ with $\mathcal{W}_0^u\left((v_2(u_*),0,s_2(u_*))\right)$, and this intersection persists for all sufficiently small $\delta>0$ and values of $c\approx c_\diamond(u_*,0)$.  Likewise, $\mathcal{W}^\mathrm{u}(P_0)$ transversely intersects the three dimensional manifold $\mathcal{W}^s(\mathcal{C}_2)$, in the hyperplane $\left\{ u = u_*,\; p = p_* \right\}$ along the intersection of $\mathcal{W}_0^u\left((0,0,0)\right)$ with $\mathcal{W}_0^s\left((v_2(u_*),0,s_2(u_*))\right)$, and this transverse intersection persists for all sufficiently small $\delta>0$ and values of $c\approx c_\dagger(u_*,s_*)$.

Now, fix $0<\delta<\delta_0$ and $c$ near $c_\dagger(u_*,s_*)$. There exists $\eps_0(\delta)$ such that for all sufficiently small $0<\eps<\eps_0(\delta)$, the unstable manifold $\mathcal{W}^\mathrm{u}(P_0)$ therefore transversely intersects $\mathcal{W}^s(\mathcal{C}_2)$ at a value of $\tilde{u}(\delta, \eps c)=u_*+\mathcal{O}(\delta+\eps +|c-c_\dagger(u_*,s_*)|)$ along a stable fiber of the perturbed trajectory $\ell_2(\tilde{u})$ on $\mathcal{C}_2$. By the exchange lemma~\cite{Schecter.2008}, $\mathcal{W}^\mathrm{u}(P_0)$ aligns exponentially close (in $\eps)$) along the three-dimensional manifold of unstable fibers $\mathcal{W}^\mathrm{u}(\ell_2(\tilde{u}))\subseteq \mathcal{W}^\mathrm{u}(\mathcal{C}_2)$ of the trajectory $\ell_2(\tilde{u})\subseteq \mathcal{C}_2$ upon exiting a small neighborhood of $\mathcal{C}_2$. We note that when $\delta=\eps=0$ and $u=u_*$, the projection of $\ell_2(u_*)$ onto $\mathcal{C}_0$ is orthogonal, and transversely intersects the stable and unstable manifolds $\ell_0^\pm$ \eqref{case(i)_superslow_ell0pm} of \eqref{eq:ODEsystem_case(i)_superslow_C0} on $\mathcal{C}_0$. Thus this also holds for values of $u$ near $u_*$, and all sufficiently small $\delta>0$, $0<\eps<\eps_0(\delta)$. In particular, the projection of $\ell_2(\tilde{u})$ onto $\mathcal{C}_0$ transversely intersects $\ell_0^\pm$ for all sufficiently small $|c-c_\dagger(u_*,s_*)|$ and $\delta>0$, and $0<\eps<\eps_0(\delta)$. 

Using the fact that $\mathcal{W}^\mathrm{s}(P_0)$ transversely intersects $\mathcal{W}^u(\mathcal{C}_2)$ when $\eps=0$ along the intersection of $\mathcal{W}_0^s\left((0,0,0)\right)$ with $\mathcal{W}_0^u \left( (v_2(u_*),0,s_2(u_*))\right)$, and the fact that this intersection breaks transversely upon varying the wave speed $c\approx c_\diamond(u_*,0)$, it remains to adjust
\begin{align}
    c=c_\diamond(u_*,0)+\mathcal{O}(\delta, \eps) = c_\dagger(u_*,s_*)+\mathcal{O}(\delta, \eps)
\end{align}
so that $\mathcal{W}^\mathrm{u}(P_0)$ intersects $\mathcal{W}^\mathrm{s}(P_0)$ upon exiting a neighborhood of $\mathcal{C}_2$ corresponding to a homoclinic orbit to the equilibrium $P_0$ of the full system~\eqref{eq:ODEsystem_case(i)} for all sufficiently small $0<\eps<\eps_0(\delta)$, which is approximated by the singular homoclinic orbit \eqref{case(i)_homoclinic_concatenation_slow_plateau}. 

\end{proof}

\begin{rmk}\label{rmk:case(i)_nocoexistence}
Comparison of the existence conditions of homoclinics with (Theorem \ref{case(i)_thm_plateau}) and without (Theorem \ref{case(i)_thm_noplateau}) superslow plateau reveals that these pulse types cannot coexist. For the homoclinic without superslow plateau, the jump value $s_*$ \eqref{case(i)_solution_sstar} needs to be below the $s$-equilibrium on the branch $\mathcal{E}_1^+$ in order for the flow on $\mathcal{E}_1^+$ to be monotonic. However, $s_*$ is equal to $s_2(u_*=1)$ \eqref{eq:case(i)_s_2_ustar}, and $s_2(u_*)$ can be seen to be strictly increasing in $u_*$. This leads to a contradiction: if a pulse with superslow plateau exists, then $s_2(u_*) < s_*$ (since $u_* < 1$), which prohibits the existence of a pulse without superslow plateau. Vice versa, if a pulse without superslow plateau exists, then $s_* < s_2(u_*)$, which implies $u_* > 1$; this prohibits the existence of a pulse with superslow plateau.
\end{rmk}

\section{Case (ii): superslow $s$} \label{sec:caseii}
Case (ii) also presents a three-tier scaling hierarchy, where $s$ is the slowest component, since
\begin{equation}\label{eq:scaling_case(ii)}
    0 < \delta := \frac{1}{c \mathcal{D}} \ll \eps \ll 1.
\end{equation}
We investigate the geometry and dynamics imposed by this scaling hierarchy. As in case (i) \eqref{eq:scaling_eps_delta_case(i)}, studied in Section \ref{sec:casei}, the scaling hiearchy \eqref{eq:scaling_case(ii)} also allows for the construction of two types of homoclinic pulses; in the language of Section \ref{sec:casei}, both with and without `superslow plateau'. However, the counterpart of the pulse \emph{without} superslow plateau (cf. Theorem \ref{case(i)_thm_noplateau}) would be a pulse solution where the $s$-component vanishes identically. As our aim is to investigate the influence of autotoxicity on travelling pulses in a Klausmeier type model, which implies nonvanishing $s$, we forego analysis of this solution type. For the construction of fronts and double fronts in systems with a similar geometry, we refer to \cite{JDC-BM.2020,Carter.2023}.\\
The remaining pulse type, to be studied in this section, is therefore \emph{with} superslow plateau, see Theorem \ref{case(ii)_thm}. Its singular concatenation \eqref{case(ii)_homoclinic_concatenation} consists of superslow $(s)$, intermediate $(u,p)$ and fast $(v,q)$ orbit segments, see also Figure \ref{fig:case(ii)_full_double_heteroclinic}. This singular construction therefore also (cf. Theorem \ref{case(i)_thm_plateau}) yields two matching conditions, for both $u$ and $s$, determining the height of the $s$-plateau and the width of the $v$-plateau.

\subsection{Superslow dynamics}
The scaling \eqref{eq:scaling_case(ii)} implies
\begin{equation}
    0< \frac{\delta}{\eps} \ll 1.
\end{equation}
Introducing the superslow coordinate $y =\frac{\delta}{\eps} z$, we rewrite \eqref{eq:ODEsystem_modified_rescaled} as
\begin{subequations}\label{eq:ODEsystem_case(ii)}
 \begin{align}
 \delta u_y &= \eps p,\\
 \delta p_y &= \eps\left[u v^2-\mathcal{A}\left(1-u\right) - \eps c p\right],\\
  \delta v_y &= q,\\
  \delta q_y &= \mathcal{B} v  - u v^2(1-k v) + \mathcal{H} v s - c q,\\
 s_y &= - \mathcal{B} v- \mathcal{H} v s + s. 
 \end{align}
\end{subequations}
Note that this is the same system as \eqref{eq:ODEsystem_case(i)}; the only difference is that the asymptotic order of $\delta$ and $\eps$ is reversed. Hence, we first fix $\eps>0$ and take the limit $\delta \downarrow 0$. The resulting reduced superslow system yields four algebraic equations. These define the one-dimensional manifolds
\begin{equation}\label{eq:case(ii)_N0}
    \mathcal{N}_0 := \left\{u=1,p=0,v=0,q=0\right\}
\end{equation}
and
\begin{equation}\label{eq:case(ii)_N1}
    \mathcal{N}_1 := \left\{u v^2 - \mathcal{A}(1-u) = 0,\;p=0,\;\mathcal{B}+\mathcal{H} s - u v(1-k v) = 0,\;q=0\right\}.
\end{equation}
On the line $\mathcal{N}_0$, the $s$-dynamics reduce to
\begin{equation}
    s_y = s.
\end{equation}
The curve $\mathcal{N}_1$ can globally be written as a graph over $v$, with
\begin{equation}\label{eq:case(ii)_N1_graph_us}
    u(v) = \frac{\mathcal{A}}{\mathcal{A}+v^2},\quad s(v) = \frac{1}{\mathcal{H}} \left[-\mathcal{B} + \frac{\mathcal{A} v(1-k v)}{\mathcal{A}+v^2}\right],
\end{equation}
both bounded in $v$. We note that $s(0) = -\frac{\mathcal{B}}{\mathcal{H}}< 0$ and $\lim_{v \to \infty} s(v) = - \frac{\mathcal{A} k + \mathcal{B}}{\mathcal{H}} < 0$; the graph of $s$ over $v$ is positive for $\hat{v}_- < v < \hat{v}_+$, with $\hat{v}_\pm = \frac{1}{2}\frac{\mathcal{A}}{\mathcal{A} k + \mathcal{B}} \left(1 \pm \sqrt{1 - 4 \mathcal{B} \frac{\mathcal{A} k + \mathcal{B}}{\mathcal{A}} }\right)$, which occurs if and only if $0<\mathcal{B}<\frac{1}{4k}$ and $\mathcal{A} > \frac{4 \mathcal{B}^2}{1-4 k\mathcal{B}}$. The graph of $s$ over $v$ has a maximum at $v_\text{max} = \mathcal{A} k \left(-1 + \sqrt{1 + \frac{1}{\mathcal{A} k^2}}\right)$, where $s(v_\text{max}) =: s_\text{max} = \frac{-\mathcal{B} + v_\text{max}/2}{\mathcal{H}}$ and $ u(v_\text{max}) =: u_\text{max} $.\\ Equilibria on $\mathcal{N}_1$ are found by intersection with the nullcline $s = \frac{\mathcal{B} v}{1-\mathcal{H} v}$; there are at most two equilibria in the positive $(v,s)$-quadrant. To effectively describe the $s$-dynamics on $\mathcal{N}_1$, we write
\begin{equation}
 \mathcal{N}_1 = \mathcal{N}_1^+ \cup \mathcal{N}_1^-,
\end{equation}
with the two branches
\begin{subequations}\label{eq:case(ii)_branches_Npm}
\begin{align}
    \mathcal{N}_1^- &:= \left\{(u,p,v,q,s)\in\mathcal{N}_1,\, 0<v<v_\text{max}\right\}\\
    \mathcal{N}_1^+ &:= \left\{(u,p,v,q,s)\in\mathcal{N}_1,\, v_\text{max}<v\right\}
\end{align}
\end{subequations}
meeting at the fold point $\mathcal{O} = \left(u_\text{max},0,v_\text{max},0,s_\text{max}\right)$. We infer that $\mathcal{N}_1^+$ contains at most one equilibrium in the positive $(v,s)$-quadrant, which is unstable in the $s$-dynamics when it exists; this equilibrium lies on $\mathcal{N}_1^+$ if and only if $s_\text{max} > \frac{\mathcal{B} v_\text{max}}{1-\mathcal{H}v_\text{max}}$. This equilibrium moves through the fold point $\mathcal{O}$ to the branch $\mathcal{N}_1^-$ when $s_\text{max} = \frac{\mathcal{B} v_\text{max}}{1-\mathcal{H}v_\text{max}}$. Equating $s=\frac{\mathcal{B} v}{1-\mathcal{H} v}$ with $s(v)$ \eqref{eq:case(ii)_N1_graph_us} and subsequently eliminating $v$ in favour of $s$, we find that the $s$-coordinate of this equilibrium is the largest real solution to the cubic equation
\begin{equation}\label{eq:case(ii)_s_equilibrium_N1}
    \mathcal{H}(1+\mathcal{A}\mathcal{H}^2) s^3 + \left(\mathcal{B}+3 \mathcal{A}\mathcal{B}\mathcal{H}^2 + \mathcal{A}(k - \mathcal{H})\right)s^2 + \mathcal{A} \mathcal{B}(3 \mathcal{B}\mathcal{H}-1)s + \mathcal{A} \mathcal{B}^3 = 0.
\end{equation}
Alternatively, we can eliminate $s$ in favour of $v$, to obtain
\begin{equation}\label{eq:case(ii)_v_equilibrium_N1}
   \mathcal{A} \mathcal{H} k v^3 - \left(\mathcal{B} + \mathcal{A}(\mathcal{H}+k)\right) v^2 + \mathcal{A} v - \mathcal{A} \mathcal{B}= 0.
\end{equation}
As the mapping $s \mapsto s(v) = \frac{\mathcal{B} v}{1-\mathcal{H} v}$ is strictly monotonically increasing, the $v$-coordinate of this equilibrium is the largest real solution to \eqref{eq:case(ii)_v_equilibrium_N1}.\\

For a sketch of $\mathcal{N}_1$ and the superslow dynamics on $\mathcal{N}_1$, see Figure \ref{fig:case(ii)_superslow_s_N1}.\\

\begin{figure}
    \centering
    \includegraphics[width=0.6\textwidth]{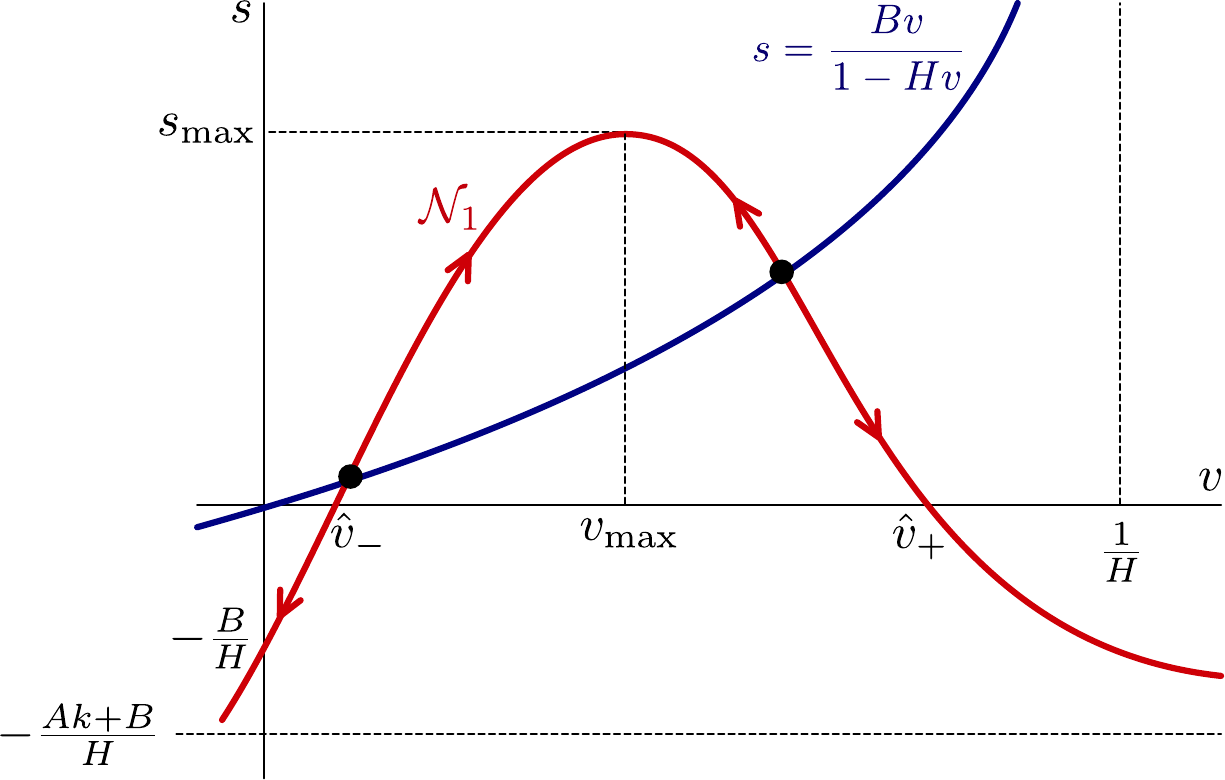}
    \caption{The curve $\mathcal{N}_1$ projected onto the $(v,s)$-plane, with the superslow dynamics on $\mathcal{N}_1$.}
    \label{fig:case(ii)_superslow_s_N1}
\end{figure}

To facilitate the upcoming analysis of the slow and fast timescales, we fix $\eps >0$ and take the limit $\delta \downarrow 0$ in \eqref{eq:ODEsystem_modified_rescaled} to obtain the so-called semi-reduced system 
\begin{subequations}\label{eq:ODEsystem_modified_rescaled_semireduced}
 \begin{align}
  u_z &= p,\\
  p_z &= u v^2-\mathcal{A}\left(1-u\right) - \eps c p,\\
  \eps v_z &= q,\\
  \eps q_z &= \mathcal{B} v  - u v^2(1-k v) + \mathcal{H} v s - c q,\\
  s &= s_0.
 \end{align}
\end{subequations}
The semi-reduced system will be studied for $0<\eps \ll 1$ in the upcoming subsections. For future reference, we define
\begin{equation}\label{eq:case(ii)_N1_proj}
    \Pi_{s_0} \mathcal{N}_1 := \left\{u v^2 - \mathcal{A}(1-u) = 0,\;p=0,\;q=0,\;s=s_0\right\}
\end{equation}
to be the projection of the curve $\mathcal{N}_1$ \eqref{eq:case(ii)_N1} onto the $s=s_0$ hyperplane, which is the invariant subspace on which the dynamics of \eqref{eq:ODEsystem_modified_rescaled_semireduced} take place.

\subsection{Slow dynamics}
The reduced slow system, obtained by taking $\eps \to 0$ in \eqref{eq:ODEsystem_modified_rescaled_semireduced}, is given by
\begin{equation}
    u_{zz} = u v^2-\mathcal{A}\left(1-u\right),
\end{equation}
and is restricted to the two-dimensional manifolds
\begin{equation}\label{eq:case(ii)_M0}
    \mathcal{M}_0 := \left\{ v=0,\; q=0,\; s=s_0\right\}
\end{equation}
and
\begin{equation}\label{eq:case(ii)_M1}
    \mathcal{M}_1 := \left\{ \mathcal{B}+\mathcal{H} s - u v(1-k v) = 0,\; q=0,\; s=s_0\right\}.
\end{equation}
Solving the equation $\mathcal{B}+\mathcal{H} s - u v(1-k v) = 0$ for $v$ yields
\begin{equation}\label{eq:case(ii)_vpm}
    v_\pm(u) = \frac{1}{2k}\left(1 \pm \sqrt{1-\frac{4 k\left(\mathcal{B} + \mathcal{H} s_0\right)}{u}}\right).
\end{equation}
Note that $\mathcal{M}_1$ is folded along the curve $\left\{u= 4 k(\mathcal{B} + \mathcal{H} s_0),\; v = \frac{1}{2k},q=0,s=s_0\right\} \subset \mathcal{M}_1$. For future reference, we define the upper and lower branches 
\begin{subequations}\label{eq:case(ii)_M1_branches}
\begin{align}
\mathcal{M}_1^+ &:= \left\{(u,p,v,q,s)\in \mathcal{M}_1\,:\, v>\frac{1}{2k}\right\}\\
\mathcal{M}_1^- &:= \left\{(u,p,v,q,s)\in \mathcal{M}_1\,:\, v<\frac{1}{2k}\right\}.
\end{align}
\end{subequations}
The critical manifold $\mathcal{M}$ of \eqref{eq:ODEsystem_modified_rescaled_semireduced} is given by the union
\begin{equation}\label{eq:case(ii)_M}
    \mathcal{M} := \mathcal{M}_0 \cup \mathcal{M}_1^+ \cup \mathcal{M}_1^-.
\end{equation}
The line $\mathcal{N}_0$ \eqref{eq:case(ii)_N0} intersects the two-dimensional hyperplane $\mathcal{M}_0$ at the equilibrium of the slow system on $\mathcal{M}_0$,
\begin{equation}\label{Eq:case(ii)_M0_dynamics}
    u_{zz} = -\mathcal{A}(1-u);
\end{equation}
see also Figure \ref{fig:case(ii)_M0_N0_ups}. Note that both $\mathcal{M}_0, \mathcal{N}_0 \subset \left\{ v=0,\;q=0\right\}$; inside this three-dimensional hyperplane, the intersection of $\mathcal{M}_0$ and $\mathcal{N}_0$ is transversal. System \eqref{Eq:case(ii)_M0_dynamics} has one saddle equilibrium at $u=1$, $p=0$; the stable and unstable manifolds of this equilibrium are given by the lines
\begin{equation}
     \ell_0^s = \left\{p = -\sqrt{\mathcal{A}}(u-1)\right\},\quad \ell_0^u = \left\{p = \sqrt{\mathcal{A}}(u-1)\right\}.
 \end{equation}

\begin{figure}
    \centering
    \includegraphics[width=0.5\textwidth]{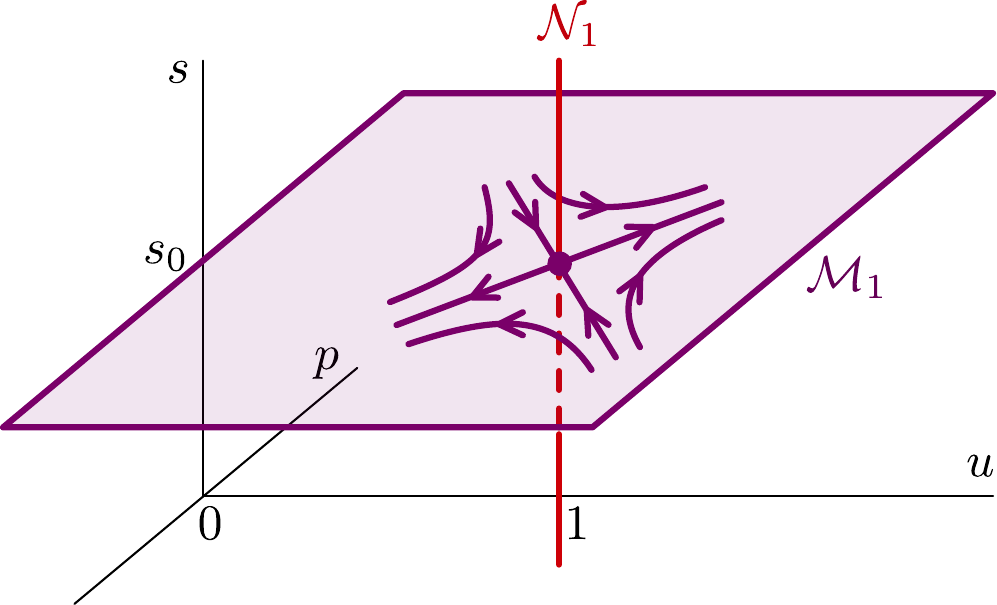}
    \caption{The manifold $\mathcal{M}_0$ \eqref{eq:case(ii)_M0} and the line $\mathcal{N}_0$ \eqref{eq:case(ii)_N1} in $(u,p,s)$-space. The linear dynamics \eqref{Eq:case(ii)_M0_dynamics} on $\mathcal{M}_0$, with a saddle equilibrium at $(1,0,s_0) \in \mathcal{M}_0$, are indicated in purple.}
    \label{fig:case(ii)_M0_N0_ups}
\end{figure}

Equivalently, the curve $\Pi_{s_0} \mathcal{N}_1$ \eqref{eq:case(ii)_N1_proj} intersects the two-dimensional manifold $\mathcal{M}_1$ at the equilibria of the slow system on $\mathcal{M}_1$,
\begin{equation}\label{eq:case(ii)_M1_dynamics}
    u_{zz} = u\, v_{\pm}(u)^2 - \mathcal{A}(1-u),
\end{equation}
with $v_\pm$ \eqref{eq:case(ii)_vpm} corresponding to the upper and lower branches of $\mathcal{M}_1$ \eqref{eq:case(ii)_M1_branches}; see also Figure \ref{fig:case(ii)_M1_N1_upv}. As the definition of $v_\pm$ is the same as in Section \ref{ssec:case(i)_slowdynamics}, equation \eqref{eq:case(i)_vpm}, we use the same notation; note that in \eqref{eq:case(i)_vpm}, $u=u_0$ is fixed and $s$ varies, whereas in the current setting, $s=s_0$ is fixed and $u$ varies.\\

For future reference, we define the $u$-coordinate of the intersection of $\Pi_{s_0} \mathcal{N}_1$ and $\mathcal{M}_1^+$ as $u_1^+$, with
\begin{equation}\label{eq:case(ii)_equilibrium_u1+}
 u_1^+:= \frac{1 + 2 k(\mathcal{B} + \mathcal{H} s_0 + \mathcal{A} k) - \sqrt{1-4\frac{\mathcal{B} + \mathcal{H} s_0}{\mathcal{A}}\left(\mathcal{B} + \mathcal{H} s_0 + \mathcal{A} k\right)}}{2\left(1+\mathcal{A} k^2\right)},
\end{equation}
so that
\begin{equation}
    \Pi_{s_0}\mathcal{N}_1 \cap \mathcal{M}_1^+ = \left\{ u=u_1^+,\,p=0,\,v=v^+(u_1^+),\,q=0,\,s=s_0\right\},
\end{equation}
provided that the intersection is nonempty, which is true whenever $\mathcal{A} > \frac{1}{k}\frac{\mathcal{B}+\mathcal{H} s_0}{1-4 k \left(\mathcal{B}+\mathcal{H} s_0\right)}$. Note that $u_1^+ < 1$ because the graph of $u$ over $v$ \eqref{eq:case(ii)_N1_graph_us} is strictly monotone; see again Figure \ref{fig:case(ii)_M1_N1_upv}. \\

From \eqref{eq:case(ii)_equilibrium_u1+}, we infer that $\Pi_{s_0} \mathcal{N}_1$ intersects $\mathcal{M}_1$ if and only if
\begin{equation}
    0<4 k \left(\mathcal{B}+\mathcal{H} s_0\right) < 1\quad\text{and}\quad \mathcal{A} > \frac{4\left(\mathcal{B}+\mathcal{H} s_0\right)^2}{1-4 k \left(\mathcal{B}+\mathcal{H} s_0\right)}.
\end{equation}
Clearly, $\Pi_{s_0} \mathcal{N}_1$ intersects $\mathcal{M}_1$ at most twice, and at most one intersection lies on the branch $\mathcal{M}_1^+$. Both intersections lie on the branch $\mathcal{M}_1^-$ if and only if $\mathcal{A} < \frac{1}{k}\frac{\mathcal{B}+\mathcal{H} s_0}{1-4 k \left(\mathcal{B}+\mathcal{H} s_0\right)}$. If $\Pi_{s_0} \mathcal{N}_1$ intersects each branch exactly once, as depicted in Figure \ref{fig:case(ii)_M1_N1_upv}, then the associated equilibria of \eqref{eq:case(ii)_M1_dynamics} are saddles; if $\Pi_{s_0} \mathcal{N}_1$ intersects the lower branch twice, then the equilibrium furthest away from the fold line is a saddle, and the one closest to the fold line therefore a center. This follows directly from the observation that $u\, v_{\pm}(u)^2 - \mathcal{A}(1-u)$ is strictly positive for sufficiently large $u$; hence, at the rightmost intersection of the graph of $u\, v_{\pm}(u)^2 - \mathcal{A}(1-u)$ with the horizontal axis the graph is increasing, which is equivalent to the associated (rightmost) equilibrium of \eqref{eq:case(ii)_M1_dynamics} being a saddle.
Note that both $\mathcal{M}_1, \Pi_{s_0} \mathcal{N}_1 \subset \left\{q=0,\;s=s_0\right\}$; inside this three-dimensional hyperplane, the intersection of $\Pi_{s_0}$ and $\mathcal{M}_1$ is generically transversal.\\

The associated saddle equilibrium $(u_1^+,0)$ of the flow on $\mathcal{M}_1^+$ \eqref{eq:case(ii)_M1_dynamics} has stable and unstable manifolds given by
\begin{equation}
    \ell_1^s = \left\{p = -\text{sgn}(u-u_1^+)\sqrt{2 V^+(u_1^+)-2 V^+(u)}\right\},\quad \ell_1^u = \left\{p = \text{sgn}(u-u_1^+)\sqrt{2 V^+(u_1^+)-2 V^+(u)}\right\},
\end{equation}
with
\begin{align}
    V^+(u) &= \left(\mathcal{A} + \frac{\mathcal{B}+\mathcal{H} s_0}{k}\right)u - \left(\mathcal{A}+\frac{1}{2k^2}\right)\frac{u^2}{2} - \frac{1}{4k^2} \left(u-2 k\left(\mathcal{B}+\mathcal{H} s_0\right)\right) \sqrt{u \left(u-4 k \left(\mathcal{B}+\mathcal{H} s_0\right)\right)}\nonumber\\
    &\quad + 2 \left(\mathcal{B}+\mathcal{H} s_0\right)^2 \log \left[\!\sqrt{u} + \sqrt{u - 4 k \left(\mathcal{B}+\mathcal{H} s_0\right)}\right] \label{eq:case(ii)_V+u}
\end{align}
Since $\left(V^+\right)'(u) = 0$ if and only if $u = u_1^+$ and $V^+$ is concave, the graph of the Hamiltonian of \eqref{eq:case(ii)_M1_dynamics}, given by
\begin{equation}
    H_1^+(u,p) = \frac{1}{2} p^2 + V^+(u) - V^+(u_1^+),
\end{equation}
can indeed be shown to be of saddle type.

\begin{figure}
    \centering
    \includegraphics[width=0.5\textwidth]{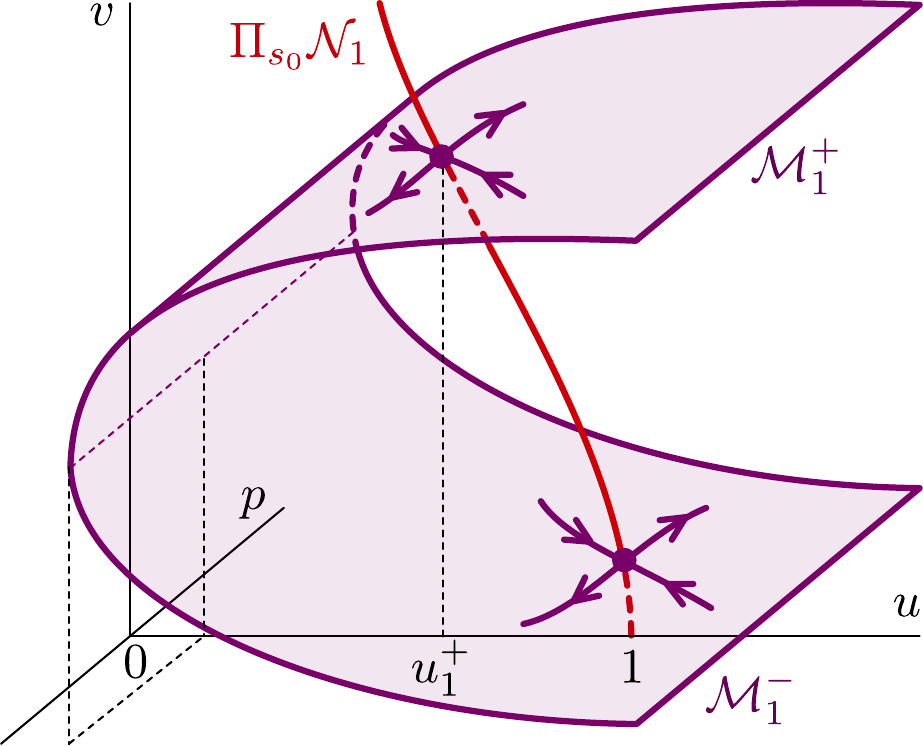}
    \caption{The manifold $\mathcal{M}_1$ \eqref{eq:case(ii)_M0} and the curve $\Pi_{s_0}\mathcal{N}_1$ \eqref{eq:case(ii)_N1} in $(u,p,v)$-space. Depending on the parameter values, the curve $\Pi_{s_0}\mathcal{N}_1$ may intersect $\mathcal{M}_1$ at several places; the equilibria of the flow \eqref{eq:case(ii)_M1_dynamics} on $\mathcal{M}_1$ are located at these intersections. When $\Pi_{s_0}\mathcal{N}_1$ intersects both branches of  $\mathcal{M}_1$, as illustrated here, the associated equilibria are saddles.}
    \label{fig:case(ii)_M1_N1_upv}
\end{figure}


\subsection{Fast dynamics}\label{sss:case(ii)_fastdynamics}

The reduced fast system is given by (cf. \eqref{eq:ODEsystem_modified_fast_rescaled})
\begin{subequations}\label{eq:ODEsystem_case(ii)_fast}
 \begin{align}
  v_\zeta &= q,\\
  q_\zeta &= \mathcal{B} v  - u v^2(1-k v) + \mathcal{H} v s - c q,
 \end{align}
\end{subequations}
with $u=u_0$, $p=p_0$ and $s=s_0$. System \eqref{eq:ODEsystem_case(ii)_fast} is equal to system \eqref{eq:ODEsystem_case(i)_layer}, which was analysed in Section \ref{sec:case(i)_fast}.\\ 
From Lemma \ref{lem:case(i)_layer_transversality}, it follows that heteroclinic orbits exist for (specific, unique) $\mathcal{O}(1)$ values of $c$, as long as \eqref{eq:ODEsystem_case(ii)_fast} has three distinct equilibria -- that is, when the line $\left\{u=u_0,\;p=p_0,\;s=s_0,\;q=0\right\}$ transversally intersects $\mathcal{M}_1$, necessarily implying that this line intersects both branches $\mathcal{M}_1^\pm$. This situation occurs if and only if $(u_0,p_0,s_0)$ is contained in the projection of $\mathcal{M}_1$ onto the $(u,p,s)$-hyperplane, which implies 
\begin{equation}
 u_0 > 4 k \left(\mathcal{B}+\mathcal{H} s_0\right).   
\end{equation}
To facilitate presentation, we introduce the parameter
\begin{equation}
    0 < \alpha := \frac{4 k}{u_0} \left(\mathcal{B}+\mathcal{H} s_0\right) < 1,
\end{equation}
so that $v^\pm(u_0) = \frac{1}{2k} \left(1 \pm \sqrt{1-\alpha}\right)$ \eqref{eq:case(ii)_vpm}. The heteroclinic orbits are explicitly given by $\phi_\dagger(\zeta;u_0,s_0)$ and $\phi_\diamond(\zeta;u_0,s_0)$ \eqref{eq:ODEsystem_case(i)_heteroclinics}, with $v$-profiles $v_\dagger(\zeta)$ and $v_\diamond(\zeta)$ \eqref{eq:ODEsystem_case(i)_heteroclinics_profiles}, and associated $c$-values.
\begin{equation}\label{eq:case(ii)_wavespeeds_heteroclinics_large_c}
c_\dagger(u_0,s_0) = - c_\diamond(u_0,s_0) = \sqrt{\frac{u_0}{8 k}}\left(1 - 3 \sqrt{1-\alpha}\right) = \sqrt{\frac{\mathcal{B} + \mathcal{H} s_0}{2} }\,\frac{1 - 3 \sqrt{1-\alpha}}{\sqrt{\alpha}},
\end{equation}
cf. \eqref{eq:ODEsystem_case(i)_heteroclinics_wavespeeds}. In particular, we observe that $\lim_{\alpha \uparrow 1} c_\dagger = \sqrt{\frac{\mathcal{B} + \mathcal{H} s_0}{2} }$, and $\lim_{\alpha \downarrow 0} c_\dagger = -\infty$; moreover, $c_\dagger$ is monotone in $\alpha$. Hence, for a fixed $\mathcal{O}(1)$ value of $0 < c < \sqrt{\frac{\mathcal{B} + \mathcal{H} s_0}{2} }$, there exists a pair of values $(\alpha_\dagger,\alpha_\diamond)$, with $\frac{16}{25} < \alpha_\diamond < \frac{8}{9} < \alpha_\dagger < 1$, such that there exists a family of heteroclinic connections from $\mathcal{M}_0$ to $\mathcal{M}_1^+$ that lie in the hyperplane $\left\{u = u_0^+,\;s = s_0\right\}$, and a family of heteroclinic connections from $\mathcal{M}_1^+$ to $\mathcal{M}_0$ that lie in the hyperplane $\left\{u = u_0^-,\;s = s_0\right\}$, with $u_0^+ < \frac{9}{2} k \left(\mathcal{B}+\mathcal{H} s_0\right) < u_0^-$; see also Figure \ref{fig:case(ii)_heteroclinic_connections_large_c}. Using \eqref{eq:case(ii)_wavespeeds_heteroclinics_large_c}, we can calculate
\begin{equation}\label{eq:case(ii)_u0pm}
    u_0^\pm = \frac{k}{4}\left(5 c^2 + 18 \left(\mathcal{B}+\mathcal{H} s_0\right) \mp 3 c \sqrt{c^2 + 4 \left(\mathcal{B}+\mathcal{H} s_0\right)}\right).
\end{equation}
For $- \sqrt{\frac{\mathcal{B} + \mathcal{H} s_0}{2} } < c < 0$, the same statement holds, with the direction of the heteroclinic connections reversed.

\begin{figure}
    \centering
    \includegraphics[width=0.6\textwidth]{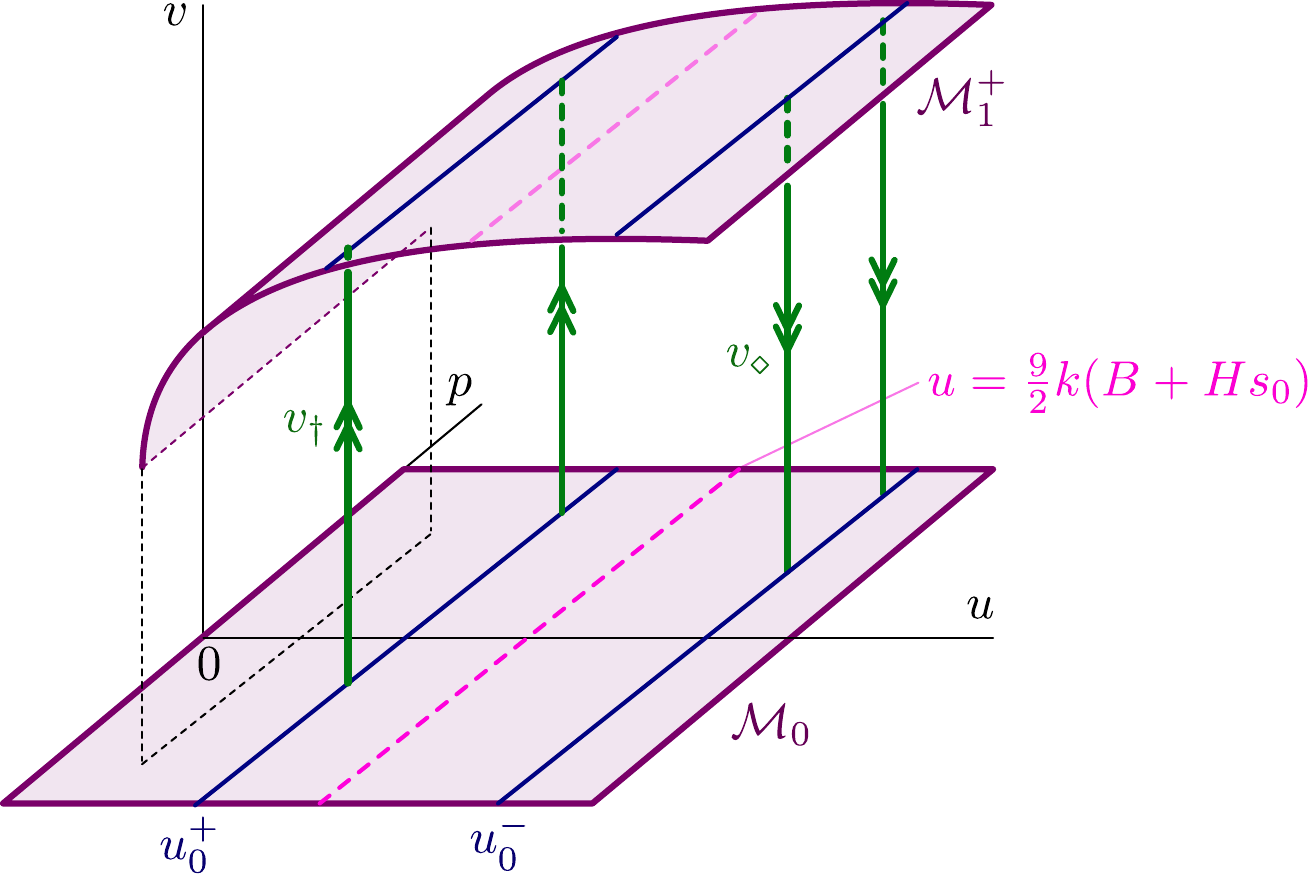}
    \caption{Heteroclinic connections $v_\dagger(\zeta)$ and $v_\diamond(\zeta)$ \eqref{eq:ODEsystem_case(i)_heteroclinics_profiles} between $\mathcal{M}_0$ and $\mathcal{M}_1^+$ for a fixed, positive, $\mathcal{O}(1)$ value of $c$. The direction of the heteroclinic connections is determined by the position of their $u$-coordinate relative to the line $u = \frac{9}{2} k \left(\mathcal{B} + \mathcal{H} s_0\right)$.}
    \label{fig:case(ii)_heteroclinic_connections_large_c}
\end{figure}

\subsection{Singular profile and existence}
We use the analysis from the previous sections to construct a singular pulse profile for system \eqref{eq:ODEsystem_modified_fast_rescaled} in the scaling regime \eqref{eq:scaling_case(ii)}. We first describe the construction of this singular profile, and then state and prove the associated existence theorem, Theorem \ref{case(ii)_thm}. \\

Our first goal is to construct a heteroclinic solution in the semireduced system \eqref{eq:ODEsystem_modified_rescaled_semireduced} between the equilibrium $(u=1,p=0,v=0,q=0,s=s_0) \in \mathcal{M}_0$ and the equilibrium $(u=u_1^+,p=0,v=v^+(u_1^+),q=0,s=s_0) \in \mathcal{M}_1^+$ \eqref{eq:case(ii)_equilibrium_u1+}. This can be done in two ways: we can either connect $\ell^u_0$ to $\ell^s_1$ via $\phi_\dagger$, or connect $\ell^u_1$ to $\ell^s_0$ via $\phi_\diamond$. From the geometry of the stable and unstable manifolds $\ell^{u,s}_0$ and $\ell^{u,s}_1$, it is clear that these constructions are possible for all values of $u_1^+$; see also Figure \ref{fig:case(ii)_heteroclinic_semireduced}. The reversibility symmetry of both \eqref{Eq:case(ii)_M0_dynamics} and \eqref{eq:case(ii)_M1_dynamics}, the unstable and stable manifolds of each saddle equilibrium are mapped onto each other by reflection in the $u$-axis. Therefore, the $u$-coordinate of the upwards jump via $\phi_\dagger$ is equal to the $u$-coordinate of the downwards jump via $\phi_\diamond$. This value of $u := u_*$ is the solution to the equation
\begin{equation}\label{eq:case(ii)_ustar_equation}
    \sqrt{\mathcal{A}} (u_*-1) = -\text{sgn}(u_*-u_1^+)\sqrt{2 V^+(u_1^+)-2 V^+(u_*)},
\end{equation}
cf. \eqref{eq:case(ii)_V+u}. 
From the analysis in \ref{sss:case(ii)_fastdynamics}, it follows that, for fixed $c$, only \emph{one} of the two connections can exist, cf. Figure \ref{fig:case(ii)_heteroclinic_connections_large_c}. The sign of $u_* - \frac{9}{2} k \left(\mathcal{B} + \mathcal{H} s_0\right)$ determines whether the upwards connection or the downwards connection exists. For the persistence of the heteroclinic orbits between the two saddle points in the semi-reduced system \eqref{eq:ODEsystem_modified_rescaled_semireduced} -- that is, for fixed $s=s_0$ -- the approach of \cite[section 3]{BCD.2019} can be used. In fact, the persistence argument from the proof of \cite[Theorem 3.4]{JDC-BM.2020} can be applied directly in this case.

\begin{figure}
    \centering
    \includegraphics[width=0.5\textwidth]{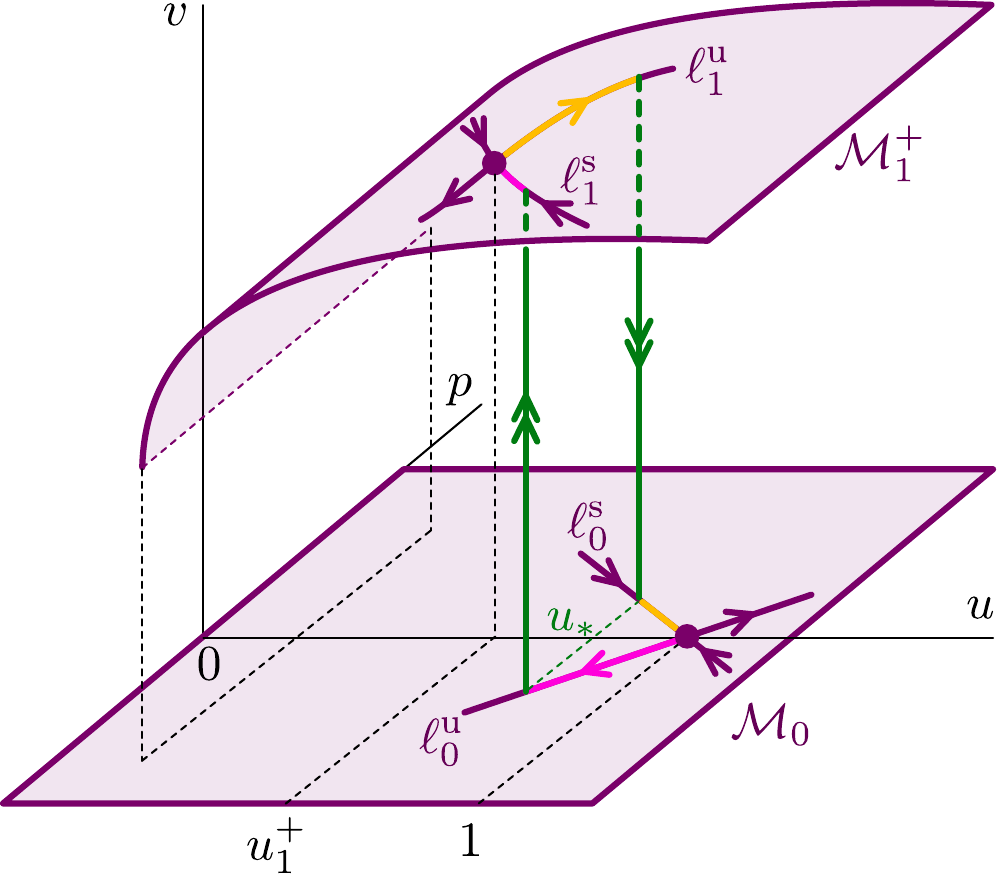}
    \caption{Heteroclinic connections between the equilibrium on $\mathcal{M}_0$ and the equilibrium on $\mathcal{M}_1^+$, for fixed $s=s_0$. These connections persist in the semireduced system \eqref{eq:ODEsystem_modified_rescaled_semireduced}.}
    \label{fig:case(ii)_heteroclinic_semireduced}
\end{figure}

Our ultimate goal is to construct a homoclinic orbit to the trivial equilibrium $(u=1,p=0,v=0,q=0,s=0)$ of the full system \eqref{eq:ODEsystem_modified_rescaled}. To that end, we use the heteroclinic orbits between the saddle equilibria in the semi-reduced system \eqref{eq:ODEsystem_modified_rescaled_semireduced}, which were constructed for fixed $s=s_0$, to jump between $\mathcal{N}_0$ and $\mathcal{N}_1$. The aforementioned global equilibrium lies on $\mathcal{N}_0$, on which the superslow $s$-flow is in the positive $s$-direction, for $s>0$. On $\mathcal{N}_1$, the superslow $s$-flow can have different directions depending on the position of the equilibria of the superslow flow on $\mathcal{N}_1$, cf. Figure \ref{fig:case(ii)_superslow_s_N1}; in particular, there exist configurations of these equilibria such that the flow on $\mathcal{N}_1$ is decreasing. Based on these observations, we propose the following construction, consisting of four segments:
\begin{enumerate}
    \item Start at $(1,0,0,0,0) \in \mathcal{N}_0$, and follow the superslow flow on $\mathcal{N}_0$ upwards to $(1,0,0,0,s_0)$.
    \item Use an \emph{upwards} heteroclinic orbit in the semi-reduced system \eqref{eq:ODEsystem_modified_rescaled_semireduced} to flow from $(1,0,0,0,s_0) \in \mathcal{M}_0$ to $(\left.u_1^+\right|_{s=s_0},0,\left.v^+(u_1^+)\right|_{s=s_0},0,s_0) \in \mathcal{M}_1^+$. Note that, by definition, $(\left.u_1^+\right|_{s=s_0},0,\left.v^+(u_1^+)\right|_{s=s_0},0,s_0) \in \mathcal{N}_1$, as long as $s_0 < s_\text{max}$, cf. Figure \ref{fig:case(ii)_superslow_s_N1}.
    \item Follow the superslow flow on $\mathcal{N}_1$ downwards to $(\left.u_1^+\right|_{s=0},0,\left.v^+(u_1^+)\right|_{s=0},0,0)$. Note that the starting point $(\left.u_1^+\right|_{s=s_0},0,\left.v^+(u_1^+)\right|_{s=s_0},0,s_0)$ must be below any equilibria on $\mathcal{N}_1$, see Figure \ref{fig:case(ii)_superslow_s_N1}.
    \item Use a \emph{downwards} heteroclinic orbit in the semi-reduced system \eqref{eq:ODEsystem_modified_rescaled_semireduced} with $s_0 = 0$, to flow from\\ $(\left.u_1^+\right|_{s=0},0,\left.v^+(u_1^+)\right|_{s=0},0,0) \in \left.\mathcal{M}_1^+\right|_{s=0}$ to $(1,0,0,0,0) \in \left.\mathcal{M}_0\right|_{s=0}$.
\end{enumerate}
This concatenation can be symbolically represented as
\begin{equation}\label{case(ii)_homoclinic_concatenation}
    \mathcal{N}_0[0,s_0] \cup \ell_0^u \cup \phi_\dagger(\cdot;u_*,s_0) \cup \ell_1^s|_{s=s_0} \cup \mathcal{N}_1[0,s_0] \cup \ell_1^u|_{s=0} \cup \phi_\diamond(\cdot;u_1^+|_{s=0},0) \cup \ell_0^s|_{s=0}.
\end{equation}
For a sketch, see Figure \ref{fig:case(ii)_full_double_heteroclinic}. This concatenated orbit can be constructed, provided that the following two matching conditions are met:
\begin{equation}\label{case(ii)_matchingconditions}
     \left. u_0^+ \right|_{s=s_0} = \left. u_* \right|_{s=s_0},\quad \left. u_0^- \right|_{s=0} = \left. u_* \right|_{s=0},
 \end{equation}
cf. \eqref{eq:case(ii)_u0pm} and \eqref{eq:case(ii)_ustar_equation}. These conditions fix $s_0$ and $c$ for given parameter values $\mathcal{A}$, $\mathcal{B}$, $\mathcal{H}$ and $k$.\\ 

\begin{figure}
    \centering
    \includegraphics[width=0.6\textwidth]{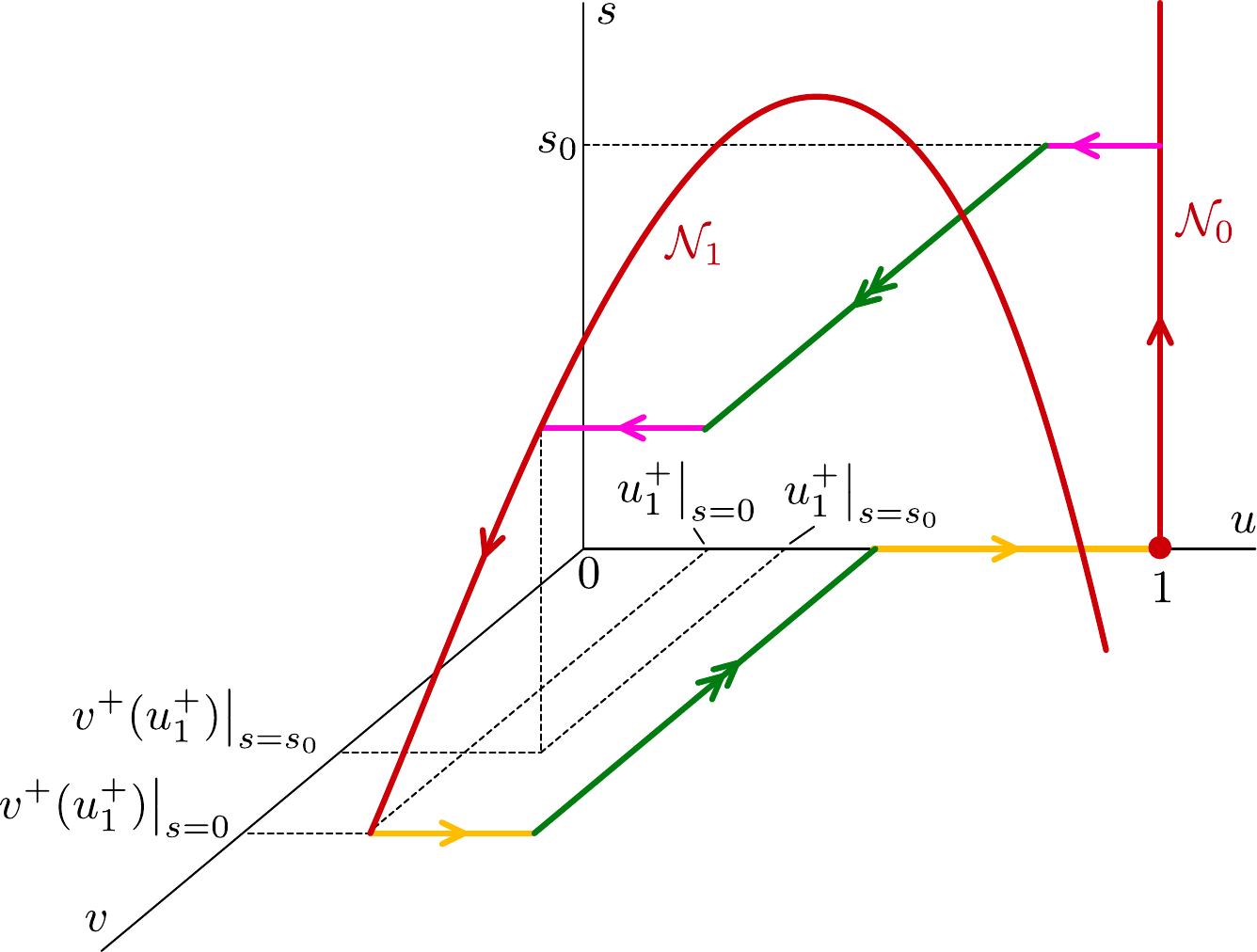}
    \caption{The skeleton construction of a homoclinic orbit to the trivial equilibrium, projected onto the $(u,v,s)$-hyperplane. The heteroclinic connections between $\mathcal{N}_0$ and $\mathcal{N}_1$ are those indicated in Figure \ref{fig:case(ii)_heteroclinic_semireduced}, with corresponding colours. For the (projected) superslow dynamics on $\mathcal{N}_1$, see Figure \ref{fig:case(ii)_superslow_s_N1}.}
    \label{fig:case(ii)_full_double_heteroclinic}
\end{figure}

The persistence of the singular construction presented in this section is stated in the following theorem, which establishes the existence of a homoclinic orbit in the the full system \eqref{eq:ODEsystem_modified_rescaled}/\eqref{eq:ODEsystem_modified_fast_rescaled} for asymptotically small values of $\eps$ and $\delta$, in the relative scaling \eqref{eq:scaling_case(ii)}. The arguments used in the proof are analogous to those used in the proof of Theorem \ref{case(i)_thm_plateau}. However, in this case, the role of $\eps$ and $\delta$ is reversed; hence, the `key players' of the proof --the critical manifolds in the various (semi-) reduced systems-- are different, yielding a qualitatively different phase space geometry.

\begin{thm}\label{case(ii)_thm} Let $s_0>0$ and $c>0$ be such that equations \eqref{case(ii)_matchingconditions} are satisfied.
Assume that
\begin{equation}
    4 k \mathcal{B} < 1,\quad \mathcal{A} > \frac{4 \mathcal{B}^2}{1-4 k \mathcal{B}},\quad\text{and}\quad \mathcal{A}>\frac{1}{k} \frac{\mathcal{B}+\mathcal{H} s_0}{1-4 k \left(\mathcal{B}+\mathcal{H} s_0\right)},
\end{equation}
and that
\begin{equation}
    \left.\frac{\mathcal{B} v^+(u_1^+)}{1-\mathcal{H} v^+(u_1^+)}\right|_{s=s_0} > s_0 > 0,
\end{equation}
where $u_1^+$ is given by \eqref{eq:case(ii)_equilibrium_u1+} and $v^+(u)$ by \eqref{eq:case(ii)_vpm}.\\
There exists a $\eps_0 > 0$ such that for all $0<\eps<\eps_0$, there exists a $\delta_0(\eps)$ such that for all $0<\delta<\delta_0(\eps)$, there exists a solution $\left(u_{h3},p_{h3},v_{h3},q_{h3},s_{h3}\right)$ to system \eqref{eq:ODEsystem_modified_fast_rescaled} that is homoclinic to the equilibrium $(1,0,0,0,0)$. To leading order in $\eps$ and $\delta$, this solution is given by the homoclinic concatenation \eqref{case(ii)_homoclinic_concatenation}, and its speed is given by
\begin{align}
    c =  c_\diamond(u_*(0),0)+\mathcal{O}(|\eps|+|\delta|),
\end{align}
where $u_*(s)$ is the solution to \eqref{eq:case(ii)_ustar_equation}.
\end{thm}

\begin{proof}
We consider system \eqref{eq:ODEsystem_modified_rescaled_semireduced}, which is independent of $\delta$. For fixed $s=s_0$, \eqref{eq:ODEsystem_modified_rescaled_semireduced} is a singularly perturbed system when $0<\eps\ll 1$. From Fenichel theory, it follows that there exists a $0<\eps_0$ such that for all $0<\eps<\eps_0$, the normally hyperbolic branches $\mathcal{M}_0$ \eqref{eq:case(ii)_M0} and $\mathcal{M}_1^\pm$ \eqref{eq:case(ii)_M1_branches} of the two-dimensional critical manifold $\mathcal{M}$ \eqref{eq:case(ii)_M} persist as locally invariant manifolds. The two-dimensional stable and unstable manifolds of these branches, $\mathcal{W}^{u,s}(\mathcal{M}_0)$ and $\mathcal{W}^{u,s}(\mathcal{M}_1^\pm)$, persist as well.\\

Fix $s_0$. For $\eps=0$, $\mathcal{W}^u(\mathcal{M}_0)$ and $\mathcal{W}^s(\mathcal{M}_1^+)$ intersect transversely in the union of heteroclinic orbits $\bigcup_{u \in \mathcal{M}_1^+} \phi_\dagger(u,s_0)$ \eqref{eq:ODEsystem_case(i)_heteroclinics}.
The concatenation $\ell_0^u \cup \phi_\dagger(\cdot;u_*(s_0),s_0) \cup \ell_1^s[s_0]$, with $u_*(s_0)$ given by \eqref{eq:case(ii)_ustar_equation} for $s=s_0$, provides the intersection of the two-dimensional unstable manifold 
of the trivial equilibrium $\mathcal{W}_0^u\left((1,0,0,0)\right)$ with $\mathcal{W}_0^s \left( (u_1^+,0,v^+(u_1^+),0)\right)$, the two-dimensional stable manifold of the unique equilibrium on $\mathcal{M}_1^+$, which exists and lies on $\mathcal{M}_1^+$ by the inequality $\mathcal{A}>\frac{1}{k} \frac{\mathcal{B}+\mathcal{H} s_0}{1-4 k \left(\mathcal{B}+\mathcal{H} s_0\right)}$. We recall from Lemma~\ref{lem:case(i)_layer_transversality} that the heteroclinic connection $\phi_\dagger(\cdot;u_*,s_0)$ breaks transversely upon varying $u \approx u_*$ or upon varying the speed $c\approx c_\dagger(u_*,s_0)$. Thus for all sufficiently small $\eps>0$, we can ensure the intersection of $\mathcal{W}_0^u\left((1,0,0,0)\right)$ with $\mathcal{W}_0^s \left( (u_1^+,0,v^+(u_1^+),0)\right)$ persists, by varying $s$ near $s=s_0$, thereby varying both $u$ and $c$ near $u=u_*(s_0)$ and $c = c_\dagger(u_*(s_0),s_0)$.

Likewise, fix $s=0$. $\mathcal{W}^u(\mathcal{M}_1^+)$ and $\mathcal{W}^s(\mathcal{M}_0)$ intersect transversely in the union of heteroclinic orbits $\bigcup_{u \in \mathcal{M}_1^+}\phi_\diamond(u,0)$ \eqref{eq:ODEsystem_case(i)_heteroclinics}, since $c = c_\dagger(u_*(s_0),s_0) = c_\diamond(u_*(0),0)$ by the choice of $u_*$ \eqref{case(ii)_matchingconditions}. The concatenation $\ell_1^u[s=0] \cup \phi_\diamond(\cdot;u_*,0) \cup \ell_0^s$ provides the intersection of $\mathcal{W}_0^u\left((u_1^+,0,v^+(u_1^+),0)\right)$, the two-dimensional unstable manifold of the unique equilibrium on $\mathcal{M}_1^+$, with the two-dimensional stable manifold of the trivial equilibrium $\mathcal{W}_0^s\left((1,0,0,0)\right)$. By a similar argument as above, for all sufficiently small $\eps>0$, we can ensure the intersection of $\mathcal{W}_0^s\left((1,0,0,0)\right)$ with $\mathcal{W}_0^u \left( (u_1^+,0,v^+(u_1^+),0)\right)$ persists, by varying $u$ (resp. $c$) near $u=u_*$ (resp. $c=c_\diamond(u_*,0)$) through varying $s$ near $s=0$.

These intersections can be lifted into the geometry of the full five-dimensional system~\eqref{eq:ODEsystem_case(i)} when $\delta=0$. We recall that system admits the one-dimensional critical manifolds $\mathcal{N}_0$ \eqref{eq:case(ii)_N0} and $\mathcal{N}_1$ \eqref{eq:case(ii)_N1}, and that the equilibrium $P_0 := (1,0,0,0,0)$ lies on $\mathcal{N}_0$. The associated values $s_\text{max}$ and $u(v_\text{max})$ \eqref{eq:case(ii)_N1_graph_us} are positive because $0<4 k \mathcal{B} <1$ and $\mathcal{A} > \frac{4 \mathcal{B}^2}{1-4 k \mathcal{B}}$. Both $\mathcal{N}_0$ and the separate branches $\mathcal{N}_1^\pm$ \eqref{eq:case(ii)_branches_Npm} are normally hyperbolic for nonnegative values of $u$, $v$ and $s$, and hence persist for sufficiently small $\delta>0$, as do their three-dimensional stable manifolds $\mathcal{W}^\mathrm{s}(\mathcal{N}_0), \mathcal{W}^\mathrm{s}(\mathcal{N}_1^\pm)$ and three-dimensional unstable manifolds $\mathcal{W}^\mathrm{u}(\mathcal{N}_0), \mathcal{W}^\mathrm{u}(\mathcal{N}_1^\pm)$. In the limit $\delta \to 0$ the unstable manifold $\mathcal{W}^\mathrm{u}(P_0)$ corresponds to the three-dimensional manifold obtained by taking the union
\begin{align}
   \mathcal{W}^\mathrm{u}(P_0) = \bigcup_{s\in \mathcal{N}_0} \mathcal{W}_0^u\left((1,0,0,0)\right)
\end{align}
of the two-dimensional manifolds $\mathcal{W}_0^u\left((1,0,0,0)\right)$ of the trivial equilibrium in $(u,p,v,q)$-space in~\eqref{eq:ODEsystem_modified_rescaled_semireduced} for $s \in \mathcal{N}_0$. Similarly the the stable manifold $\mathcal{W}^\mathrm{s}(P_0)$ of $P_0$ corresponds to the three-dimensional manifold obtained by taking the union
\begin{align}
   \mathcal{W}^\mathrm{s}(P_0) = \bigcup_{s\in \mathcal{N}_0} \mathcal{W}_0^s\left((1,0,0,0)\right)
\end{align}
of the two-dimensional stable manifolds $\mathcal{W}_0^s\left((1,0,0,0)\right)$ for $s\in \mathcal{N}_0$.
By the argument above, when $c=c_\dagger(u_*,s_0)=c_\diamond(u_*,0)$ and $\delta=0$, $\mathcal{W}^\mathrm{s}(P_0)$ transversely intersects the three-dimensional manifold $\mathcal{W}^u(\mathcal{N}_1^+)$ in the hyperplane $\left\{ u = u_*,\; p = -p_* \right\}$ along the intersection of $\mathcal{W}_0^s\left((1,0,0,0)\right)$ with $\mathcal{W}_0^u\left((u_1^+,0,v^+(u_1^+),0)\right)$, and this intersection persists for all sufficiently small $\eps>0$ and values of $c\approx c_\diamond(u_*(0),0)$.  Likewise, $\mathcal{W}^\mathrm{u}(P_0)$ transversely intersects the three-dimensional manifold $\mathcal{W}^s(\mathcal{N}_1^+)$, in the hyperplane $\left\{ u = u_*,\; p = p_* \right\}$ along the intersection of $\mathcal{W}_0^u\left((1,0,0,0)\right)$ with $\mathcal{W}_0^s\left((u_1^+,0,v^+(u_1^+),0)\right)$, and this transverse intersection persists for all sufficiently small $\eps>0$ and values of $c\approx c_\dagger(u_*(s_0),s_0)$.

Now, fix $0<\eps<\eps_0$ and $c$ near $c_\dagger(u_*(s_0),s_0)$. There exists $\delta_0(\eps)$ such that for all sufficiently small $0<\delta<\delta_0(\eps)$, the unstable manifold $\mathcal{W}^\mathrm{u}(P_0)$ therefore transversely intersects $\mathcal{W}^s(\mathcal{N}_1^+)$ at a value of $\tilde{u}(\eps, \delta c)=u_*+\mathcal{O}(\eps+\delta +|c-c_\dagger(u_*(s_0),s_0)|)$ along a stable fiber of the perturbed trajectory on $\mathcal{N}_1^+$. Due to the inequality $\left.\frac{\mathcal{B} v^+(u_1^+)}{1-\mathcal{H} v^+(u_1^+)}\right|_{s=s_0} > s_0 > 0$, the dynamics on $\mathcal{N}_1^+$ are monotonic by the absence of an equilibrium of the superslow $s$-flow. By the exchange lemma~\cite{Schecter.2008}, $\mathcal{W}^\mathrm{u}(P_0)$ aligns exponentially close (in $\delta)$) along the three-dimensional manifold $\mathcal{W}^\mathrm{u}(\mathcal{N}_1^+)$ upon exiting a small neighborhood of $\mathcal{N}_1^+$. We note that when $\eps=\delta=0$ and $u=u_*$, the projection of $\mathcal{N}_1^+$ onto $\mathcal{N}_0$ is orthogonal, thus also for values of $u$ near $u_*$, and all sufficiently small $\eps>0$, $0<\delta<\delta_0(\eps)$. In particular, the projection of $\mathcal{N}_1^+$ onto $\mathcal{N}_0$ transversely intersects $\mathcal{N}_0$ for all sufficiently small $|c-c_\dagger(u_*(s_0),s_0)|$ and $\eps>0$, and $0<\delta<\delta_0(\eps)$. 

Using the fact that $\mathcal{W}^\mathrm{s}(P_0)$ transversely intersects $\mathcal{W}^u(\mathcal{N}_1^+)$ when $\delta=0$ along the intersection of $\mathcal{W}_0^s\left((1,0,0,0)\right)$ with $\mathcal{W}_0^u \left( (u_1^+,0,v^+(u_1^+),0)\right)$, and the fact that this intersection breaks transversely upon varying the wave speed $c\approx c_\diamond(u_*(0),0)$, it remains to adjust
\begin{align}
    c=c_\diamond(u_*(0),0)+\mathcal{O}(\eps, \delta) = c_\dagger(u_*(s_0),s_0)+\mathcal{O}(\eps, \delta)
\end{align}
so that $\mathcal{W}^\mathrm{u}(P_0)$ intersects $\mathcal{W}^\mathrm{s}(P_0)$ upon exiting a neighborhood of $\mathcal{N}_1^+$ corresponding to a homoclinic orbit to the equilibrium $P_0$ of the full system~\eqref{eq:ODEsystem_case(i)} for all sufficiently small $0<\delta<\delta_0(\eps)$, which is approximated by the singular homoclinic orbit \eqref{case(ii)_homoclinic_concatenation}. 
\end{proof}


\section{Numerics} \label{sec:num}

In this section, we present a numerical investigation of System \eqref{eq:UVS_RDsystem_modified} aimed to confirm and extend the analytical results obtained in Sections \ref{sec:casei}-\ref{sec:caseii}. In particular, in Section~\ref{sec:numerics-auto} we use the numerical software AUTO to illustrate the existence results provided in Theorems \ref{case(i)_thm_noplateau}-\eqref{case(i)_thm_plateau} and \ref{case(ii)_thm}. Then in~\ref{sec:numerics-dns}, we perform numerical simulations of System \eqref{eq:UVS_RDsystem_modified} for parameter values corresponding to case (i) and case (ii) in order to explore the dynamics of the resulting pulses, and to provide numerical evidence for the stability of the pulses in the different cases considered in Sections \ref{sec:casei}-\ref{sec:caseii}.

\subsection{AUTO continuation: pulses of cases (i) and (ii) with(out) superslow plateau}\label{sec:numerics-auto} 
Using AUTO70p, we illustrate the results of Theorems \ref{case(i)_thm_noplateau}-\ref{case(i)_thm_plateau} and \ref{case(ii)_thm}, and in particular for small $\eps,\delta$, we explore the resulting traveling wave solutions in relation to the singular limit structures in Sections~\ref{sec:casei}-\ref{sec:caseii}.

We begin with the pulses constructed in Section~\ref{sec:casei}. Fixing the parameters $\mathcal{A}=1.5$, $\mathcal{B}=0.2$, $\mathcal{H}=0.1$, $\eps=10^{-3}$, \mbox{$\delta=10^{-2}$}, we find a pulse \emph{with} superslow plateau at the value $k=0.955$, corresponding to a traveling wave solution as in Theorem~\ref{case(i)_thm_plateau}. Continuing in AUTO, we slightly increase $k$ and find that the superslow plateau reduces in width and eventually disappears, corresponding to the case of Theorem \ref{case(i)_thm_noplateau}, and showing that in fact these families of pulses can lie on the same solution branch. The results of the continuation are depicted in Figure~\ref{fig:auto_casei}.

\begin{figure}
    \centering
    \includegraphics[width=0.4\textwidth]{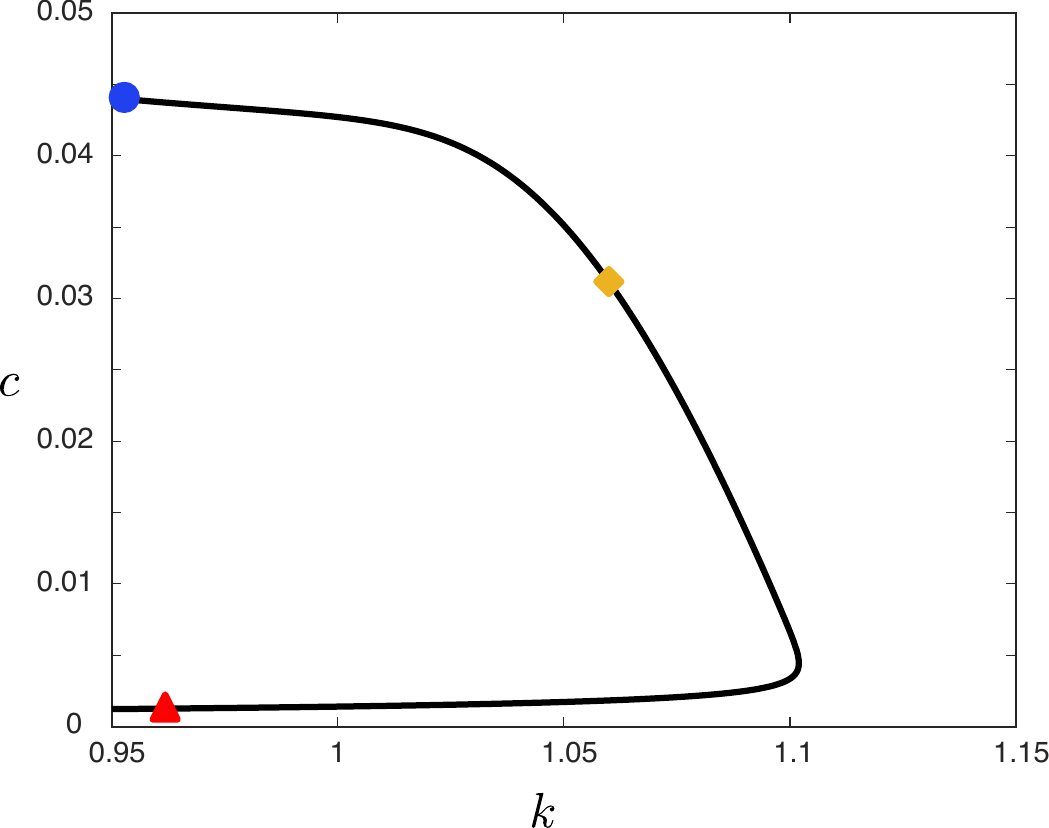}\hspace{0.1\textwidth}
    \includegraphics[width=0.38\textwidth]{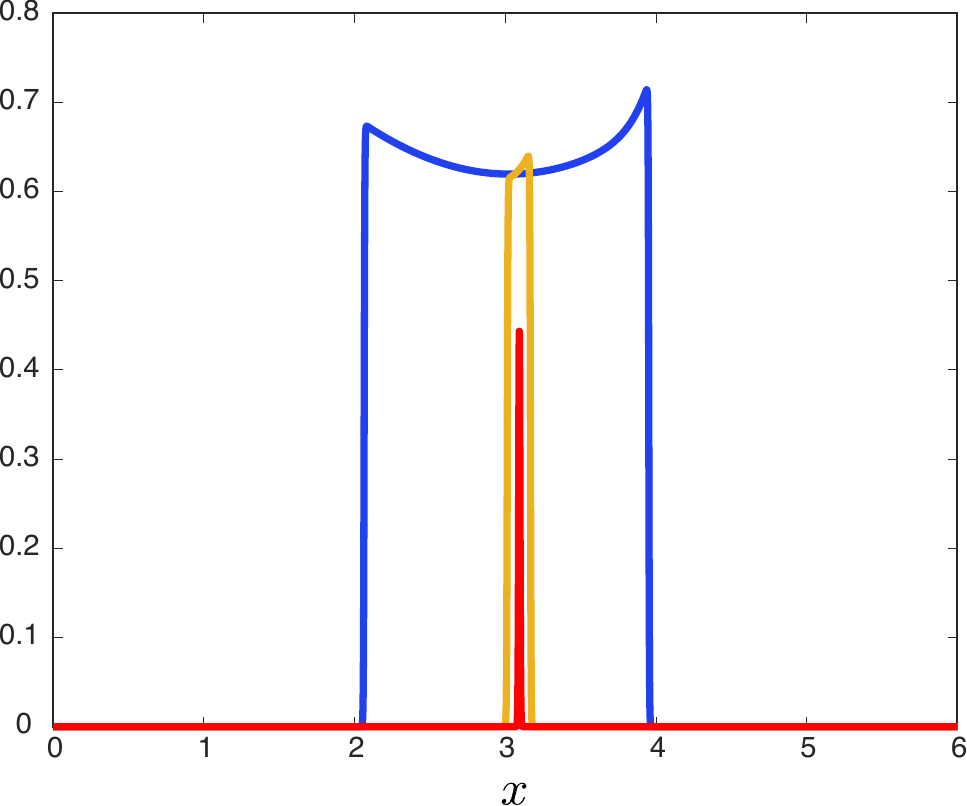}\\~\\
    \includegraphics[width=0.4\textwidth]{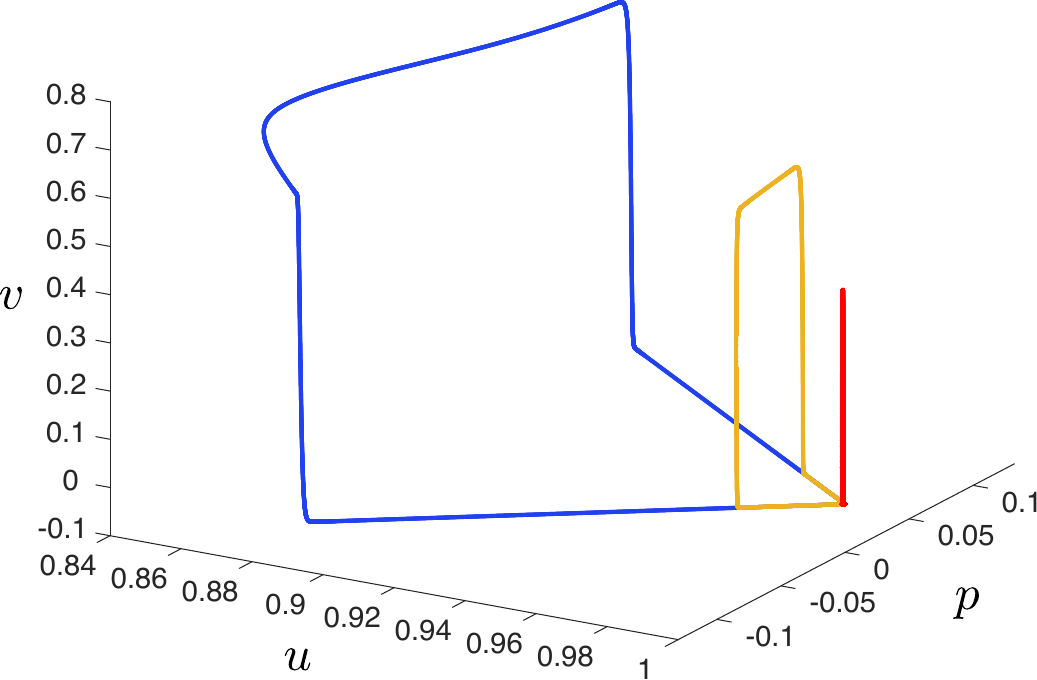}\hspace{0.1\textwidth}
    \includegraphics[width=0.4\textwidth]{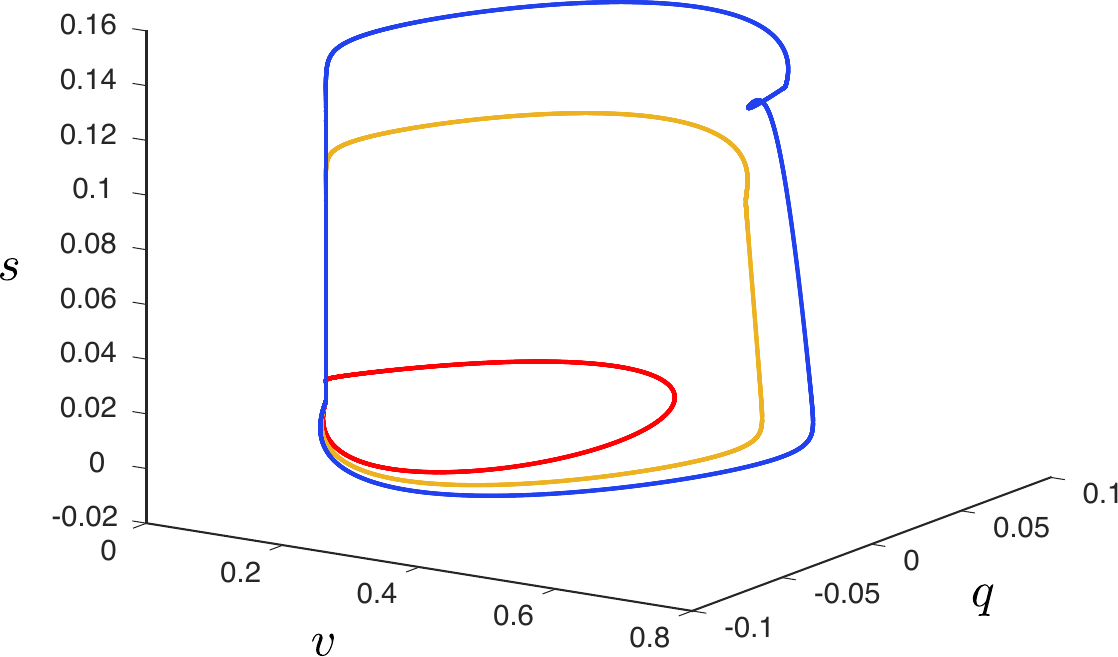}
    \caption{Shown are the results of numerical continuation in the parameter $k$ and wave speed $c$ for the parameters $\mathcal{A}=1.5$, $\mathcal{B}=0.2$, $\mathcal{H}=0.1$, $\eps=10^{-3}$, \mbox{$\delta=10^{-2}$}, corresponding to case (i) in Section~\ref{sec:casei}. The upper left figure depicts $c$ vs. $k$, while the upper right figure depicts $v$-profiles for $3$ different pulses along the continuation curve (color matches corresponding symbol along continuation curve), showing the decreasing plateau width. The lower panels depict the same $3$ pulses in $(u,p,v)$-space (left) and $(v,q,s)$-space (right). Note that the continuation curve eventually turns back on itself, resulting in a fast homoclinic pulse with trivial $s$ and $u$ dynamics (red). }
    \label{fig:auto_casei}
\end{figure}

Proceeding similarly for the pulses constructed in Section~\ref{sec:caseii}, we fix the parameters $\mathcal{A}=1.5$, $\mathcal{B}=0.2$, $\mathcal{H}=0.1$, $\eps=10^{-3}$, $\delta=10^{-4}$, and we again find a pulse at the value $k=0.955$, corresponding to a traveling wave solution of Theorem~\ref{case(ii)_thm}; see Figure~\ref{fig:auto_caseii}. As above, we similarly continue in $k$ and note that again the super slow plateau disappears, resulting in a pulse with trivial $s$-dynamics; see the paragraph directly below \eqref{eq:scaling_case(ii)}.

\begin{figure}
    \centering
    \includegraphics[width=0.4\textwidth]{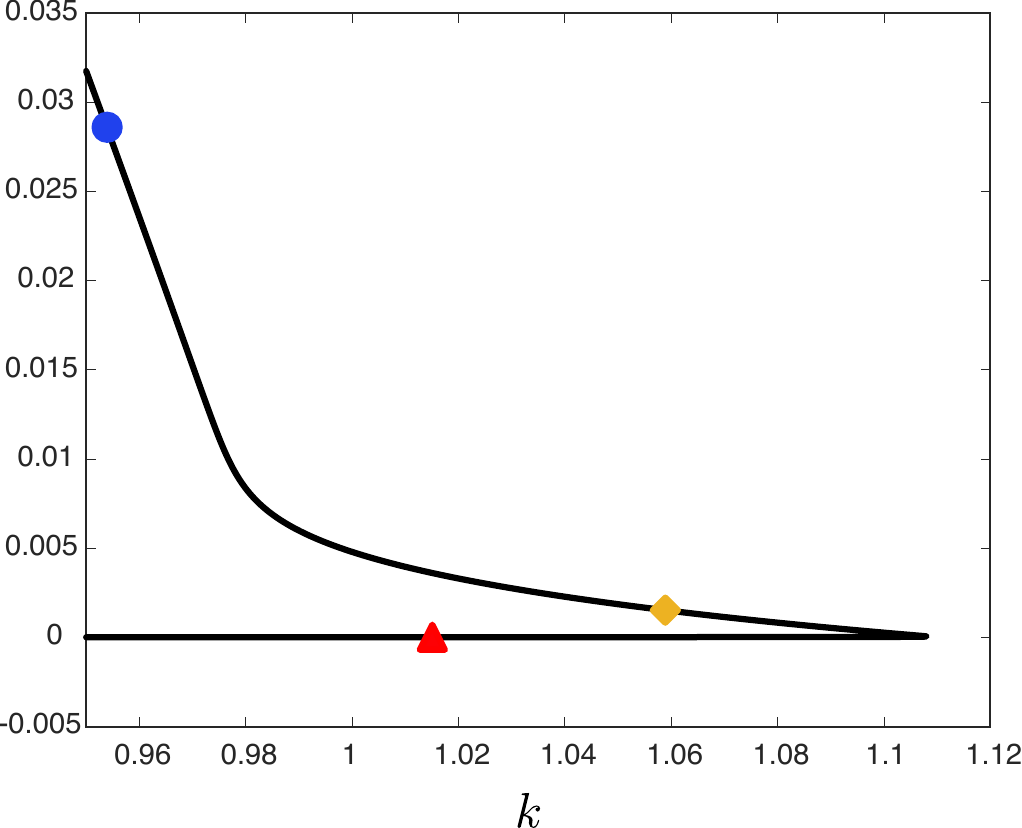}\hspace{0.1\textwidth}
    \includegraphics[width=0.39\textwidth]{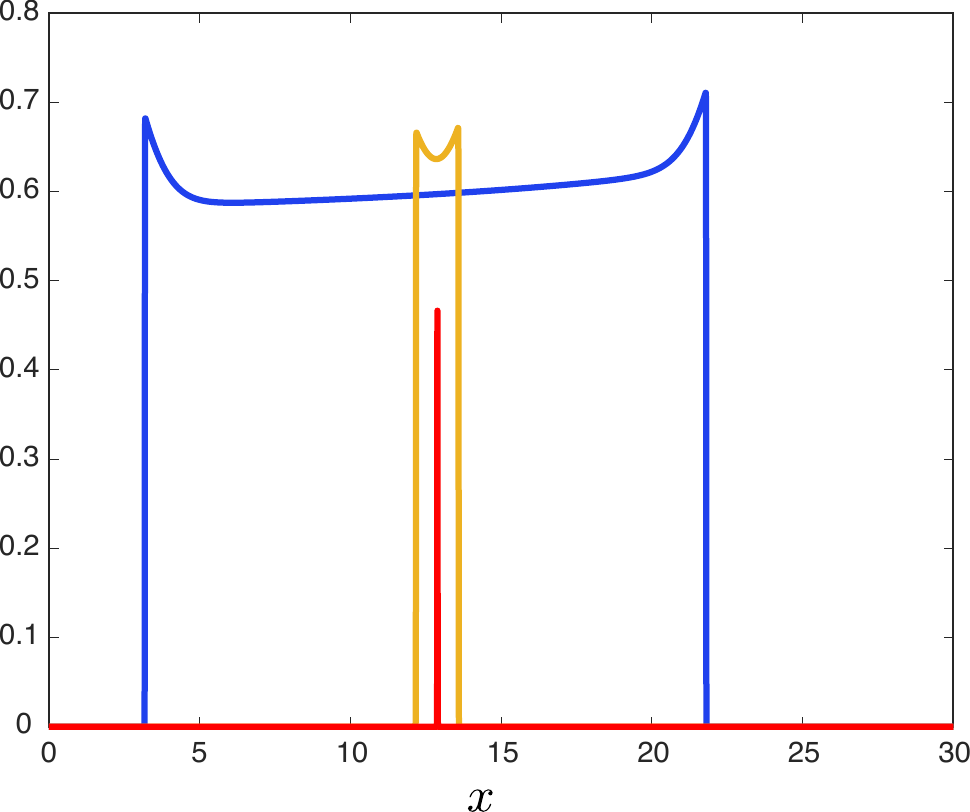}\\~\\
    \includegraphics[width=0.4\textwidth]{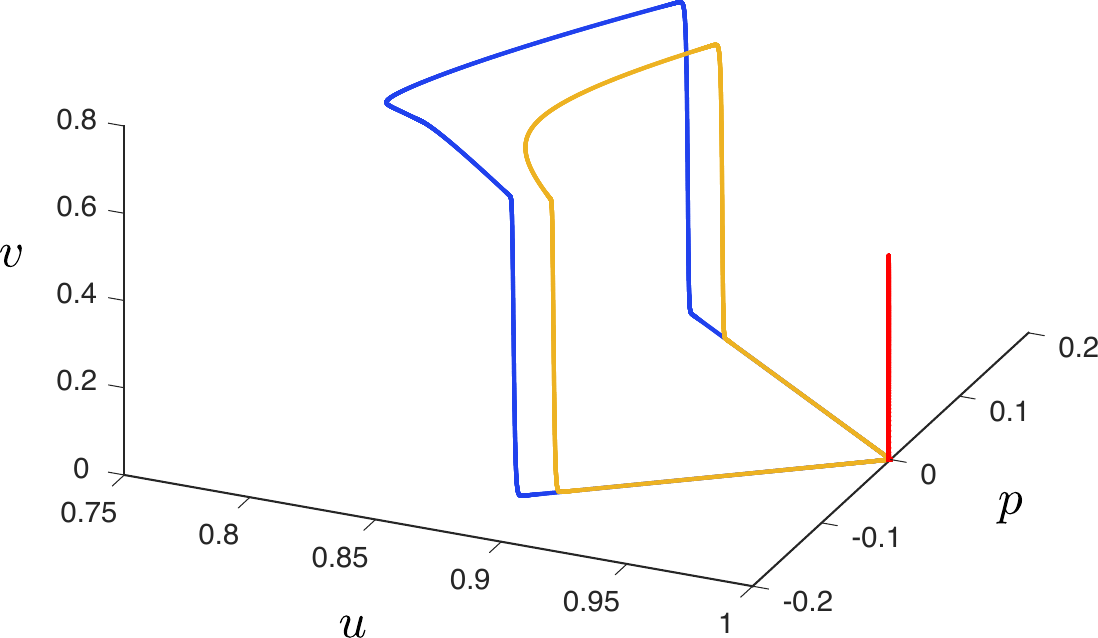}\hspace{0.1\textwidth}
    \includegraphics[width=0.4\textwidth]{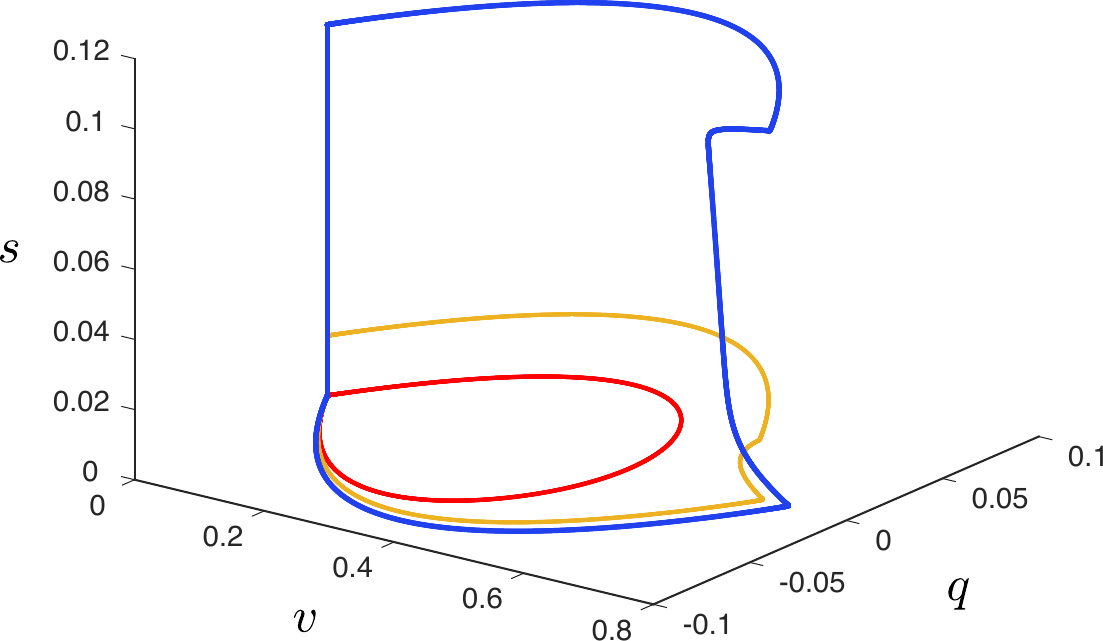}
\caption{Shown are the results of numerical continuation in the parameter $k$ and wave speed $c$ for the parameters $\mathcal{A}=1.5$, $\mathcal{B}=0.2$, $\mathcal{H}=0.1$, $\eps=10^{-3}$, $\delta=10^{-4}$, corresponding to case (ii) in Section~\ref{sec:caseii}. The upper left figure depicts $c$ vs. $k$, while the upper right figure depicts $v$-profiles for $3$ different pulses along the continuation curve (color matches corresponding symbol along continuation curve), showing decreasing plateau width and resulting in trivial $s$-dynamics (yellow). Similarly to case (i) in Figure~\ref{fig:auto_casei}, the continuation curve eventually turns back sharply on itself, resulting in a fast homoclinic pulse with both trivial $s$ and $u$ dynamics (red). The lower panels depict the same $3$ pulses in $(u,p,v)$-space (left) and $(v,q,s)$-space (right).  }
    \label{fig:auto_caseii}
\end{figure}

Finally, we fix $k=0.955$ and continue between the solutions of cases (i) and (ii) above, by decreasing $\delta$ from $10^{-2}$ to $10^{-4}$. The results are depicted in Figures~\ref{fig:auto_case_i_ii} and \ref{fig:auto_case_i_ii_pulses}.

\begin{rmk}\label{rmk:case_iv} The continuation between cases (i) and (ii), as shown in Figures \ref{fig:auto_case_i_ii} and \ref{fig:auto_case_i_ii_pulses}, can be seen as an example of case (iv), i.e. $\eps \sim \delta$. Although we do not treat this case analytically, a rough back-of-the-envelope calculation would start with writing $\delta = \delta_0 \eps$ for fixed $\delta_0$. Application of this scaling to systems \eqref{eq:ODEsystem_modified_rescaled} and \eqref{eq:ODEsystem_modified_fast_rescaled} reveals that its fast reduced system is equal to \eqref{eq:ODEsystem_case(i)_layer}, as treated in Section \ref{sec:case(i)_fast}. The (single) slow reduced system associated to the scaling $\delta = \delta_0 \eps$ has simultaneous $(u,p)$ and $s$-dynamics, given by \eqref{eq:ODEsystem_case(i)_superslow} and \eqref{eq:case(i)_sflow_E} in the formulation of Section \ref{sec:casei}; or, equivalently, by \eqref{eq:case(i)_sflow_E} and \eqref{eq:case(ii)_M1_dynamics} in the formulation of Section \ref{sec:caseii}. The associated trivial and nontrivial critical manifolds are given by
    \begin{align*}
        \mathcal{Q}_0 &:= \left\{v=q=0\right\} = \bigcup_{u_0,p_0} \mathcal{E}_0 = \bigcup_{s_0} \mathcal{M}_0, \\
        \mathcal{Q}_1 &:= \left\{\mathcal{B}+\mathcal{H}s-u v(1-k v)= 0\right\} = \bigcup_{u_0,p_0} \mathcal{E}_1^\pm = \bigcup_{s_0} \mathcal{M}_1,
    \end{align*}
cf. \eqref{eq:case(i)_E0_Epm} and \eqref{eq:case(ii)_M0}, \eqref{eq:case(ii)_M1}. The key observation enabling analytical accessibility is that the slow dynamics are decoupled, that is, $(u,p)$ does not influence $s$ and vice versa. The singular construction of a travelling pulse and its proof of persistence can therefore be carried out in a manner analogous to the analysis presented in sections \ref{sec:casei} and \ref{sec:caseii}. One would have to carry out such analysis in detail to investigate how the existence conditions of Theorem \ref{case(i)_thm_plateau} would match with those of Theorem \ref{case(ii)_thm}.
\end{rmk}

\begin{figure}
    \centering
    \includegraphics[width=0.4\textwidth]{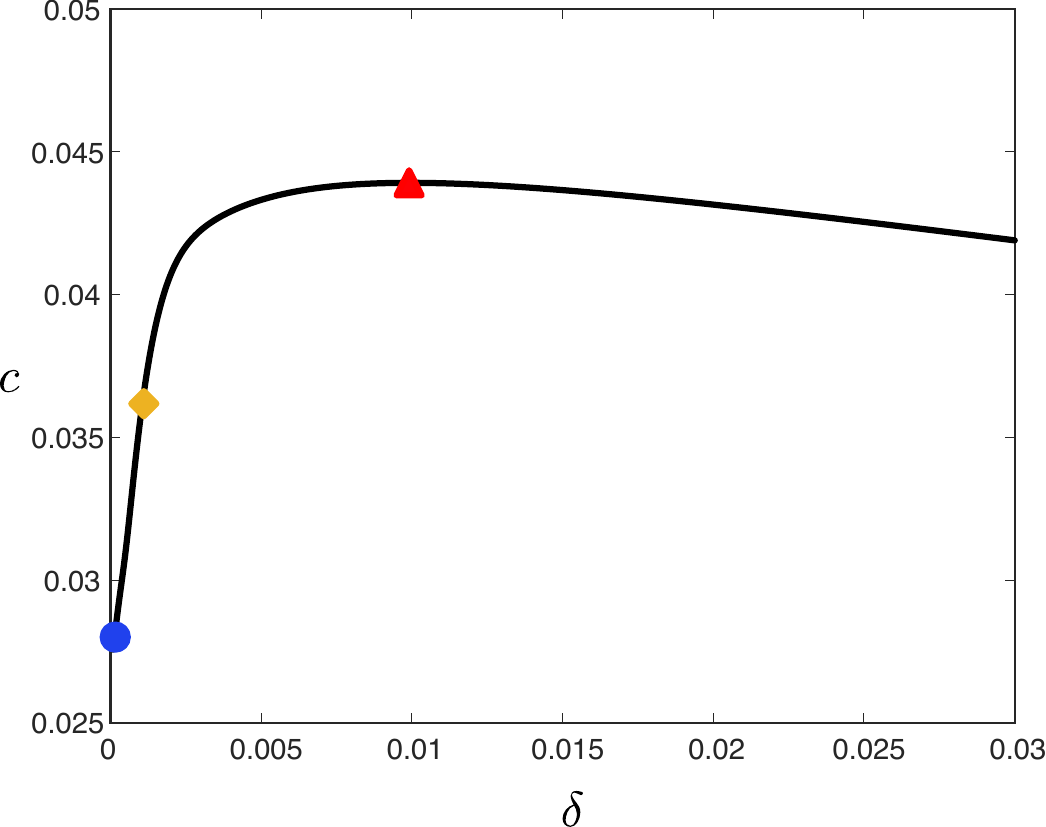}\hspace{0.05\textwidth}
    \includegraphics[width=0.5\textwidth]{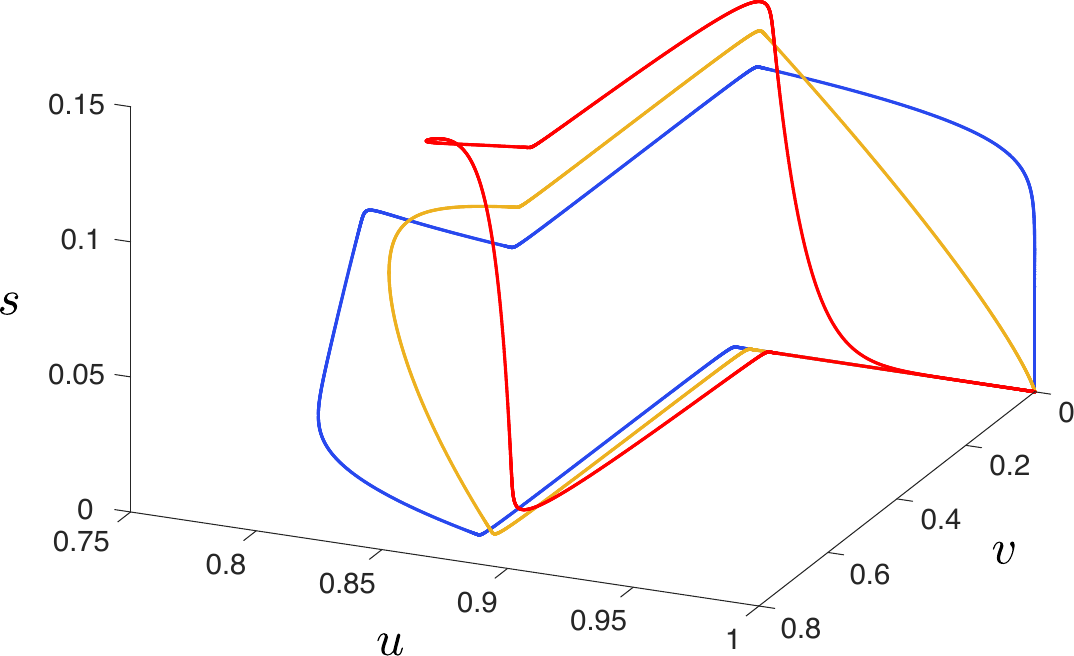}
\caption{(Left) Shown are the results of numerical continuation in the parameter $\delta$ and wave speed $c$ for the parameters $A=1.5, B=0.2, H=0.1, \eps=10^{-3}, k=0.955$, depicting a transition between the pulses of cases (i) and (ii). (Right) Depicted in in $(u,v,s)$-space are $3$ pulses obtained along the transition at the values $\delta=10^{-2}, 10^{-3}, 10^{-4}$. The color of the orbit matches that of the corresponding symbol in the continuation curve. The associated $u,v$, and $s$ profiles are shown in Figure~\ref{fig:auto_case_i_ii_pulses}. }
    \label{fig:auto_case_i_ii}
\end{figure}

\begin{figure}
    \centering
    \includegraphics[width=0.3\textwidth]{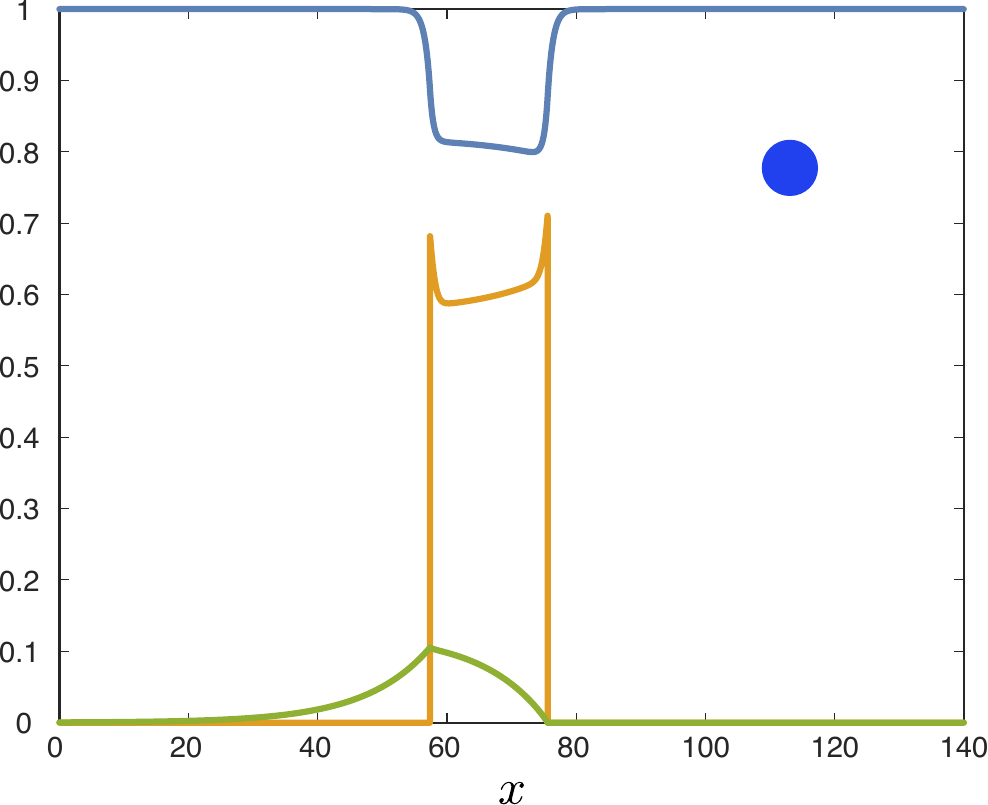}\hspace{0.025 \textwidth}
    \includegraphics[width=0.3\textwidth]{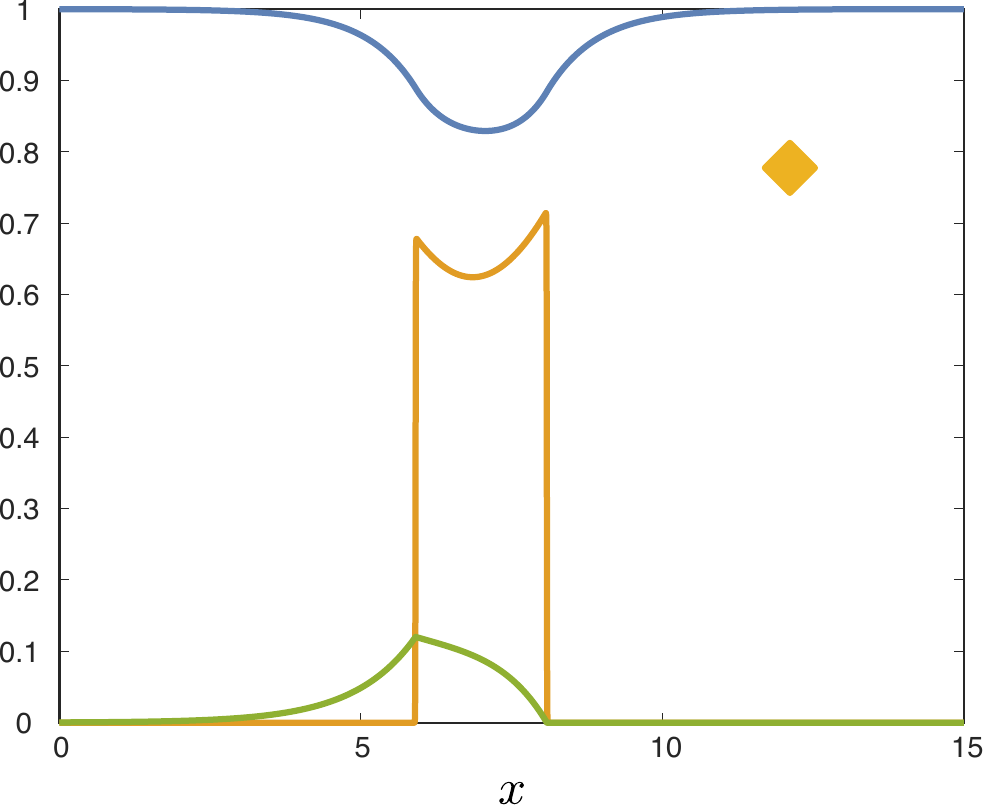}\hspace{0.025 \textwidth}
    \includegraphics[width=0.3\textwidth]{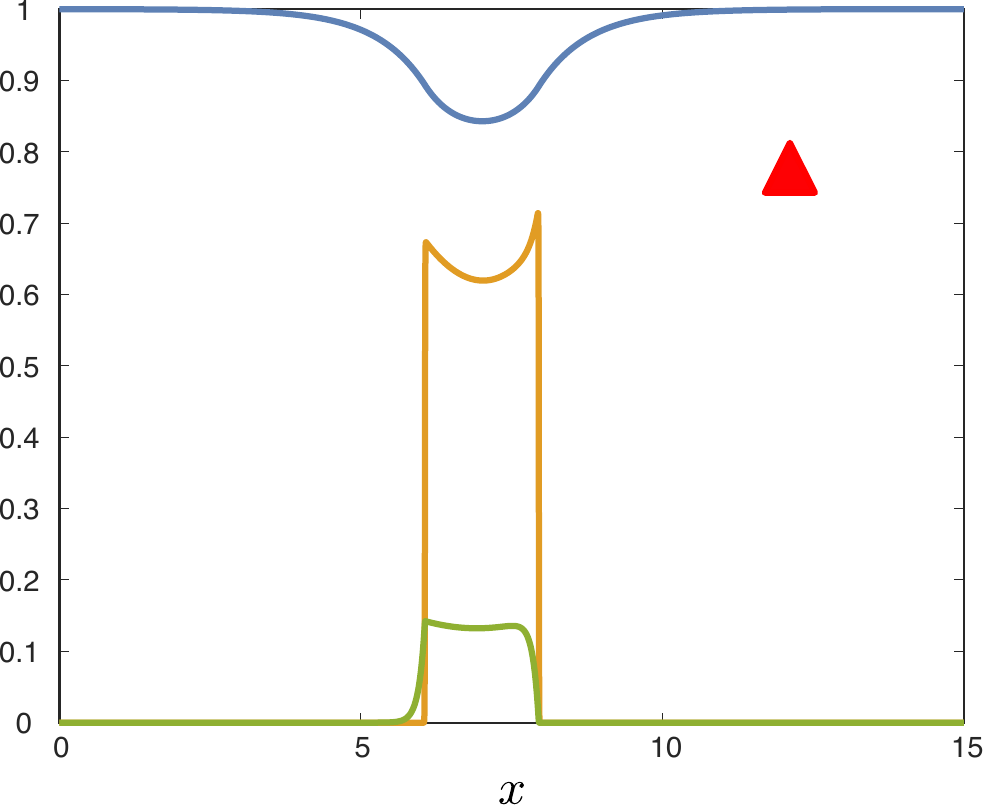}
\caption{Shown are the pulses from Figure~\ref{fig:auto_case_i_ii} obtained at the values (from left to right) $\delta=10^{-4}, 10^{-3}, 10^{-2}$, labeled with the symbols as in the continuation curve in Figure~\ref{fig:auto_case_i_ii}. The $u,v,s$ profiles are depicted in blue, orange, and green, respectively. We can see how the structure of the profile changes as we transition between cases (i) and (ii) of Sections~\ref{sec:casei}-\ref{sec:caseii}. (Note the different length scale on the leftmost plot.)}
    \label{fig:auto_case_i_ii_pulses}
\end{figure}

\subsection{Direct simulations}\label{sec:numerics-dns}

The scope of this section is to illustrate travelling wave solutions in cases (i) and (ii) as predicted by Theorems \ref{case(i)_thm_noplateau}-\eqref{case(i)_thm_plateau} and \eqref{case(ii)_thm}, respectively. To this aim, we perform direct simulation of System \eqref{eq:UVS_RDsystem_modified} in MATLAB by using finite differences for spatial discretization with periodic boundary conditions together with MATLAB’s ode15s routine for time integration. Our investigation is performed in case (i) considering both profiles without and with superslow plateau and in case (ii). The number $N$ of points used for the spatial discretization is $N=29970$ for case (i) and $N=10000$ for case (ii). The initial conditions are chosen from the profiles obtained with the AUTO continuation discussed in Section \ref{sec:numerics-auto} as follows. For case (i), we fix the parameters $\mathcal{A}=1.5$, $\mathcal{B}=0.2$, $\mathcal{H}=0.1$, $\eps=10^{-3}$, $\delta=10^{-2}$, and $k=1.059$. Moreover, we consider a domain length $L=10$. This leads to a numerical speed $c=0.0316493$, which in turn implies $\mathcal{D}=3160$ (see Eq.~\eqref{eq:delta}). The resulting profile obtained in Section \ref{sec:numerics-auto} is used as initial condition for a travelling pulse corresponding to case (i) without superslow plateau.
On the other hand, choosing the profile corresponding to $k=0.955$, $c=0.0439164$ in the AUTO continuation -- which leads to $\mathcal{D}=2277$ -- and $L=20$ while leaving the parameters $\mathcal{A}$, $\mathcal{B}$, $\mathcal{H}$, $\varepsilon$, and $\delta$ unaltered, we obtain the initial condition used for the numerical simulations corresponding to case (i) with superslow plateau. Finally, fixing $\mathcal{A}$, $\mathcal{B}$, $\mathcal{H}$, and $k$ as in this last scenario while considering $\eps=10^{-2}$, $\delta=10^{-3}$, $c=0.0266721$, $\mathcal{D}=37492$, and $L=60$ we obtain the AUTO orbit to use as initial condition for case (ii).
In case (i), we recover the two types of travelling wave solutions -- without and with superslow plateau -- predicted by Theorems \ref{case(i)_thm_noplateau} and \eqref{case(i)_thm_plateau}, respectively: for a comparison of the orbits in $(V, \, V_x, \, S)$ phase space (corresponding to $(v, \, q, \, s)$ in our analysis), see Figures \ref{fig:case(i)_no_plateau}-\ref{fig:casei}(a) for solutions without superslow plateau, and Figures \ref{fig:case(i)_plateau}-\ref{fig:casei}(b) for solutions with superslow plateau. In case (ii), we retrieve the travelling wave solution constructed in Theorem \ref{case(ii)_thm}, respectively: for a comparison of the orbits in $(V, \, U, \, S)$ phase space (corresponding to $(v, \, u, \, s)$ in our analysis), see Figures \ref{fig:case(ii)_full_double_heteroclinic}-\ref{fig:caseii}. The space-time evolution of the travelling pulses is shown in Fig.~\ref{fig:timeev_casei} for case (i) and Fig.~\ref{fig:timeev_caseii} for case (ii).



\begin{figure}
\begin{minipage}{.3\textwidth}
    \centering
    \begin{overpic}[scale=0.34]{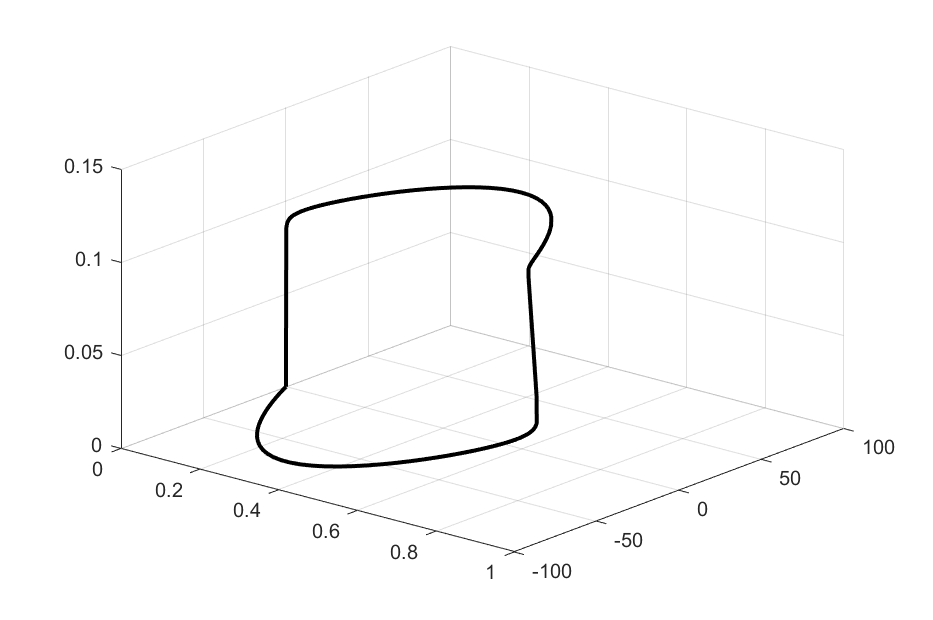}
        \put(50,65){\textbf{(a)}}
  	\put(72,3){$V$}
  	\put(23,5){$V_x$}
  	\put(2,35){$S$}
    \end{overpic}
    \end{minipage}
\hspace{3.5cm}
\begin{minipage}{.3\textwidth}
    \centering
    \begin{overpic}[scale=0.34]{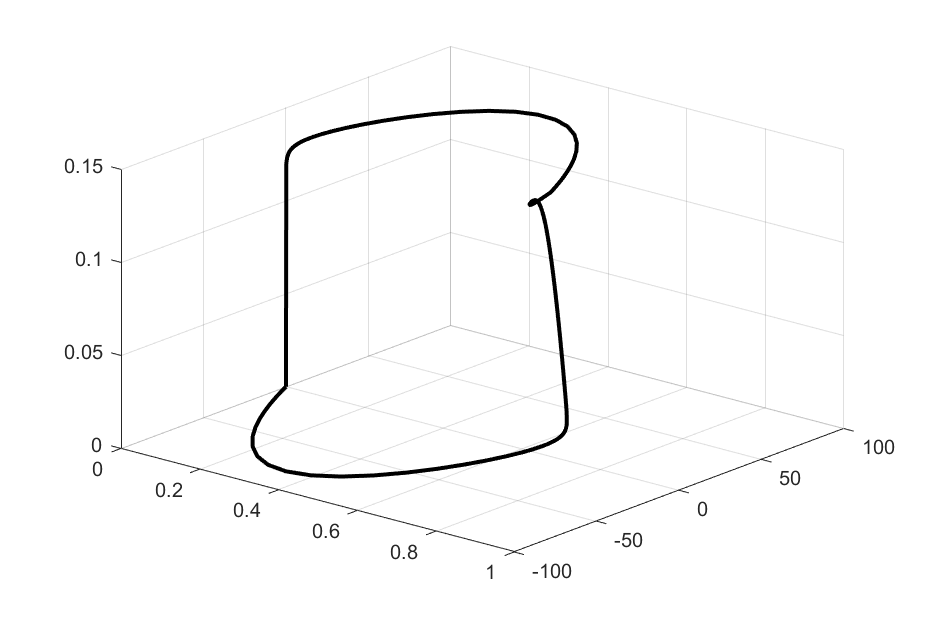}
        \put(50,65){\textbf{(b)}}
  	\put(72,3){$V$}
  	\put(23,5){$V_x$}
  	\put(2,35){$S$}
    \end{overpic}
    \end{minipage}
   \caption{Travelling wave solutions of System \eqref{eq:UVS_RDsystem_modified} in phase-space $(V, \, V_X, \, S)$ obtained by direct simulations with periodic boundary conditions on a 1D domain of length $L$ up to time $T$ using parameter values compatible with case (i). Panels (a) and (b) show a solution without and with superslow plateau, respectively. The parameter values used in the simulations are $\mathcal{A} = 1.5$, $\mathcal{B} = 0.2$, $\mathcal{H} = 0.1$, $\eps = 10^{-3}$, $T=5 \cdot 10^4$, and (a) $\mathcal{D} = 3160$, $k = 1.059$, $L=10$; (b) $\mathcal{D} = 2277$, $k = 0.955$, $L = 20$.}
   \label{fig:casei}
\end{figure}

\begin{figure}[!ht]
    \centering
    \begin{overpic}[scale=0.35]{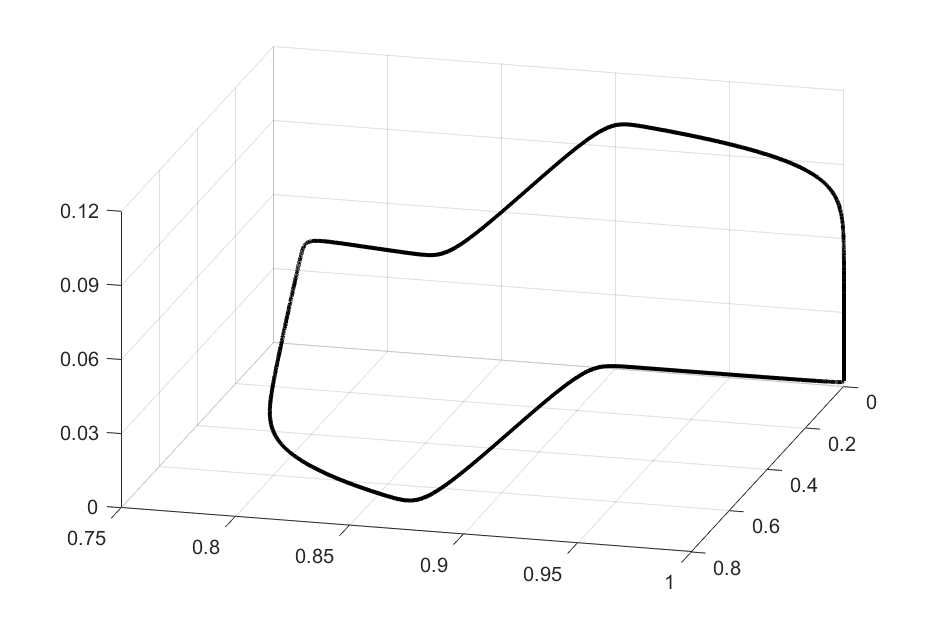}
  	\put(92,12){$V$}
  	\put(40,0){$U$}
  	\put(2,25){$S$}
    \end{overpic}
   \caption{Travelling wave solutions of System \eqref{eq:UVS_RDsystem_modified} in phase-space $(V, \, V_X, \, S)$ obtained by direct simulations with periodic boundary conditions on a 1D domain of length $L$ up to time $T$ using parameter values compatible with case (ii). The parameter values used in the simulations are $\mathcal{A} = 1.5$, $\mathcal{B} = 0.2$, $\mathcal{D} = 37492$, $\mathcal{H} = 0.1$, $\eps = 0.01$, $k = 0.955$, $L = 60$, and $T=5 \cdot 10^4$.}
   \label{fig:caseii}
\end{figure}

\begin{figure}[!ht]
    \begin{minipage}{.3\textwidth}
    \centering
    \begin{overpic}[scale=0.32]{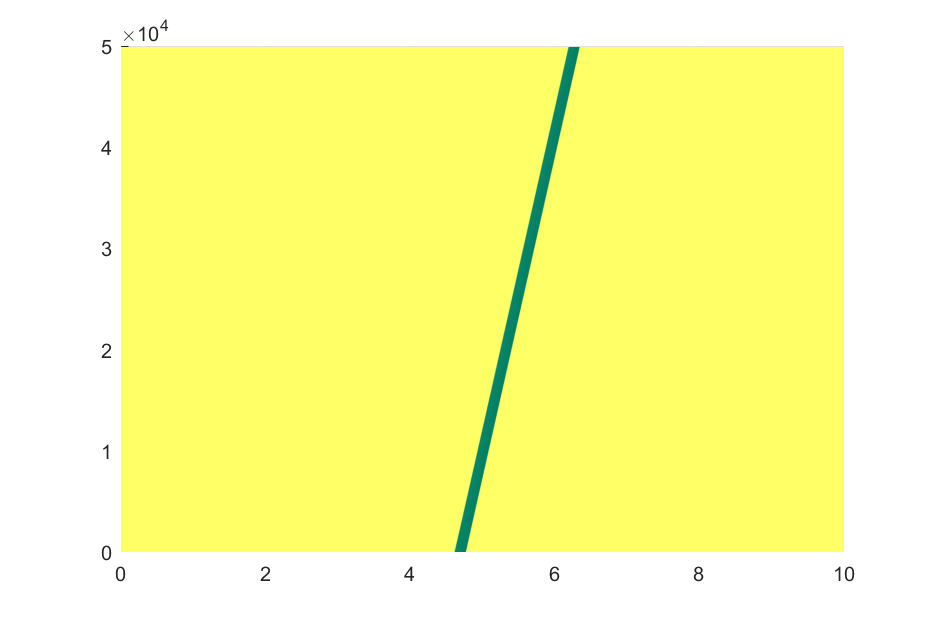}
        \put(50,65){\textbf{(a)}}
  	\put(50,-1){$x$}
  	\put(3,35){$t$}
    \end{overpic}
    \end{minipage}
    \hspace{3cm}
    \begin{minipage}{.3\textwidth}
    \centering
    \begin{overpic}[scale=0.32]{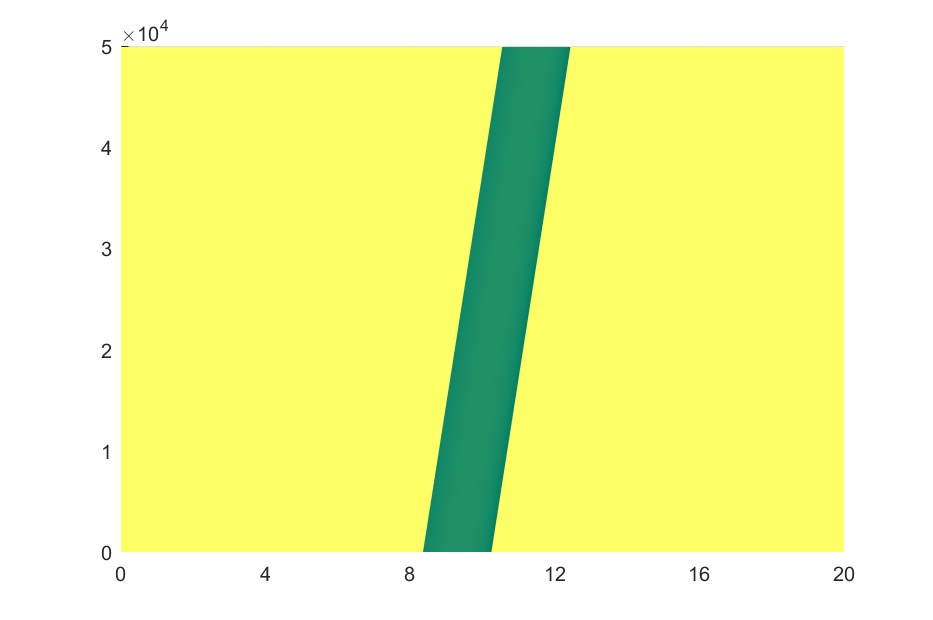}
        \put(50,65){\textbf{(b)}}
  	\put(50,-1){$x$}
  	\put(3,35){$t$}
    \end{overpic}
    \end{minipage}
   \caption{Space-time evolution of travelling wave solutions of System \eqref{eq:UVS_RDsystem_modified} obtained by direct simulations with periodic boundary conditions on a 1D domain of length $L$ up to time $T$ using parameter values compatible with case (i). Panels (a) and (b) show a solution without and with superslow plateau, respectively. The parameter values used in the simulations are $\mathcal{A} = 1.5$, $\mathcal{B} = 0.2$, $\mathcal{H} = 0.1$, $\eps = 10^{-3}$, $T=5 \cdot 10^4$, and (a) $\mathcal{D} = 3160$, $k = 1.059$, $L=10$; (b) $\mathcal{D} = 2277$, $k = 0.955$, $L = 20$.}
   \label{fig:timeev_casei}
\end{figure}

\begin{figure}[!ht]
    \centering
    \begin{overpic}[scale=0.32]{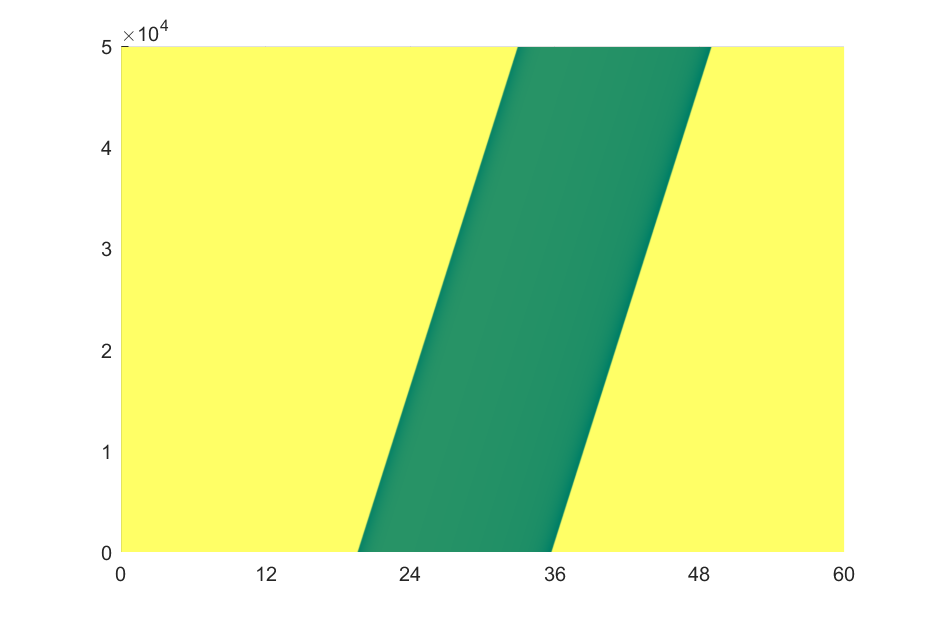}
  	\put(50,-1){$x$}
  	\put(3,35){$t$}
    \end{overpic}
   \caption{Space-time evolution of travelling wave solutions of System \eqref{eq:UVS_RDsystem_modified} obtained by direct simulations with periodic boundary conditions on a 1D domain of length $L$ up to time $T$ using parameter values compatible with case (ii). The parameter values used in the simulations are $\mathcal{A} = 1.5$, $\mathcal{B} = 0.2$, $\mathcal{D} = 37492$, $\mathcal{H} = 0.1$, $\eps = 0.01$, $k = 0.955$, $L = 60$, and $T=5 \cdot 10^4$.}
   \label{fig:timeev_caseii}
\end{figure}



\section{Discussion} \label{sec:concl}

In this paper, we considered the existence of traveling homoclinic pulse solutions to the Klausmeier model with autotoxicity~\eqref{eq:UVS_RDsystem_modified}. Based on the interplay between three timescales, determined by the relative scaling of the parameters $\eps,\delta$, related to the diffusion coefficient ratio and autotoxicity timescale, respectively, we constructed a variety of homoclinic stripe solutions for $\mathcal{O}(1)$ values of the system parameters $\mathcal{A},\mathcal{B}, \mathcal{H}, k$. In particular, the inclusion of the effects of autotoxicity allows for a much broader variety of traveling stripe solutions, with or without extended spatial plateaus, than in prior studies of the Klausmeier model without autotoxicity effects~\cite{BCD.2019}. 

While the results of the current work are focused on existence of traveling waves, the numerical simulations in~\S\ref{sec:numerics-dns} suggest that the traveling waves constructed in the scaling regimes (i) and (ii) are in fact stable as solutions of~\eqref{eq:UVS_RDsystem_modified}, at least in one spatial dimension. A natural question concerns the spectral stability of the waves constructed here: such an analysis would inherit the technical challenges present in the existence construction due to the three-timescale structure. Nevertheless, we expect that existing techniques, using exponential trichotomies and/or Evans function constructions~\cite{BastiaansenDoelman.2019, DoelmanVeerman.2015}, could similarly be extended to handle the geometry of the linearized system associated with~\eqref{eq:UVS_RDsystem_modified}, to track the small eigenvalues near the origin associated with the (multiple) fast fronts which appear in the singular homoclinic constructions. 

Our analysis focuses on the construction of localized homoclinic one-dimensional waves, though from the ecological point of view, spatial vegetation patterns are inherently two-dimensional and not necessarily homoclinic (i.e. `solitary'): they can form complex structures such as (interacting and deforming (multiple)) stripes and/or spots, and labyrinthine patterns~\cite{Deblauwe.2012, meron2019vegetation, valentin1999soil}. In one space dimension, the traveling waves constructed here form the basis for the construction of spatially periodic patterns or, more generally, for the reduction of the PDE to a finite dimensional ODE (on an approximate finite-dimensional manifold) describing the dynamics of multi-pulse patterns -- see \cite{BastiaansenDoelman.2019,SewaltDoelman.2017,promislow2002renormalization} and the references therein. The three spatial scales structure of the patterns considered here are expected to also add a deeper layer to the methods developed in the literature on spatially periodic and interacting pulses -- perhaps giving rise to unexpected challenges and/or phenomena.

In the planar setting, structures such as planar stripes or spot/gap/ring-type solutions can be built from the one-dimensional traveling waves studied here -- by either extending them in the direction transverse to propagation or by imposing radial symmetry. In the simpler Klausmeier model (without autotoxicity)~\cite{BastiaansenDoelman.2019, Klausmeier.1999}, 
\begin{subequations}\label{eq:pdeklaus}
\begin{align}
 \frac{\partial U}{\partial t} &= \Delta U + \mathcal{A} \left(1-U\right) - U V^2,\\
 \frac{\partial V}{\partial t} &= \eps^2 \Delta V + U V^2(1-k V) - \mathcal{B} V
\end{align}
\end{subequations}
and in other models with similar geometry~\cite{van2011planar}, the stability and dynamics of such singularly perturbed, far-from-onset, structures in two spatial dimensions has been considered. In~\eqref{eq:pdeklaus}, two-dimensional stripes, and large radial spot or ring solutions are expected to be unstable, though there is evidence that smaller spot solutions can be stable~\cite{byrnes2023large}, as can planar stripe solutions in the presence of advection (sloped terrain)~\cite{BastiaansenDoelman.2019,Carter.2023}. An area of future work concerns whether the consideration of autotoxicity in~\eqref{eq:UVS_RDsystem_modified} may have a stabilizing effect on the one-dimensional traveling waves constructed here, when extended into two dimensions as stripes or radially symmetric spot or ring-type solutions. Solutions constructed in the present study may help us better understanding from the mathematical viewpoint the different types of structures, such as fairy-rings, that emerge in ecological settings~\cite{Salvatori_2023}.

\paragraph{Acknowledgements} PC was partially supported by the National Science Foundation through the grants DMS-2105816 and DMS-2238127. The research of AD is supported by the European Research Council (ERC-Synergy project RESILIENCE, proposal nr. 101071417). AI is a member of Gruppo Nazionale per la Fisica Matematica (GNFM), Istituto Nazionale di Alta Matematica (INdAM) and acknowledges partial support from an FWF Hertha Firnberg Research Fellowship (T 1199-N).



\bibliography{travtoxpulses}
\bibliographystyle{test4}

\end{document}